\documentclass[12pt,pdftex,a4paper]{amsart}
\pdfoutput = 1

\usepackage[centering,text={15.5cm,23cm}]{geometry} 

\usepackage{amsfonts,amssymb, amscd, amsmath, latexsym, graphicx, psfrag, color,float,bbm}

\usepackage{graphicx,color,caption} 
\usepackage{comment}
\usepackage[shortlabels]{enumitem}
\usepackage{wrapfig}
\usepackage[all]{xy}
\usepackage{mathrsfs}
\usepackage{mathtools}
\usepackage{marvosym}
\usepackage{stmaryrd}
\usepackage{srcltx}
\usepackage{hyperref}
\usepackage{upgreek} 

\usepackage{tikz}
\usetikzlibrary{matrix,shapes,arrows,arrows.meta,calc,topaths,intersections,hobby,positioning,decorations.pathreplacing,decorations.pathmorphing,fit,patterns}
\tikzset{cross/.style={cross out, draw=black, fill=none, minimum
    size=2*(#1-\pgflinewidth), inner sep=0pt, outer sep=0pt},
  cross/.default={2pt}}

\tikzset{thickstyle/.style={shape=circle,fill=black,scale=0.3}}
\tikzset{thinstyle/.style={shape=circle,fill=black,scale=0.15}}

\usepackage{braket}

\definecolor{mygray}{gray}{0.75} 

\usepackage{extarrows}

\definecolor{shadecolor}{rgb}{1,0.9,0.7}

\newtheorem{thm}{Theorem}[section]
\newtheorem{lem}[thm]{Lemma}
\newtheorem{lemma-definition}[thm]{Lemma-Definition}
\newtheorem{pro}[thm]{Proposition}
\newtheorem{cor}[thm]{Corollary}

\theoremstyle{definition}
\newtheorem{setup}[thm]{Setup}
\newtheorem{dfn}[thm]{Definition}
\newtheorem{construction}[thm]{Construction}
\newtheorem{notation}[thm]{Notation}
\newtheorem{assumption}[thm]{Assumption}

\newtheorem{convention}[thm]{Convention}
\newtheorem{exa}[thm]{Example}

\theoremstyle{remark}
\newtheorem{rem}[thm]{Remark}

\newlength{\assumlabelwd}
\newlist{assumptions}{enumerate}{1}
\setlist[assumptions]{
  label=\textbf{\thethm.}, before=\let\orgitem\item \renewcommand{\item}{\stepcounter{thm}\orgitem},
  wide=0pt,                         
  labelsep=1em,                     
  topsep=1ex,                       
  itemsep=1.5ex plus .2ex,          
  parsep=0pt,                       
  partopsep=0pt                     
}

\numberwithin{equation}{section}
\numberwithin{figure}{section}

\newcommand{\Mod}[1]{\ (\mathrm{mod}\ #1)}

\newcommand{\NN} {\mathbb{N}}
\newcommand{\ZZ} {\mathbb{Z}}

\newcommand{\RR} {\mathbb{R}}

\newcommand{\PP} {\mathbb{P}}
\renewcommand{\AA} {\mathbb{A}}

\newcommand {\shC}  {\mathcal{C}}

\newcommand {\shDiv} {\mathcal{D} \!\text{\textit{iv}}}

\newcommand {\shExt} {\mathcal{E} \!\text{\textit{xt}}}
\newcommand {\shExtc} {\mathcal{E} \!\text{\textit{xt}}_c}

\newcommand {\shF}  {\mathcal{F}}

\newcommand {\shHom} {\mathcal{H}\!\text{\textit{om}}}
\newcommand {\shI}  {\mathcal{I}}

\newcommand {\shK}  {\mathcal{K}}
\newcommand {\shL}  {\mathcal{L}}
\newcommand{\shLS}{\mathcal{LS}}
\newcommand {\shM}  {\mathcal{M}}

\newcommand {\shN}  {\mathcal{N}}
\newcommand {\shO}  {\mathcal{O}}

\newcommand {\shR}  {\mathcal{R}}
\newcommand {\shS}  {\mathcal{S}}
\newcommand {\shT}  {\mathcal{T}}

\newcommand {\shP}  {\mathcal{P}}

\newcommand{\uM}{\underline{M}}

\newcommand{\uT}{\underline{T}}

\newcommand {\foM}  {\mathfrak{M}}

\newcommand {\foP}  {\mathfrak{P}}
\newcommand {\foQ}  {\mathfrak{Q}}

\newcommand {\foX}  {\mathfrak{X}}

\newcommand {\fom}  {\mathfrak{m}}

\newcommand{\cO}{\mathcal{O}}

\newcommand {\Bl}  {\operatorname{Bl}}

\newcommand {\chr} {\operatorname{char}}

\newcommand {\codim} {\operatorname{codim}}

\renewcommand{\div}  {\operatorname{div}}

\newcommand {\Div}  {\operatorname{Div}}

\newcommand {\Ext}  {\operatorname{Ext}}

\newcommand {\ggtc} {ggtc}

\newcommand {\gp}  {{\operatorname{gp}}}

\newcommand {\Hom}  {\operatorname{Hom}}
\newcommand {\twoHom}  {\operatorname{2-Hom}}

\newcommand {\id}  {\operatorname{id}}
\newcommand {\im}  {\operatorname{im}}

\newcommand {\LS}{\operatorname{LS}_{k^\dagger}}

\renewcommand{\O}  {\mathcal{O}}

\newcommand {\Ob}  {\operatorname{Ob}}

\newcommand {\one}{\mathbf{1}}   

\newcommand {\Pic} {\operatorname{Pic}}
\newcommand {\PPic}{\operatorname{\underline{Pic}}}
\newcommand {\shPPic}{{\underline{\shP ic}}}

\newcommand {\rk} {\operatorname{rk}}

\newcommand {\Sing} {\operatorname{Sing}}

\newcommand {\Spec} {\operatorname{Spec}}

\newcommand {\W} {\mathscr{W}}

\def\mydate{\ifcase\month \or January\or February\or March\or
April\or May\or June\or July\or August\or September\or October\or 
November\or December\fi \space\number\day,\space\number\year}

\makeatletter
\newcommand{\labeltext}[2]{%
  \@bsphack
  \csname phantomsection\endcsname 
  \def\@currentlabel{#1}{\label{#2}}%
  \@esphack
}
\makeatother

\makeatletter
\let\oldcite\cite
\def\@newcite[#1]#2{\oldcite{#2},\,#1}
\renewcommand{\cite}{\@ifnextchar[{\@newcite}{\oldcite}}
\makeatother

\begin{document}
\bibliographystyle{alpha}

\title[How to make log structures]
{How to make log structures}

\date{\today}

\author{Alessio Corti}
\address{Department of Mathematics\\ Imperial College London\\
  London SW7 2AZ, UK}
\email{a.corti@imperial.ac.uk}
\thanks{A.C.\  received support from EPSRC Programme Grant
  EP/N03189X/1.}
\author{Helge Ruddat}
\address{ Department of Mathematics and Physics\\ University of Stavanger\\ P.O. Box 8600 Forus\\ 4036 Stavanger\\ Norway}
\email{helge.ruddat@uis.no}
\thanks{H.R.\ received support from DFG grant RU 1629/4-1.}

\begin{abstract}
  We introduce the concept of a \emph{viable} \emph{generically Gorenstein
    toroidal crossing} (\ggtc{}) space $Y$. 
  This generalizes the concept of Gorenstein toroidal crossing
  scheme, which in turn generalizes that of a simple
  normal crossing scheme.
  
  On such a space $Y$, we define a sheaf $\shLS_Y$, intrinsic to $Y$,
  by means of an explicit construction. 
  Our main theorem establishes a bijection between 
  the set $\LS (Y)$ of isomorphism classes of log
  structures on $Y$ over the log point $\Spec k^\dagger$ that are compatible with the \ggtc{} 
  structure and the set $\Gamma(Y,\shLS_Y^\times)$ of nowhere vanishing global sections of $\shLS_Y$.

  The definition of $\shLS_Y$ by explicit construction permits the
  \emph{effective construction} of log structures on $Y$; it also
  enables \emph{logarithmic birational geometry}, in particular the
  construction --- in some cases --- of resolutions of singular log
  structures.
  
  Our work generalizes~\cite[Theorem~3.22]{MR2213573},
  adapting the original proof with techniques from the theory of
  \mbox{$2$-groups} and local line bundle systems.
\end{abstract}

\maketitle
\tableofcontents

\setcounter{tocdepth}{1}

\section{Introduction}
\label{sec:introduction}

\begin{assumption}
  Throughout this paper, we work over a perfect field
  $k$.\footnote{We assume that $k$ is perfect in order to use the
    Cohen structure theorem, see Setup~\ref{setup:ggtc_space}.} All
varieties, schemes and Deligne--Mumford stacks are of finite type over
$k$.
\end{assumption}

Informally speaking, a \emph{generically Gorenstein toroidal crossing
  (\ggtc{}) space} is a stratified Deligne--Mumford stack $Y$ over $k$ such that, at the
generic point of each stratum, the stabilizer is trivial and $Y$ is formally isomorphic to the
spectrum of a Stanley--Reisner ring. (See
Definition~\ref{dfn:ggtc_space} for a formal statement.)

Our concept of \ggtc{} space generalizes the concept of Gorenstein
toroidal crossing space~\cite{MR2218822}. For simplicity, the reader
may think of $Y$ as a scheme with simple normal crossing
singularities; however, our main result has a much simpler formulation
in this very special situation\footnote{For future applications, we
  need to include in our consideration a singular space like
  $(xy=0)$ in $\frac1r(1, -1, a, -a)$, which is conveniently viewed as
  a toroidal crossing Deligne--Mumford stack.}.

On such a space $Y$, we define a sheaf $\shLS_Y$, intrinsic to $Y$, by
means of an explicit construction. Our main result
Theorem~\ref{mainmaintheorem} constructs a bijection from 
  the set $\LS (Y)$ of isomorphism classes of log
  structures on $Y$ over the standard log point $k^\dagger$ compatible with the \ggtc{}
  structure to the
  set $\Gamma(Y,\shLS_Y^\times)$ of nowhere vanishing sections.

  Our construction generalizes the one given in~\cite{MR2213573} in
  the case when $Y$ is a union of toric varieties meeting in boundary
  strata.\footnote{More precisely, in the notation of
    \cite{MR2213573}, if $Y=X_0(B, \shP, s)$.}  Even in that case, our
  result improves upon~\cite{MR2213573}, because our construction of
  the sheaf $\shLS_Y$ is more natural and thus it has an independent
  geometric interpretation: see the discussion in
  \S~\ref{sec:comments-our-result} below.\footnote{For the technically
    savvy, a further advantage of our treatment is that the ``gluing
    data'' for the space $Y$ is not explicitly an input of the
    construction: our space $Y$ comes already glued.}

To us, the key point of our formulation is that it allows us to
construct log structures on $Y$ effectively.

\smallskip

This study is motivated by our program to construct smooth (or mildly singular) 
Fano and Calabi--Yau varieties. We aim to do so by smoothing a reducible toroidal crossing space 
equipped with a compatible log structure on a dense open set whose complement $  Z  $ is of codimension two, 
while carefully controlling the geometry of $  Z  $. 
In practice, we construct a section of $  \shLS_Y  $ whose zero locus is $  Z  $.

Chan--Leung--Ma set up a framework by which one can smoothen a
singular log scheme under a list of strong assumptions
\cite{MR4643802}. The list was then verified in Felten--Filip--Ruddat
\cite{MR4304077} for schemes with Gross--Siebert-type log
singularities. The framework was subsequently refined, generalized and
placed in the context of curved Gerstenhaber differential graded
$L_\infty$ algebras by Felten \cite{MR4359498,Felten}. Our examples of
interest in the context of smoothing Fano schemes, e.g.,
\S\ref{sec:mbox3-fold-transv}, are not of Gross--Siebert type but we
expect them to fall under the notion of \emph{unisingular
  deformations}, see \S14.3 in \cite{Felten}, and be therefore amenable
to the smoothing framework.

A second motivation is the desire to work with singular log
structures, and hence for a language that allows us to speak of, and
construct explicitly, log resolutions of log structures. Our results
indeed enable us to do all this, see for example \eqref{blow-up-LS},
\S~\ref{sec-blowup-threefold} and \cite{corti2025singular}. 
These examples hint at a theory of
\emph{log crepant log resolutions} of singular log structures: a
subject that we plan to pursue in the near future.

\subsection{Informal description of results}
\label{sec:inform-descr-results}

We describe our results informally. We begin by stating informally the
definition of generically Gorenstein toroidal crossing (\ggtc{}) space
$Y$ over a field $k$; we proceed to summarise the construction of the
sheaf $\shLS_Y$; and we conclude with a discussion of our main result,
exhibiting a bijection between the set $\LS (Y)$ of isomorphism
classes of compatible log structures on $Y$ over $\Spec k^\dagger$ and
$\Gamma (Y, \shLS_Y^\times)$.

\smallskip

Let $M\cong \ZZ^r$ be a lattice of rank $r$,
$\Sigma$ a rational polyhedral fan in $M$, and $K$ a field. The
\emph{Stanley--Reisner ring} $K[\Sigma]$ is the free $K$-vector space over
the monomials $z^m$, $m\in M\cap |\Sigma|$, where
\[
  z^m\cdot z^{m'}=\left\{ \begin{array}{ll} z^{m+m'}&\hbox{if there is a cone }\sigma\in\Sigma\hbox{ that contains }m,m',\\ 
                            0& \hbox{otherwise.}\end{array}\right.
\]

We consider stratified spaces
\[
Y=\coprod_{\eta\in T} Y_\eta^\star 
\]
over $k$ with locally closed strata indexed by a finite poset
$T$.\footnote{Where $\eta_1\leq \eta_2$ if and only if
  $Y_{\eta_1}^\star$ is contained in the Zariski closure of
  $Y_{\eta_2}^\star$. We spell out our terminology on stratified
  spaces in \S~\ref{sec:stratified_spaces}.} We denote by $Y_\eta$ the
Zariski closure of $Y_\eta^\star$. We always assume that $Y$ is
reduced and equidimensional, and that the irreducible components of
$Y$ are normal. We identify a point $\eta \in T$ with the generic
point of the corresponding stratum and we denote by $T^{[c]}\subset T$
the set of strata of codimension $c$. It follows that the irreducible
components of $Y$ are the closures of the strata of codimension $0$
and
\[
Y=\bigcup_{\sigma \in T^{[0]}} Y_\sigma \, .
\]

In short, a \emph{generically Gorenstein toroidal crossing (\ggtc{})} space is a
stratified space $Y$ that, at the generic point $\eta\in Y$ of every
stratum, is formally isomorphic to the
spectrum $\Spec k(\eta) [\Sigma_\eta]$ of the Stanley--Reisner
ring (over the residue field $k(\eta)$) of
a fan $\Sigma_\eta$ in a lattice $M_\eta$ of rank
\[
  \rk (M_\eta) = \codim (\eta)\,.
\]
These data are subject to compatibility conditions that are spelled
out in Definition~\ref{dfn:ggtc_space} below: the most important
requirement is that the lattices $M_\eta$ are the stalks of a sheaf (in the Zariski
topology) $\shM$ of abelian groups on $Y$, called the \emph{relative
  ghost sheaf} of the \ggtc{} space, and that  the fans $\Sigma_\eta$
in the lattices $M_\eta$ are the stalks of a sheaf of fans. 

In the paper we always assume that $Y$ is \emph{viable}: a technical
condition that can be ignored in this informal discussion and that is
stated in Definition~\ref{dfn:viability}.

\smallskip

Given a viable \ggtc{} space $Y$, we now summarise the construction of
the sheaf $\shLS_Y$. 

\begin{dfn}
  Let $Y$ be a \ggtc{} space. A \emph{slab} is a codimension one point
  $\rho\in T^{[1]}$; a \emph{joint} is a codimension two point $\omega\in T^{[2]}$.
\end{dfn}

We define, for all $\rho \in T^{[1]}$, a line bundle $\shL_\rho$ on
$Y_\rho$ that we call a \emph{slab bundle}.  Given $\rho \in T^{[1]}$,
there are exactly two distinct $\sigma,\, \sigma' \in T^{[0]}$ such
that $Y_\rho \subset Y_{\sigma} \cap Y_{\sigma'}$.  Fix $y\in Y_\rho$
and let $\eta$ be the generic point of the stratum containing $y$, so
$\eta\leq \rho$.  We identify $\sigma,\,\sigma'$ with the
corresponding maximal cones of the fan $\Sigma_\eta$, and $\rho$ with
the submaximal cone in $\Sigma_\eta$ that it corresponds to. Next
choose $v_y\in M_\eta$ at integral affine distance $1$ from
$\rho$.\footnote{We say that $v_y$ is \emph{at integral affine distance
  $1$ from $\rho$} if $\langle d, v_y\rangle=1$, where
  $d \in \sigma^\vee \subset \Hom (M_\eta, \ZZ)$ is the primitive
  vector that pairs to zero with all points in $\rho$.}
To simplify the discussion, assume that there is a local 
isomorphism\footnote{It turns out that this assumption holds in many
  cases of interest.} (and not just a formal isomorphism)
$f_\eta\colon k(\eta)[\Sigma_\eta]_{\mathfrak{m}}\to \cO_{Y,\eta}$ (where
$\mathfrak{m}\subset  k(\eta)[\Sigma_\eta]$ is the maximal ideal at the origin).
Our notion of viability for $Y$ implies that there exists
a Zariski open neighbourhood $V_y$ of $y$ in $Y$ such that the divisor
germ $\div\bigl(f_\eta(z^{v_y})\bigr)$ lifts to a Cartier divisor
$D_{y,\sigma}$ over $Y_{\sigma}\cap V_y$ and similarly for $z^{-v_y}$
on $Y_{\sigma'}\cap V_y$.  We define the slab bundle $\shL_\rho$ on
$V_y$ to be
$$(\shL_\rho)_{|V_y}:=\shO_{Y_\sigma\cap V_y}(D_{y,\sigma})_{|Y_\rho\cap V_y}\otimes \shO_{Y_{\sigma'}\cap V_y}(D_{y,\sigma'})_{|Y_\rho\cap V_y}.$$ 
The main result of \S~\ref{sec:slab-bundles} is Lemma~\ref{lem-no-choice} stating that these local definitions glue to a global line bundle $\shL_\rho$ on $Y_\rho$.
\footnote{The tensor product of the divisor of $z^{v_y}$ with the one of
$z^{-v_y}$ is reminiscent to the tensor product of the two normal
bundles $N_{Y_\sigma} Y_\rho$ and $N_{Y_{\sigma'}} Y_\rho$, but note
that in general neither $Y_\rho$, $Y_\sigma$ nor $Y_{\sigma'}$ can be
assumed smooth.}

The sheaf $\shLS_Y$ is defined as the subsheaf of the direct sum of all the
slab bundles:
\begin{equation}
\label{shLS-eq}
\shL\shS_Y \subset \oplus_{\rho\in T^{[1]}} \shL_\rho
\end{equation}
consisting of sections that, for every joint $\omega \in T^{[2]}$,
satisfy the \emph{joint condition} that we describe next.

For every joint $\omega \in T^{[2]}$, we can identify the slabs
incident at $\omega$ with the rays of the $2$-dimensional fan
$\Sigma_\omega$ in $M_\omega$.  Let $\rho_1, \dots, \rho_n$ be a
cyclic enumeration of the slabs incident at $\omega$, and let
$d_i\in N_\omega=\Hom (M_\omega, \ZZ)$ be the primitive normal to
$\rho_i$ such that $d_i>0$ on
$\rho_{i+1}$. Corollary~\ref{cor:joint-compat} in
\S~\ref{sec:joint-condition} states that, for every joint
$\omega \in T^{[2]}$, we have a well-defined isomorphism
\[
J_\omega \colon \bigotimes_{i=1}^n d_i\otimes \shL_{\rho_i|Y_\omega} \cong
0 \otimes \shO_{Y_\omega} \quad \text{in} \quad N_\omega \otimes
\Pic Y_\omega \,.
\]

In \S~\ref{sec:joint-condition}, Definition~\ref{def_shLS},
$\shLS_Y \subset \bigoplus_{\rho\in T^{[1]}} \shL_\rho$ is defined to
be the subsheaf consisting of sections $(f_\rho)_{\rho \in T^{[1]}}$
such that for all joints $\omega$
\[
  J_\omega\big(d_i\otimes (f_{\rho_i|Y_\omega})\big)= 0 \otimes 1
\]
at the generic point $\omega$ of $Y_\omega$.

\smallskip

We conclude with a statement of our main result. We begin
with some preliminaries that are needed before we can talk about our
notion of a compatible log structure on a \ggtc{} space. A fuller
discussion, including a short summary of basic facts on log structures
and a road-map of the proof, can be found in \S~\ref{sec:stat-main-result}.

Fix a viable \ggtc{} space $Y$. Recall that a \emph{log
  structure} on $Y$ is a pair $(\foP, \alpha)$ where $\foP$ is a sheaf
(in the Zariski topology) of monoids and
$\alpha \colon \foP \to (\shO_Y, \times)$ is a homomorphism of sheaves
of monoids such that
\[
  \alpha_{|\alpha^{-1}(\shO_Y^\times)}\colon
  \alpha^{-1}(\shO_Y^\times) \to \shO_Y^\times
\]
is an isomorphism. A log scheme is a scheme equipped with a log
structure and we denote by $Y^\dagger$ a log scheme with underlying
scheme $Y$. The \emph{ghost sheaf} is the quotient
sheaf of monoids ${\overline \foP}:=\foP/\alpha^{-1}(\shO_Y^\times)$.

Recall that the \emph{standard log point} is the log scheme
$\Spec k^\dagger=(\Spec k,\foP_k)$, where $\foP_k=k^\times\times\NN$
and
\[
\alpha (a,n)=
\begin{cases}
  \text{$0$ if $n>0$ and} \\
  \text{$a$ if $n=0$.}
\end{cases}
\]
A log scheme \emph{over the standard log point} --- or simply \emph{over}
$k^\dagger$ --- is a log scheme $Y^\dagger$ equipped with a morphism
$Y^\dagger \to \Spec k ^\dagger$ to the standard log point. We denote
a log scheme over $k^\dagger$ by the symbol $Y^\dagger/k^\dagger$. A
log scheme $Y^\dagger/k^\dagger$ comes with a global section
\[
  \one_{\foP}\in\Gamma(Y,\foP),
\]
the image of $1\in \NN$. With $\one_{\overline\foP}$ the image of $\one_{\foP}$ in $\overline\foP$, the \emph{relative ghost sheaf} is
the quotient sheaf $\overline\shM = \overline\foP/\one_{\overline\foP}$. 
In our context, $\overline\shM$ is going to be a sheaf of abelian groups.

Definition~\ref{dfn:compatible} below is a precise formulation of our
notion of a \emph{compatible log structure $Y^\dagger/k^\dagger$ on
  a \ggtc{} space $Y$}. The key requirement is the datum of an identification of the relative ghost sheaf
of the log structure with the relative ghost sheaf of the \ggtc{}
space, $\overline\shM\stackrel{\cong}{\longrightarrow}\shM$. We denote by $\LS (Y)$ the
set of isomorphism classes of compatible log structures over $k^\dagger$.


\subsection{Main theorem and its discussion}
\label{sec:comments-our-result}

The main theorem is about the sheaf $\shLS_Y$ in \eqref{shLS-eq}
constructed in detail in \S~\ref{sec:construction}.  This sheaf has a
subsheaf $\shLS_Y^\times$ of nowhere zero sections.

\begin{thm}[Theorem \ref{mainmaintheorem}]
Let $Y$ be a viable \ggtc{} space, and let
$\shLS_Y\subset\bigoplus_\rho \shL_\rho$ be the subsheaf of the direct
sum of slab line bundles constructed in
\S~\ref{sec:construction}.

Denote by $\LS(Y)$ the set of isomorphism classes of log structures
on $Y$ over $k^\dagger$ compatible with the \ggtc{} structure.

The set-theoretic function
\[
  r\colon \LS(Y) \to \Gamma (Y,\shLS_Y^\times)
\]
constructed in \eqref{morphism-r} is a bijection.
\end{thm}

We prove the theorem by considering the \emph{subsheaf of regular extensions}
\[
\shExtc^1(\shM, \shO_Y^\times) \subset \shExt^1(\shM, \shO_Y^\times)
\]
that already appeared in \cite[Theorem~3.22]{MR2213573}.\footnote{Here
$\shM$ and $\shO_Y^\times$ are sheaves of abelian group and
$\shExt^1(\shM, \shO_Y^\times)$ is the sheaf of extensions in the category of
sheaves of abelian groups.} We construct a morphism of sheaves
$\varphi\colon\shLS_Y^\times \to \shExtc^1(\shM, \shO_Y^\times)$ and
show that it is bijective. 
Finally, we prove that the assignment that sends a log
structure/$k^\dagger$ to its extension class in $\shExt^1(\shM,
\shO_Y^\times)$ in fact gives a bijection $\LS (Y)\to
\Gamma\left(Y, \shExtc^1(\shM, \shO_Y^\times)\right)$, and we obtain
$r$ as the composition of two bijections. 

The content of our main theorem is not that there is \emph{some} sheaf
$\shLS_Y$ with a natural identification of $\LS (Y)$ with
$\Gamma (Y, \shLS_Y^\times)$. Indeed, it is a basic general fact that
$\LS (-)$ is a sheaf. It is not even that there is \emph{some
  construction} of a sheaf $\shLS_Y^\times$ and a bijection from $\LS (Y)$ to $\Gamma (Y, \shLS_Y^\times)$. Indeed,
for example, $\shExtc^1(\shM, \shO_Y^\times)$ is an example of such a construction.

Our point is that the statement is true with the description of $\shLS_Y$ given in
\S~\ref{sec:construction}, and that this particular description allows to construct log structures
effectively. We illustrate
this point here with a very simple example. (More examples can be found in
\S~\ref{sec:examples}.)  Consider the case --- see also
\S~\ref{sec:two-components} --- of a scheme $Y$ that is the union of
two smooth components $Y_1$, $Y_2$ meeting transversally along a
smooth irreducible divisor $D$. We show in \S~\ref{sec:two-components} that
\begin{equation}
\label{LS-normal-bundle}
\shLS_Y  \cong  (N_{Y_1} D)\otimes (N_{Y_2}D) 
\end{equation}
where $N_{Y_i} D$ denotes the \emph{normal bundle} of $D$ in
$Y_i$.\footnote{The fact that $\LS (Y)$ is in bijective
  correspondence with $\Gamma\bigl( D , ( (N_{Y_1} D)\otimes
  (N_{Y_2}D) )^\times\bigr)$ is elementary and well-known. As important original references, we recommend taking a look at Definition 1.9 in~\cite{MR0707162} and Lemma 2.2 in~\cite{MR1754621}.
Further relevant works on the sheaf $(N_{Y_1} D)\otimes (N_{Y_2}D)$
and its cousin in the normal crossing case
include~\cite{MR1312571,MR1296351,MR2218822,MR4304077}; see also Theorem 1.2 in~\cite{MR1993863}.} The description \eqref{LS-normal-bundle} is particularly useful when the line bundle $\shLS_Y = (N_{Y_1}
D)\otimes (N_{Y_2}D)$ is not trivial, and thus it does not have a
nowhere-vanishing section. Consider for example the case
when $\shLS_Y$ is, say, base point free, and let $Z\subset D$ be the
vanishing locus of a general section $s\in \Gamma (D, \shLS_Y)$ giving an
isomorphism $\shO_D(Z)\cong \shLS_Y$. In
language introduced in Definition~\ref{dfn-gs}(1), we say that such an $s$ gives a log
structure/$k^\dagger$ \emph{singular along $Z$}. The log structure in
question is the push forward to
$Y$ of the log smooth log structure that we have on $Y\setminus Z$. This push
forward log structure is rather badly behaved, for example it is not
coherent. However, it has a particularly
nice \emph{log resolution}. Indeed, let
\[
\widetilde{Y}_1=\Bl_Z Y_1
  \]
  be the blow up of $Z\subset Y_1$. The strict transform of $D$ in $\widetilde{Y}_1$ is isomorphic to $D$, so
  we can glue $\widetilde{Y}_1$ to $Y_2$ along $D$ to form a scheme $f\colon \widetilde{Y}\to
  Y$. Denoting by $E=f^{-1}Z\subset \widetilde{Y}$ the exceptional
  set, we have $\widetilde{Y}\setminus E = Y\setminus
  Z$. It is clear that
  \begin{multline*}
   N_{\widetilde{Y}_1} D=(N_{Y_1} D)(-Z), \; \text{and hence} \\  \shLS_{\widetilde{Y}} =
    (N_{\widetilde{Y}_1}D)\otimes (N_{\widetilde{Y}_2}D) = \bigl((N_{Y_1}
    D)\otimes (N_{Y_2}D)\bigr)(-Z) = \shLS_Y (-Z)= \shO_D\,.
  \end{multline*}
 All of this goes to shows that there exists a unique log structure on $\widetilde{Y}$ smooth over $k^\dagger$ and a log morphism $\widetilde{Y}^\dagger\to Y^\dagger$
 over $k^\dagger$ that, when restricted to $\widetilde{Y}\setminus E=Y\setminus Z$, is an isomorphism and so it is the log structure
  given by the section $s\in \Gamma (Y, \shLS_Y)$. We call $\widetilde{Y}^\dagger\to Y^\dagger$
   a \emph{log resolution}.\footnote{A nice property of this
   resolution is that $K_{Y^\dagger/k^\dagger}$ is $f$-trivial: we
   call a resolution with this property \emph{log crepant}.}

 Our point, again, is that it would be awkward to establish these facts
 directly from the definition of log structure, and impossible to
 derive it off the shelf from the constructions and the statements
 in~\cite{MR2213573} (because that paper assumes that all components
 of $Y$ are toric varieties).
 It is the independent geometric interpretation of the sheaf
 $(N_{Y_1} D)\otimes (N_{Y_2}D)$ as the tensor product of normal bundles of
 $D$ in $Y_1, Y_2$ that makes the verification straightforward, by tracking the
 way that normal bundles change under blow ups. Our construction of the
 sheaf $\shLS_Y$, given in \S~\ref{sec:construction} below, in the
 two-component case immediately specialises to
 $(N_{Y_1} D)\otimes (N_{Y_2}D)$. In the general case of a \ggtc{} space
 $Y$, the construction is more involved, but it retains the geometric
 interpretation, making it possible, in many cases of interest, to
 construct log structures and log resolutions effectively.

 \smallskip

 It is well-known that, when $Y$ is simple normal crossings, the sheaf
 $\shLS_Y$ is naturally isomorphic to
 $\shT^1_Y=\shExt^1_{\shO_Y} (\Omega^1_Y,\shO_Y)$, see Theorem~5.5 and
 Remark~5.1 in~\cite{MR4304077} for further references, and so
 $\shLS_Y$ is in fact a coherent sheaf in this case. However, this
 rather special situation is somewhat misleading because the joint
 condition for gluing $\shLS_Y$ from line bundles
 happens to be linear when $Y$ is normal crossing (or a product of
 normal crossing spaces) while in general it is a polynomial condition
 that results in a non-coherent sheaf, see for example
 \eqref{joint-cond-del-Pezzo}.  The precise form of this polynomiality
 was already shown in \cite[Theorem~3.22]{MR2213573} and in fact
 everything we do reduces to this explicit local description when
 choosing a log smooth chart of a compatible log structure.

\subsection{Summary of previous work}
\label{sec:relat-prev-work}

We already indicated several prior works in the semistable situation, \cite{MR0707162,MR1754621,MR1312571,MR1296351,MR2218822,MR1993863}, so we now focus on singular log structures and more general spaces.
The paper~\cite{MR2213573} is concerned with toroidal crossing spaces
$Y$ that are a union of toric varieties meeting along boundary
strata. Among many other things, for such a $Y$, that paper defines a
sheaf $\shLS_Y$ and proves a natural identification
$\LS (Y)=\Gamma(Y, \shLS_Y^\times)$. Essentially,
$\shLS_Y$ is defined to be $\shExtc^1(\shM, \shO_Y^\times) \subset \shExt^1(\shM,
\shO_Y^\times)$ --- see \S~\ref{sec:comments-our-result} above and
\S~\ref{sec:stat-main-result} --- but the paper also gives an explicit
local description~\cite[Theorem~3.22]{MR2213573} in terms of local functions that satisfy the joint
condition, and then shows~\cite[Theorem~3.28]{MR2213573}
that these local functions are sections of explicit line bundles
$\shN_\rho$ (corresponding to our $\shL_\rho$). 
The description of the sheaf $\shLS_Y$ in \cite{{MR2213573}} is sufficiently concrete to enable the effective construction of elements in $\LS(Y)$ 
when $Y$ is a union of toric varieties meeting along boundary strata.

The paper~\cite{MR2218822} introduces the notion of \emph{Gorenstein
  toroidal crossing space}, and goes on to study log structures on
these.

\subsection{Our work in relation to previous work}
\label{sec:our-work-relation}

Our definition of generic Gorenstein toroidal crossing space is a
generalization of the Gorenstein toroidal crossing spaces
of~\cite{MR2218822}.
Our work is closely related to~\cite{MR2218822}, but there are two
important differences. The first key difference
is that we work with log structures over
$k^\dagger$, where~\cite{MR2218822} works with absolute log
structures: this change of perspective is essential to the
applications that we have in mind and it results in surprising simplifications. 
The second key difference is that we require the Gorenstein toroidal
crossing condition to hold only at the generic point of every stratum,
as opposed to everywhere.

Our paper generalizes the corresponding part of~\cite{MR2213573}, from
toroidal crossing spaces that are union of toric varieties meeting
along boundary strata, to the case of viable \ggtc{} spaces. Here the
key point of our study is a more natural and more general construction of
the sheaf $\shLS_Y$.

In outline, our proof follows the proof
of~\cite[Theorem~3.22]{MR2213573}, with changes necessary to
work with our construction of the sheaf $\shLS_Y$. Indeed, our main innovation
is the construction of the sheaf $\shLS_Y$, where we use the Picard
stack to show that the local descriptions of the slab bundles
$\shL_\rho$, when formulated not in terms of functions but of
divisors, \emph{glue automatically} to give the slab bundles globally,
and that the joint conditions automatically make sense
globally. Unlike~\cite{MR2213573}, we never work with local charts for
log structures and
our approach is closer to the Deligne--Faltings view of log structures
as systems of line bundles with sections.

More detail about where exactly and how specifically we depart
from~\cite{MR2213573} can be found in the outline of the proof in
\S~\ref{sec:stat-main-result}. In particular, as was pointed out
by Bernd Siebert, our point of view in the proof of
Proposition~\ref{pro:main_th_local_case} --- where we construct
a log structure from a section of $\shLS_Y$ --- is closely related to
that of~\cite{MR2964607}. 
We learned that the approach via line bundle systems is also useful for recasting logarithmic data in symplectic-geometric terms, \cite{farajzadeh2025stable}.

\subsection{Description of contents}
\label{sec:description-contents}

In \S~\ref{sec:gener-toro-cross} we introduce \emph{generically Gorenstein
  toroidal crossing spaces} --- \ggtc{} spaces for short.  The definition
is local in nature, and we take the time to describe two global
objects that are naturally attached to them, the \emph{cone sheaf} and
the \emph{divisor system}. We also introduce a property, which we call
\emph{viability}, that allows us to do log geometry on a \ggtc{} space.
We work in the Zariski topology for
simplicity and because it is sufficient for many
applications.\footnote{In short, all the examples we are currently interested in
  are unions $Y=\cup Y_\sigma$ of irreducible components $Y_\sigma$
  each of which admits a \emph{toric model}, i.e., a birational
  morphism to a toric variety.

  For certain applications, it will be necessary eventually to work
  with monodromy in the \'etale topology, and a future theory will
  have to implement these features. However, these features
  become a hindrance when working with examples that don't need them,
  which gives us another reason for not discussing these matters
  here.}

In \S~\ref{sec:construction} we explicitly construct a sheaf
$\shLS_Y$ that naturally exists on every viable \ggtc{} space $Y$. Later
in \S~\ref{section4} we prove that this sheaf $\shLS_Y$
classifies log structures on $Y$ over the standard log point and
compatible with the \ggtc{} structure. In this paper, we aim to address a
reader whose goal is to make a log structure on $Y$ explicitly. The
most efficient way to do this is to construct a nowhere-vanishing global
section of $\shLS_Y$, and for this she only needs to know how
$\shLS_Y$ is constructed; she does not need to know the proof that $\shLS_Y$
classifies log structures. Our presentation aims to facilitate
explicit constructions.

In the final \S~\ref{sec:examples} we give some examples.

\subsection{Acknowledgements}
\label{sec:acknowledgements}

We are grateful to Simon Felten and Andrea Petracci for many useful
discussions on the subject of this paper. We thank Kevin Buzzard, Simon Felten,
Martin Olsson, Bernd Siebert, Mattia Talpo, Anna-Maria Raukh and an anonymous referee
for comments on earlier versions of this article and helpful
suggestions at various stages of the long revision process.

Kevin had a strong influence on shaping our view of the subject; in
particular, he explained to us that the words ``canonical,''
``natural,'' ``distinguished,'' etc.\ are meaningless,
see~\cite{MR4866877}. As a result of our conversations with him, we
eliminated all occurrences of these words in the text, and instead we
made the effort of stating all the properties that our constructions
need to satisfy. Kevin also explained to us some key points about
coherence theorems in monoidal categories.

The authors were hosted by Mathematisches Forschungsinstitut
Oberwolfach as Oberwolfach Research Fellows in 2021 and 2022.
Substantial parts of the work on this project were produced during
these retreats.  Gratitude for hospitality also goes to the Department
of Mathematics at Imperial College London, the University of
Hamburg and the Department of Mathematics and Physics at the
University of Stavanger.

\section{Generically Gorenstein toroidal crossing spaces}
\label{sec:gener-toro-cross}

In \S~\ref{sec:defin-gener-gorenst} we give the formal definition of generic Gorenstein
toroidal crossing (\ggtc{}) space. Basically, a \ggtc{} space is a
stratified space $Y$ such that if $\eta \in Y$ is the generic point of
a stratum, there is a fan $\Sigma_\eta$ in a lattice $M_\eta$
and an isomorphism from  the formal completion
$\widehat{k(\eta)[\Sigma_\eta]}$ at the origin to the formal
completion $\widehat{\shO_{Y,\eta}}$ at the maximal ideal. These
isomorphisms need to satisfy certain coherence
conditions that are best kept track of by a Kato fan in the sense
of~\cite[\S~4]{MR3702314}.

In \S~\ref{sec:notat-conv-fans}--~\ref{sec:basic_setup} we set out carefully our
notation and conventions on monoids, fans and stratified spaces: this
material is elementary but tedious. 

In \S~\ref{sec:viable-ggtc-spaces} we define a property that we call viability, which allows
us to do logarithmic geometry.

The definition of a \ggtc{} space is local in nature. In
\S~\ref{sec:cone_sheaf} we define two global objects that exist
on the normalization of a \ggtc{} space --- the cone sheaf and the
divisor system --- that enter crucially the construction of the slab
bundles and the precise formulation of the joint condition. 

\subsection{Notation and conventions for monoids and fans}
\label{sec:notat-conv-fans}
Our terminology on monoids
mostly follows~\cite[Ch.~I]{MR3838359}. The following summary of
terminology and notation is intended for reference, not as a complete
dictionary on monoids.  

\begin{convention} \label{conv:monoids}
  \begin{enumerate}[(1)]
  \item A \emph{lattice} is a free abelian group $M\cong \ZZ^r$ of
  finite rank $r$. We denote $M_\RR:=M\otimes_\ZZ\RR$.
\item A \emph{monoid} is a commutative semigroup with neutral
  element. Our default position is to denote the operation and unit of
  a monoid additively by $+$,
  $0$. (If $R$ is a ring, for example,  we
  may want to consider the monoid $(R,\times, 1)$.)
\item If $M$ is a lattice and $S\subset M$ a subset, we denote by $\langle S \rangle$ the
saturation in $M$ of the subgroup generated by $S$ and by
$\langle S \rangle_+$ the saturation in $M$ of the submonoid generated
by $S$. Similarly, if $\tau \subset M_\RR$, we denote $\langle \tau \rangle:=\langle \tau\cap M \rangle$ and $\langle \tau \rangle_+:=\langle \tau\cap M \rangle_+$.
\item The group of units of a monoid $P$ is denoted by $P^\times$. A
  monoid $P$ is \emph{sharp} if $P^\times =(0)$.
\item A monoid $P$ is \emph{toric} if there exist a lattice $M$ of
  finite rank and a closed convex rational polyhedral cone
  $\sigma \subset M_\RR$ such that $P=\sigma \cap M$. In this paper we
  often work with toric monoids and we almost always assume them to be
  sharp.
  \item Let $P$ be a monoid. A submonoid $F\subset P$ is a \emph{face} if the following condition is satisfied.
  For all $u, v\in P$, if $u+v\in F$, then $u,v\in F$. We write $F\leq
  P$ to mean that $F$ is a face of $P$. The notation $F<P$ means that
  $F$ is a proper face of $P$, that is, $F$ is a face and $F\neq P$.   
\item When $F\leq P$ we denote by
  $F^{-1} P$ the \emph{localization} of $P$ at $F$. We call the monoid
  homomorphism $P\to F^{-1} P$ a \emph{face localization}. When $F\leq
  P$, the \emph{quotient} $P/F$ is the monoid $F^{-1} P/F$. We call the monoid
  homomorphism $P\to P/F$ a \emph{face quotient}.
\item If $P$ is a monoid and $\textbf{1}\in P$ an element, we use $P/\textbf{1}$ denote the quotient of $P$ by the submonoid generated by $\textbf{1}$. For example $\mathbb{N}^2/(1,1)\cong\mathbb{Z}$.
\item If $P$ is a monoid, we denote by $P^\gp$ the universal (Grothendieck)
  group of $P$.
\item 
  If $R$ is a ring and $P$ a monoid, we denote by $R[P]$ the monoid ring.
  \end{enumerate}
\end{convention}

\begin{dfn}
  \label{dfn:fan}
  Let $M\cong \ZZ^r$ be a lattice.
  \begin{enumerate}[(1)]
  \item A \emph{fan in $M$} is a finite set $\Sigma$ of closed convex
    rational polyhedral cones in $M_\RR$ such that:
\begin{enumerate}[(i)]
\item For all $\tau\in \Sigma$, if
  $\mu\leq \tau$ is a face, then $\mu\in\Sigma$;
\item For all $\tau, \mu \in \Sigma$, $\tau \cap \mu$ is a face of
  both $\tau$ and $\mu$.
\end{enumerate}

\noindent The \emph{support} of the fan, denoted by $|\Sigma|\subset M_\RR$, is
the union of the cones of $\Sigma$.

\noindent The fan is said to be \emph{complete} if $|\Sigma|=M_\RR$.
\item Let $\Sigma$ be a fan and $\rho\in \Sigma$ a cone.
  \begin{enumerate}[(a)]
  \item The \emph{localization} $\rho^{-1}\Sigma$ of $\Sigma$ in
    $\rho$ is the fan in $M$ that consists of the convex cones
    $\sigma-\rho=\{x-y\in M_\RR|x\in\sigma,y\in\rho\}$ where $\sigma$
    ranges over all cones in $\Sigma$ that contain $\rho$.
\item The \emph{quotient} $\Sigma/\rho$ of $\Sigma$ by $\rho$ is the fan in
$M/\langle\rho\rangle$ obtained from the localization
$\rho^{-1}\Sigma$ by projecting each of its cones under the linear map
$M\to M/\langle\rho\rangle$.
  \end{enumerate}
\item If $\Sigma$ is a fan in $M$ and $R$ is a ring, the
\emph{Stanley--Reisner ring} $R[\Sigma]$ is the free $R$-module over
the symbols $z^m$ for those $m\in M$ which are also contained in some
cone of $\Sigma$, with multiplication defined by
\[
  z^m\cdot z^{m'}=\left\{ \begin{array}{ll} z^{m+m'}&\hbox{if there is a cone }\sigma\in\Sigma\hbox{ that contains }m,m',\\ 
                            0& \hbox{else.}\end{array}\right.
\]                         
We denote by $\fom$ the ideal generated by all the symbols $z^{m}$ with $m\neq 0$. If
$R$ is a field, $\fom$ is a maximal ideal.
  \end{enumerate}
\end{dfn}

\begin{lem}
  \label{lem:monoidsANDfans}
  There is an equivalence between the following two
  categories:
  \begin{enumerate}[(1)]
  \item The category with objects pairs $(P, \mathbf{1})$ of a sharp
    toric monoid $P$ and an element
    $\mathbf{1}\in P$ such that $P \setminus (\mathbf{1}+P)$ is the union
    of the proper faces of $P$, and morphisms face quotients.
\item The category whose objects are pairs $(\Sigma, \varphi)$ of a
  rational polyhedral fan $\Sigma$, not necessarily complete but with convex support,
and a \emph{polarization}\label{polar}, that is, a strictly convex
piecewise linear function $\varphi \colon \vert \Sigma \vert\cap M \to
\ZZ$ up to the addition of an integral linear function $M\to\ZZ$, and morphisms fan quotients.
\end{enumerate}
Under this equivalence,
\[
P\setminus (\mathbf{1}+P)= \cup_{F<P} F
\]
if and only if the fan $\Sigma$ is complete.

Also, under this equivalence, $\ZZ[P]/(z^\textbf{1})=\ZZ[\Sigma]$.
\end{lem}


\begin{proof}[Sketch of Proof]
 Starting from $(\Sigma, \varphi)$, let $P$ be the supergraph of $\varphi$ in $M \oplus \ZZ$ and
$\textbf{1}=(0,1)$.
  
Viceversa, starting from $(P, \mathbf{1})$, let $M$ be the universal
group of the quotient monoid $P /\mathbf{1}$ of $P$ by the congruence
relation generated by the submonoid $\mathbf{1}\NN$, and let $\Sigma$
be the fan in $M$ whose cones are the projections under the obvious
homomorphism $P\to M$ of the proper faces of $P$ that do not contain
$\mathbf{1}$.
\end{proof}



\subsection{Notation and conventions for stratified spaces}
\label{sec:stratified_spaces}

The next two definitions correspond to the notion of finite partition
of~\cite[\href{https://stacks.math.columbia.edu/tag/09XZ}{Tag
  09XZ}]{stacks-project}
and finite good stratification of~\cite[\href{https://stacks.math.columbia.edu/tag/09Y0}{Tag
  09Y0}]{stacks-project}.

\begin{dfn}
  Let $X$ be a topological space. A \emph{partition of} $X$ is a
  decomposition
  \[
X=\coprod_{\eta\in T} X_\eta^\star 
\]
into locally closed subsets $X_\eta$ indexed by a finite set $T$. The
$X_\eta^\star$ are called the \emph{parts} of the partition.

We denote by $X_\eta=\overline{X_\eta^\star}$ the closure of $X_\eta^\star$.
\end{dfn}

\begin{dfn}
  Let $X$ be a topological space. A \emph{good stratification} of $X$
  is a partition $X=\coprod_{\eta\in T}X_\eta^\star$ such that for all $\mu, \eta
  \in T$ we have
  \[
X_\mu^\star \cap X_\eta \neq \emptyset \quad \Rightarrow \quad
X_\mu^\star \subset X_\eta \,.
\]
The $X_\eta^\star$ are called the \emph{strata} of the
stratification.

Given a good stratification $X=\coprod_{\eta \in T}X_\eta^\star$, we obtain a
partial ordering on the index set $T$ by setting $\mu \leq \eta$ if
and only if $X_\mu^\star \subset X_\eta$. It then follows that
\[
X_\eta = \coprod_{\mu \leq \eta}X_\mu^\star  \,.
\]
\end{dfn}

\begin{dfn}
  \label{dfn:stratified_space}
  \begin{enumerate}[(1)]
  \item A \emph{space} is a scheme or
  Deligne--Mumford stack of finite type over $k$.
  \item A \emph{stratified space} is a space $Y$ endowed with a good stratification
  \[
    Y=\coprod_{\eta\in T} Y_\eta^\star
  \]
  such that all strata are irreducible. Under this assumption, $T$ is
  identified with the set of generic points of the strata, and the
  partial ordering on $T$ is induced by specialization: for all
  $\eta_1,\eta_2\in T$, $\eta_1\leq \eta_2$ if and only if $\eta_1$ is
  a specialization of $\eta_2$. 

  When $Y$ is a Deligne--Mumford
  stack, we assume in addidion that the generic points of strata have
  trivial stabilizers.
\item The
  \emph{codimension} of $\eta \in T$ is the codimension in $Y$ of
  the corresponding stratum. We denote by $T^{[i]}\subset T$ the set
  of points of codimension $i$ and we write
  \[
    Y^{[i]}=\coprod_{\eta \in T^{[i]}} Y_\eta, \qquad
    Y^{(i)}=\bigcup_{\eta \in T^{[i]}} Y_\eta \,.
  \]
  \end{enumerate}
\end{dfn}

\begin{lem}
  \label{lem:retraction}
  Let $Y=\coprod_{\eta\in T} Y_\eta^\star$ be a stratified
  space.

  The subset topology of $T\subset Y$ is the order topology: $W\subset
  T$ is open if and only if for all $\tau_1 \in W$, if $\tau_2\geq
  \tau_1$ then $\tau_2\in W$. We have:
\begin{enumerate}[label=(\roman*)]
\item For all $\eta \in T$, 
  \[
  T_{\geq \eta} = \{\mu \in T \mid \mu \geq \eta \}
  \] 
  is the smallest open subset of $T$ that contains $\eta$; 
  
 \item The inclusion 
\begin{equation*}
  a\colon T \hookrightarrow Y \quad \text{has a continuous
  retraction} \quad b\colon Y\to T
\end{equation*}
defined such that  $b(y)=\eta $ if $y\in Y^\star_\eta$;
\item The map $b$ is open and for all Zariski open subset $U\subset Y$,
  $b(U)=a^{-1}(U)=U\cap T$.
\end{enumerate} 
\end{lem}

\begin{proof}[Sketch of Proof]
  The map $b$ is continuous: indeed for all $\eta \in T$ we have that
  \[
    b^{-1} (T_{\geq \eta})=\{y\in Y \mid b(y)\geq \eta\}=\coprod_{\mu
      \geq \eta} Y_\mu^\star
  \]
   is the union of all locally closed strata that have $\eta$ in their
   Zariski closure and hence it is Zariski open in $Y$.

   Consider a Zariski open subset $U\subset Y$. If $\eta \in b(U)$,
   then that means that $Y_\eta^\star \cap U \neq \emptyset$, or,
   equivalently, that $\eta \in U$. This shows that $b(U)=a^{-1}(U)$
   and in particular it is open.
\end{proof}

\begin{dfn}
  \label{dfn:open_star} Let $Y=\coprod_{\eta\in T} Y_\eta^\star$ be a stratified
  space. For all $\eta\in T$, the \emph{open star} of $\eta$ is
  the Zariski open subset
  \[
    U_\eta =\{y\in Y \mid b(y) \geq \eta\} = b^{-1}(T_{\geq \eta}) \subset Y,
 \]
that is, the union of all strata of $Y$ that have $\eta$ in their
Zariski closure. 
\end{dfn}

\begin{cor}
  \label{cor:sheavesonT} In the situation of
  Lemma~\ref{lem:retraction}, if $\shF$ is a sheaf on $T$ then the
  sheaves $a_\star \shF$ and $b^{-1} \shF$ on $Y$ are isomorphic. 
\end{cor}

\begin{rem}
  \label{rem:sheavesonT} In the situation of
  Lemma~\ref{lem:retraction}, we often use
  Corollary~\ref{cor:sheavesonT} to identify a sheaf $\shF$ on $T$
  with a sheaf on $Y$, and we think of it as $a_\star \shF$ or
  $b^{-1}\shF$ depending of which point of view is more convenient.
\end{rem}
  
\subsection{The basic setup and assumptions for toroidal crossing
  spaces} \label{sec:basic_setup}

In what follows $Y=\coprod_{\eta\in T} Y_\eta^\star$ is a stratified space 
  satisfying the following assumptions:
  \begin{enumerate}[(1)]
  \item $Y$ is reduced, equidimensional, and the irreducible
    components of $Y$ are normal.
  \item $Y$ is normal crossing in codimension~$1$: denoting by
\[
\varepsilon \colon Y^{[0]}=\coprod_{\sigma \in T^{[0]}} Y_\sigma \longrightarrow Y
\]
the normalization, the restriction of $\varepsilon$ to
to $\varepsilon^{-1}(Y^{(1)}\setminus Y^{(2)})$ is a degree-two
disconnected finite \'etale cover over each component of the target.
\item $Y$ is the push-out of the diagram of spaces
  $Y^{[1]} \rightrightarrows Y^{[0]}$ where the two maps are obtained
  from the inclusions $Y_\rho \to Y_{\sigma}$,
  $Y_\rho \to Y_{\sigma'}$.
  \end{enumerate}

\subsection{Generically Gorenstein toroidal
  crossing space }
\label{sec:defin-gener-gorenst}

\begin{setup}
  \label{setup:ggtc_space}
  We introduce objects and notation that are used  in
  Definition~\ref{dfn:ggtc_space} below.
  
  \smallskip
  
  Fix a finite poset $T$ equipped with the order topology. 
  \begin{enumerate}[(1)]
  \item   To give a sheaf of monoids $\shP$ on $T$ is equivalent to give the
  following data subject to obvious compatibilities:
\begin{enumerate}[(i)]
\item For all $\eta \in T$ a monoid $P_\eta$, and 
\item For all $\eta_1\le \eta_2$, a \emph{generization}
homomorphism
\[
  P_{\eta_1} \to P_{\eta_2}\,.
\]
\end{enumerate}
\item A \emph{Kato fan}~\cite[\S~4]{MR3702314} is a sharp
  monoidal space that is locally isomorphic to $\Spec P$, $P$ a sharp
  toric monoid. Let $\shP$ be a sheaf of monoids on $T$ making $T$ a Kato fan.
  In particular, all the $P_\eta$ are sharp
  toric monoids, and all the generization homomorphisms
  are face quotients. It follows from
  the definition that, for all $\eta \in T$, there is a bijective
  identification of the poset of faces of
  $P_\eta$ with $\{\tau \in T \mid \tau \geq \eta\}$. 
\item Given a Kato fan $(T,\shP)$, consider a global section $\mathbf{1}\in \Gamma(T, \shP)$
  corresponding to the datum, for all $\eta \in T$, of sections
  $\mathbf{1}_\eta \in P_\eta$ such that $P_\eta \setminus (
  \mathbf{1}_\eta + P_\eta)$ is the union of all the proper faces of
  $P_\eta$. In this situation, by Lemma~\ref{lem:monoidsANDfans}, the pair
  $(\shP,\mathbf{1})$ gives rise to a sheaf of complete fans
  $\mathbf{\Sigma}$ in the sheaf of lattices
  $\shM=\shP/\mathbf{1}$. This unpacks in the data, for all
  $\eta \in T$, of a fan $\Sigma_\eta$ in
  $M_\eta = P_\eta /\mathbf{1}_\eta$, and for all $\eta_1 \leq \eta_2$
  generization maps (viewing $\eta_2\in \Sigma_{\eta_1}$)
  quotient homomorphisms $M_{\eta_1}\to
  M_{\eta_2}=M_{\eta_1}/\langle \eta_2\rangle$ identifying
  $\Sigma_{\eta_2}$ with the quotient fan $\Sigma_{\eta_1}/\eta_2$.
For all $\eta \in T$, the poset of faces of
$P_\eta$ is identified with the fan $\Sigma_\eta$, and this induces an identification:
\[
  \Sigma_\eta =\{ \tau \in T \mid \tau \geq \eta\} \,.
  \]
  \end{enumerate}
\medskip
 Now fix a stratified space $Y=\coprod_{\eta\in T} Y_\eta^\star$ satisfying the assumptions of
  Section~\ref{sec:basic_setup}. Assume that $T$ is endowed with a
  pair $(\shP, \mathbf{1})$ of a sheaf of monoids $\shP$ making $T$ a
  Kato fan and global section $\mathbf{1}\in \Gamma (Y,\shP)$ as in~(3).
  \begin{enumerate}[(4)]
  \item Fix $\eta \in T$. The space $\Spec k(\eta)[\Sigma_\eta]$ has a
    natural stratification indexed by the cones of $\Sigma_\eta$
    ordered by inclusion. For a cone $\tau\in \Sigma_\eta$ we denote
    by $O_\tau \subset \Spec k(\eta)[\Sigma_\eta]$ the
    corresponding stratum. For all $\tau_1,\tau_2\in \Sigma_\eta$,
    $\tau_1\leq \tau_2$ if and only if
    $O_{\tau_1}\subset \overline{O}_{\tau_2}$. We denote by
    $ \widehat{k(\eta)[\Sigma_\eta]}$ the formal completion at the
    origin, and, for all $\tau\geq \eta$, we denote by
    $\widehat{O_{\tau}}\subset \Spec \widehat{k(\eta)[\Sigma_\eta]}$
    the induced subscheme.
\item[(5)] Fix $\eta \in T$. The local ring
  $(\shO_{Y,\eta},\mathfrak{m}_\eta)$ of $Y$ is a local Noetherian
  $k$-algebra ($k$ is the base field) with residue field
  $k(\eta)$. We denote by $\widehat{\shO_{Y,\eta}}$ the formal
  completion at $\mathfrak{m}_{\eta}$. By the Cohen structure
  theorem~\cite[\href{https://stacks.math.columbia.edu/tag/032A}{Tag
    032A}]{stacks-project} $\widehat{\shO_{Y,\eta}}$ contains a field
  isomorphic to $k(\eta)$. For all $\tau \geq \eta$, we denote by
  $\widehat{Y_\tau}\subset \Spec \widehat{\shO_{Y,\eta}}$ the
  induced subscheme.
  \end{enumerate}
\end{setup}

\begin{dfn}
  \label{dfn:ggtc_space} Importing Setup~\ref{setup:ggtc_space}, a
  \emph{generically Gorenstein toroidal crossing space} --- \ggtc{}
  space for short --- is a tuple:
  \[
    \bigl( Y=\coprod_{\eta\in T} Y_\eta^\star,(\shP,
    \mathbf{1}),\{\widehat{f_\eta}\mid \eta\in T\}\bigr)
  \]
  of a stratified space $Y=\coprod_{\eta\in T} Y_\eta^\star$ satisfying the assumptions of
  Section~\ref{sec:basic_setup}, and:
  \begin{enumerate} [label=(\alph*)]
 \item A pair $(\shP, \mathbf{1})$ of a Zariski sheaf $\shP$ of
   monoids on $T$ making $T$ a Kato fan, and a global section
   $\mathbf{1}\in \Gamma(T, \shP)$ given by elements
   $\mathbf{1}_\eta \in P_\eta$ such that
   $P_\eta \setminus (\mathbf{1}_\eta + P_\eta)$ is the union of all
   the proper faces of $P_\eta$. The section $\mathbf{1}$ induces a
   sheaf of complete fans $\mathbf{\Sigma}$ in the sheaf of lattices
   $\shM=\shP/\mathbf{1}$;
      \item For all $\eta \in T$, a ring isomorphism:
  \[
    \widehat{f_\eta}\colon
  \widehat{k(\eta)[\Sigma_\eta]}
  \overset{\cong}{\longrightarrow} \widehat{\shO_{Y, \eta}} \, ,
  \]
\end{enumerate}
subject to the following condition: For all $\tau \in \Sigma_\eta$,
the isomorphism $\widehat{f_\eta}$ identifies $\widehat{Y_\tau^\star}$
with $\widehat{O_\tau}$.

The sheaf $\shP$ is called the \emph{ghost sheaf} of the \ggtc{}
space; $\shM$ the \emph{relative ghost sheaf}; $\widehat{f_\eta}$
the \emph{local formal frames}.
\end{dfn}

\begin{convention}
  Sometimes for simplicity we write ``let $Y$ be a \ggtc{}
  space''. When we do this, it is understod that $Y$ is a stratified
  space, and that it has a ghost sheaf, a relative ghost sheaf, etc.,
  and we take it for granted that they will be denoted by $\shP$,
  $\shM$, etc.
\end{convention}

\begin{rem}
  \label{rem:ggtc}
  
  \begin{enumerate}[(1)]
  \item The reader should not be overly concerned about the formal local
  frames. We could replace~(b) with the following (possibly more familiar) requirement: for all $\eta
  \in T$, there is a Zariski closed subset $Z_\eta \subsetneq Y_\eta^\star$
  such that, for all $y\in Y_\eta \setminus Z_\eta$, there is a
  neighbourhood $y\in W \subset Y_\eta \setminus Z_\eta$ and a smooth
  morphism
  \[
W \to \Spec k[\Sigma_\eta] \,.
\]
 These morphisms would have to satisfy a more-or-less obvious
 coherence condition.
\item Let $Y$ be a \ggtc{} space. 
For all $\eta \in T$, $\Spec \ZZ[P_\eta]$ is a Gorenstein toric
variety with reduced boundary $B=\div z^{\mathbf{1}}$.
\item In~\cite{MR2218822} ``gtc'' is an acronym of ``Gorenstein toroidal
  crossing'' in the \'etale topology. On the one hand, in this paper, we work in the
  Zariski topology. On the other hand, we only require a Gorenstein toroidal
  crossing structure at the generic points of strata.
  \end{enumerate}
\end{rem}

\begin{exa}
  In many applications of interest,\footnote{For instance if the irreducible
    components of $Y$ are toric varieties, or more generally if they
  have toric models.} the formal frame isomorphisms
  $\widehat{f_\eta}\colon \widehat{k(\eta)[\Sigma_\eta]}
  \overset{\cong}{\longrightarrow} \widehat{\shO_{Y,\eta}}$ arise from
  bona fide frame isomorphisms
  \[
f_\eta\colon k(\eta)[\Sigma_\eta]_{\mathfrak{m}}
  \overset{\cong}{\longrightarrow} \shO_{Y,\eta} \,.
  \]
  In general, however, this is not possible. Consider for example a
  situation where $Y=Y_1\cup Y_2$ is a union of two smooth elliptic
  curves that meet in a point $P$ with residue field $k=k(P)$. The
  discrete valuation ring of $P$ in either component is not isomorphic
  to $k[x]_{(x)}$.
\end{exa}

\subsection{Viable \ggtc{}  spaces}
\label{sec:viable-ggtc-spaces}

In this section we define the key notion of viability of a \ggtc{} space. 

\begin{dfn}
  \label{dfn:viability}
 Fix a \ggtc{} space $\bigl( Y=\coprod_{\eta\in T} Y_\eta^\star,(\shP,
 \mathbf{1}),\{\widehat{f_\eta}\mid \eta\in T\}\bigr)$.

  For all $\eta \in T$, $\Spec k(\eta)[\Sigma_\eta]$ has a
   stratification with strata $\{O_\tau^\star \mid \tau \in
   \Sigma_\eta\}$, $O_\tau=\overline{O_\tau^\star}$. For $\tau\in \Sigma_\eta$, denote by
   \[
\Omega_\tau = \coprod_{\tau \leq \tau^\prime} O_{\tau^\prime}^\star
\]
the open star of $\tau$, and denote by $\overline{\Omega}_\tau$
its Zariski closure. Also, denote by
\[
  \Div^+_\flat \overline{\Omega}_\tau
\]
the monoid of effective Cartier divisors on $\overline{\Omega}_\tau$
supported on $\cup_{\rho \in T^{[1]}} O_\rho$.

For $\tau \in \Sigma_\eta$ and $m\in \tau\cap M_\eta$, denote by
$z^m\in k(\eta)[\Sigma_\eta]$ the corresponding monomial. For all
$\tau \in \Sigma_\eta$ we have a monoid homomorphism
\[
\div_{\eta,\tau} \colon \langle\tau\rangle_+ \ni m \mapsto \div z^m \in
\Div^+_\flat \overline{\Omega}_\tau \, .
\]

For $m\in \langle\tau\rangle_+$, write $\div_{\eta,\tau} (m) = \sum_{\rho \in T^{[1]}} m_\rho
O_\rho$ and consider the effective divisor:
\[
  D_{\eta,\tau} (m)=\sum_{\rho \in T^{[1]}} m_\rho Y_\rho
  \quad \text{on $\overline{U}_\tau\cap U_\eta$} \,.
\]

We say that the \ggtc{} space is \emph{viable} if, for all
$\eta \in T$, for all $\tau \in \Sigma_\eta$ and
$m\in \tau \cap M_\eta$ as above, the divisor $D_{\eta,\tau} (m)$
is a Cartier divisor.
\end{dfn}

\begin{notation}
  \label{nota:viability}
Let $Y$ be a viable \ggtc{} space. For all $\eta \in T$ and $\tau \in
\Sigma_\eta$, we denote by
\[
D_{\eta, \tau} \colon \langle \tau \rangle_+ \to 
\Div^+_\flat \bigl( \overline{U}_\tau\cap U_\eta \bigr)\, .
\]
the monoid homomorphism provided by Definition~\ref{dfn:viability}.

The monoid homomorphism $D_{\eta, \tau}$ extends to a group
homomorphism that, abusing notation, we still denote by
$D_{\eta, \tau} \colon \langle \tau \rangle \to 
\Div_\flat \bigl( \overline{U}_\tau\cap U_\eta \bigr)$.
\end{notation}

\begin{rem}
  \label{rem:viability} In the situation of
  Definition~\ref{dfn:viability}, for all \ggtc{} spaces, not
  necessarily viable, it can be shown that for all $\eta \in T$,
  $\tau\in \Sigma_\eta$ and $m\in \tau \cap M_\eta$, there is a
  Zariski neighbourhood
  \[
\eta \subset W \subset \overline{U}_\tau\cap U_\eta  
  \]
  such that $D_{\eta,\tau} (m)$ is Cartier on $W$.
\end{rem}

\begin{exa}
  \label{exa:key_example}
  Consider $r,a>0$ with $\gcd (r,a)=1$ and $\chr(k)\nmid r$.
  The natural \ggtc{} structure on the $3$-dimensional Deligne--Mumford stack
  \[
    Y=(uv=0)\subset [\AA^4/\mu_r]\, ,
  \]
  where $u,v,w,z$ are coordinates on $\AA^4$ and $\mu_r$ acts with
  weights $(1,-1,a,-a)$ is viable, as we are going to explain now.  By
  assumption, there is an integer $s$ so that $sa=r-1$ modulo
  $r$. This is one of those cases where the formal frames arise from
  bona fide frame isomomorphisms: an example of a frame at $\eta=(u,v)$
  is the isomorphism
  $$f_\eta\colon \left(k(\eta)[x,y]/(xy)\right)_{(x,y)}\to \shO_{Y,\eta}, \quad x\mapsto uw^s,\ y\mapsto vz^s.$$
  There are other possibilities for frame maps. We could post-compose this frame with any equivariant change of coordinates at $\eta$ that preserves strata, e.g., $u\mapsto uw^{i}z^{j}$, $v\mapsto vw^{k}z^{l}$ with 	$(i-j)$ and $(k-l)$ divisible by $r$, so there is a variety of choices.

  Most importantly for the \ggtc{} property, the divisor
  $\div(f_\eta(x))$ equals $\div (u)$ in $\Spec\shO_{Y,\eta}$ and
  therefore extends as a Cartier divisor in the component
  $(v=0)\subset Y$,
  supported on $(u=v=0)$, even at $y=(0,0,0,0)\in (v=0)$. Indeed,
  because we are working with the Deligne--Mumford stack $Y$, rather
  than its coarse moduli space, the divisors $ \div (u), \div (v)$ are
  well-defined Cartier divisors, also at the point $y=(0,0,0,0)\in Y$:
  even though the functions $u,v$ are not well-defined, each gives
  rise to a descent datum for a Cartier divisor from $\AA^4$ to $Y$.
  These divisors are not Cartier on the coarse moduli on the other
  hand.

This example turns up very frequently in applications to smoothing
toric Fano varieties, and it is precisely to allow for this example
that we need to be working with log structures on \ggtc{} Deligne--Mumford
stacks, as opposed to just \ggtc{} schemes.  

The coarse moduli scheme of $Y$ is a \ggtc{} scheme that is
not viable if $r>1$.
\end{exa}

\begin{exa}
    It is easy to construct examples of \ggtc{} spaces that are
    not toroidal crossing spaces. For example, assume $f,g,h$ are mutually coprime and glue 
    $$X_1=\left(y g(x_1,    \dots, x_n)+f(x_1, \dots x_n)=0\right)\subset \AA^{n+1},$$ 
    $$X_2= \left(z h(x_1, \dots, x_n)+f(x_1, \dots x_n) =0\right)\subset
    \AA^{n+1}$$ 
    along the common subvariety $\left( f(x_1,\dots,
    x_n)=0\right)\subset \AA^n$. If $f(x_1,\dots,x_n)$ is reduced, then
    the total space is \ggtc{}. For it to be toroidal crossing, $X_1$, $X_2$ and their respective subset given by $f=0$ must be smooth.
\end{exa}
\subsection{The divisor system}
\label{sec:cone_sheaf}



\begin{construction}
  \label{con:cone-sheaf}
  Fix a \ggtc{} space $\bigl( Y=\coprod_{\eta\in T} Y_\eta^\star,(\shP,
  \mathbf{1}),\{\widehat{f_\eta}\mid \eta\in T\}\bigr)$. As usual we
  denote by $\shM$ the relative ghost sheaf. Strictly speaking $\shM$
  is a sheaf on $T$, but we identify it with a sheaf on $Y$ as in 
  Remark~\ref{rem:sheavesonT}.

  Denote by
  \[
\varepsilon \colon X=Y^{[0]} \to Y
\]
 the normalization. We construct a sheaf of monoids on $X$, which we
 call the cone sheaf.

 The space $X$ is naturally stratified by strata indexed by the finite
 poset
 \[
S=\{(\eta, \sigma) \mid \eta \in T,\; \sigma \in T^{[0]},\; \eta \leq
\sigma \}
\]
 where $(\eta_1, \sigma_1)\leq (\eta_2,\sigma_2)$ if
 $\sigma_1=\sigma_2$ and $\eta_1\leq \eta_2$. The stratum
 corresponding to $\xi=(\eta,\sigma)\in S$ is
 \[
   X^\star_\xi \cong Y^\star_\eta \,.
   \]

 We construct a sheaf of monoids $\shC$ on
 $S$ and we identify it with a sheaf of monoids on $X$ as in
 Remark~\ref{rem:sheavesonT}.

 For all $\xi \in S$, we need to assign a monoid $C_\xi$ and for all
 $\xi_1\leq \xi_2$ we need to assign generization homomorphisms $C_{\xi_1}\to
 C_{\xi_2}$.

 For $\xi=(\eta, \sigma)\in S$, identify $\sigma$ with a maximal cone 
 of the fan $\Sigma_\eta$ and set

 \[
C_{\xi}  = \sigma \cap M_\eta \,.
   \]
   If $\xi_1=(\eta_1,\sigma_1)\leq \xi_2=(\eta_2,\sigma_2 )$, then
   $\eta_1\leq \eta_2$ so $\Sigma_{\eta_2}=\Sigma_{\eta_1}/\eta_2$
   and the generization morphism
   \[
M_{\eta_1}\to M_{\eta_2} = M_{\eta_1}/\langle\eta_2\rangle
     \]
     maps  $C_{\xi_1}$ to $C_{\xi_2}$.

     It is clear that these assignments define a sheaf of monoids $\shC$ on $S$.    
\end{construction}

\begin{dfn}
  \label{dfn:cone-sheaf}
  Let $Y$ be a viable \ggtc{} space. The \emph{cone sheaf} is the
  subsheaf
  of monoids
  \[
    \shC \subset \varepsilon^\star \shM
  \]
  of Construction~\ref{con:cone-sheaf}.
 \end{dfn}

\begin{construction}
  \label{con:divisor_system}
  Let $Y$ be a \ggtc{} space, $\varepsilon \colon X \to Y$ the
  normalization. The space $X=\coprod_{\xi\in S} X_\xi^\star$ is stratified as
  explained in Construction~\ref{con:cone-sheaf}.  The boundary of $X$ is
  the union of the strata of codimension $\geq 1$, and we denote by
  $\shDiv^+_\flat$ the sheaf in the Zariski topology of effective
  Cartier divisors on $X$ supported on the boundary.

  Let $\shC$ be the cone sheaf on $X$. Assuming
  that $Y$ is viable, we construct a homomorphism of sheaves of
  monoids on $X$:
  \[
\widetilde{D}\colon \shC \to \shDiv^+_\flat \,,
\]
which we call the divisor system.

For $\sigma \in T^{[0]}$, denote by
$\varepsilon_\sigma \colon Y_\sigma \hookrightarrow Y$ the natural
closed immersion, so that 
  $\varepsilon=\sqcup_{\sigma \in T^{[0]}} \,\varepsilon_\sigma \colon
  X \to Y$.  

We will use the notation set out in
Construction~\ref{con:cone-sheaf}.

Denote by $\widetilde{a} \colon S \to X$ the natural inclusion and by
 $\widetilde{b}\colon X \to S$ the retraction of
 Lemma~\ref{lem:retraction}: we have that $a(\eta, \sigma)=a(\eta) \in
 Y_\sigma$, and for $x\in Y_\sigma$, $\widetilde{b} (x)=(b(x), \sigma)$.

More precisely, $\shC$ is the sheaf
on $S$ of Construction~\ref{con:cone-sheaf} and we construct a sheaf
homomorphism $\widetilde{b}^{-1} \shC \to \shDiv^+_\flat$, which is
the same as a homomorphism $\shC \to \widetilde{b}_\star
\shDiv^+_\flat$ of sheaves on $S$.

For all $\xi = (\eta, \sigma)\in S$, let
\[
\widetilde{U}_\xi = \widetilde{b}^{-1} (S_{\geq \xi}) = \{x\in Y_\sigma \mid
b(x)\geq \eta\} \,.
  \]
To construct our sheaf homomorphism, for all $\xi \in S$ we need to
assign a monoid homomorphism:
\[
\widetilde{D}_\xi \colon C_\xi \to \Div^+_\flat (\widetilde{U}_\xi)
\]
and these monoid homomorphisms need to be compatible with generization
in the obvious way.

In Definition~\ref{dfn:viability}, for $m\in C_\xi=\sigma \cap
M_\eta$ we defined a Cartier divisor
\[
D_{\eta,\sigma} (m) \in \Div^+_\flat (\overline{U_\sigma} \cap U_\eta) 
  \]
Noting that $\widetilde{U}_\xi=\varepsilon_\sigma^{-1}(\overline{U_\sigma} \cap
U_\eta)$, we define $\widetilde{D}_\xi (m)=\varepsilon_\sigma^\star
D_{\eta, \sigma} (m)$.  
\end{construction}

\begin{dfn}
  \label{dfn:divisor_system}
  Let $Y$ be a viable \ggtc{} space. The \emph{divisor system} is the
  homomorphism of Zariski sheaves of monoids on $X=Y^{[0]}$
  \[
    \widetilde{D} \colon \shC \to \shDiv^+_\flat
    \quad \text{on}\; X
  \]
  of Construction~\ref{con:divisor_system}.
\end{dfn}

\begin{rem} 
One can show that $\widetilde{D}$ is an isomorphism.
\end{rem}

\begin{notation}
  \label{nota:open-stars}
  Let $Y=\coprod_{\eta \in T} Y^\star$ be a ggtc{} space, $a\colon
  T\to Y$ the inclusion and $b\colon Y \to T$ the retraction of
  Lemma~\ref{lem:retraction}. For all $y\in Y$, the open star of $y$
  is 
  \[
    U_y=U_\eta,\quad \text{where} \quad y\in Y_\eta^\star
  \]
and $U_\eta = b^{-1} T_{\geq \eta}$ is the open star of $\eta$, cf.\
Definition~\ref{dfn:open_star}.

 Similarly let $\varepsilon \colon X=\coprod_{\xi \in S} X^\star_\xi
 \to Y$ be the normalization, stratified as in
 Construction~\ref{con:cone-sheaf}, $\widetilde{a}\colon S \to X$ the inclusion and
 $\widetilde{b}\colon X\to S$ the retraction of
 Lemma~\ref{lem:retraction}.
 For all $x\in X$, the open star of $x$ is
   \[
    \widetilde{U}_x=\widetilde{U}_\xi,\quad \text{where} \quad x\in X_\xi^\star
  \]
and $\widetilde{U}_\xi = \widetilde{b}^{-1} (S_{\geq \xi})$ is the open star of $\xi$, cf.\
Definition~\ref{dfn:open_star}.
\end{notation}

\begin{lem} 
  \label{lem-commute-res-V}
  Let $Y$ be a viable \ggtc{} space, $\varepsilon \colon X \to Y$ the
  normalization stratified as in Construction~\ref{con:cone-sheaf}. For all $y\in Y$, consider the
  fan $\Sigma_y\subset M_y$, and let $\sigma_1, \sigma_2\in \Sigma_y$ be two
  maximal cones adjacent along a submaximal cone
  $\rho = \sigma_1 \cap \sigma_2$.  Then $y\in Y_\rho$ and let
  $x_1\in Y_{\sigma_1}\subset X$, $x_2 \in Y_{\sigma_2}\subset X$ be the two
  lifts, so $C_{x_1}=\sigma_1$, $C_{x_2}=\sigma_2$ and in this sense $C_{x_1}\cap C_{x_2}=\rho$. There are obvious inclusions
\[
  \iota_1\colon Y_\rho \hookrightarrow Y_{\sigma_1} \quad \text{and}\quad
  \iota_2 \colon Y_{\rho}\hookrightarrow Y_{\sigma_2}
  \]
  and we write (following Notation~\ref{nota:open-stars})
  \[
    V_y=U_y\cap Y_\rho \,.
  \]
  Note that $V_y=\iota_1^{-1}(\widetilde{U}_{x_1})=\iota_2^{-1}(\widetilde{U}_{x_2})$. 

In this situation, denote by
$\iota_1^!\colon \Div^+_\flat \widetilde{U}_{x_1}\dashrightarrow \Div^+_\flat V_y$ the
partially defined restriction homomorphism that is defined for each
divisor that intersects $V_y$ properly, and similarly $\iota_2^!$.

The following diagram is commutative:
\[
  \xymatrix{
C_{x_1} \ar[rr]^{\widetilde{D}_{x_1}} & & \Div^+_\flat \widetilde{U}_{x_1} \ar^{\iota_1^!}@{-->}[d]\\ 
C_{x_1}\cap C_{x_2}\ar@{^{(}->}[u] \ar@{_{(}->}[d]& &     \Div^+_\flat V_y\\
C_{x_2} \ar[rr]^{\widetilde{D}_{x_2}} & & \Div^+_\flat
\widetilde{U}_{x_2}\ar_{\iota_2^!}@{-->}[u]} \,,
 \]
where we note that the restriction $\iota_1^!$ is
well-defined on $\widetilde{D}_{x_1}(C_{x_1}\cap C_{x_2})$ and, similarly,
$\iota_2^!$ is well-defined on $\widetilde{D}_{x_2} (C_{x_1}\cap C_{x_2})$).

The diagram is compatible with generization in the obvious way.
\end{lem} 

\begin{proof} Straightforward. \end{proof}

\section{Construction of the sheaf \protect$\shLS_Y$}
\label{sec:construction}
In this section, given a viable \ggtc{} space $Y$, we construct a
sheaf of sets $\shL\shS_Y$ on $Y$, intrinsic to $Y$. In
\S~\ref{section4}, we will prove that $\shLS_Y$ classifies compatible
log structures on $Y$ over $k^\dagger$. We refer to
\S~\ref{sec:inform-descr-results} for an informal summary of the
construction of $\shLS_Y$.

In \S~\ref{sec:slab-bundles}, for every slab $\rho \in T^{[1]}$, we
give the construction of the slab line bundle $\shL_\rho$ on $Y_\rho$.
In Corollary~\ref{cor:joint-compat} of \S~\ref{sec:joint-condition},
we show that, for every joint $\omega\in T^{[2]}$, the restrictions
$\shL_{\rho|Y_\omega}$ for all $\omega <\rho$ satisfy the joint
 condition. This is seen as an easy consequence of an
abstract joint condition that is stated in
Lemma~\ref{lem:abstract_joint_conditon}. In Definition~\ref{def_shLS},
the joint conditions are used to define the sheaf
$\shLS_Y$.

The construction of the slab line bundles $\shL_\rho$ and, especially,
the formulation of the joint condition, are delicate. In order to glue
the slab bundles from local data, and in order to formulate the joint
condition, we need carefully to keep track of isomorphisms between
line bundles and not just the line bundles themselves. For this reason
it is necessary that we work with the Picard \mbox{$2$-group}
$\PPic Y$ (and, implicitly, the Picard stack $\shPPic_Y$), for which
we refer the reader to~\cite[\S1.4]{MR0354654}.

\subsection{The slab line bundles}
\label{sec:slab-bundles}

In this subsection we fix throughout a viable \ggtc{} space $Y$. The
goal is to construct, for all $\rho \in T^{[1]}$, a line bundle
$\shL_\rho \in \PPic Y_\rho$ that we call a slab bundle.

Before describing the construction, we recall the following.

\begin{dfn}
  \label{dfn:2-group_homo} Let $M$ be an abelian group and
  $(\underline{A},\otimes,\mathbf{1})$ a strictly commutative \mbox{$2$-group}, where
  \begin{enumerate}[(a)]
  \item For objects $A_1$, $A_2$ of $\underline{A}$, we denote by
    \[
      t\colon A_1\otimes A_2 \overset{\cong}{\longrightarrow}
      A_2\otimes A_1
    \]
the isomorphism provided by the \mbox{$2$-group} structure; 
  \item For objects $A_1$, $A_2$, $A_3$ of $\underline{A}$, we denote
    by
    \[
      s\colon A_1\otimes (A_2\otimes A_3)
      \overset{\cong}{\longrightarrow} (A_1\otimes A_2) \otimes A_3
    \]
the isomorphism provided by the \mbox{$2$-group} structure.
  \end{enumerate}
  A \emph{\mbox{$2$-homomorphism}} $\mu\colon M \to \underline{A}$ is an assignment
  $\mu \colon M \to \Ob \underline{A} $ together with the datum, for all
  $m_1,m_2\in M$, of an isomorphism
    \[
      \mu(m_1+m_2)\overset{\cong}{\longrightarrow} \mu(m_1)\otimes
      \mu(m_2)
    \]
    such that
  \begin{enumerate}[(i)]
  \item $\mu(0)=\mathbf{1}$;
  \item for all $m_1,m_2\in M$ the following diagrams are commutative
    \[
      \xymatrix{\mu(m_1+m_2) \ar[r]\ar@{=}[d]& \mu(m_1)\otimes \mu(m_2)\ar[d]^t\\
\mu(m_2+m_1)\ar[r] & \mu(m_2)\otimes \mu(m_1)}
      \]
  \item for all $m_1,m_2,m_3\in M$ the following diagrams are
    commutative
    \[
      \xymatrix{\mu\left(m_1+(m_2+m_3)\right)\ar[r]\ar@{=}[d]&\mu(m_1)\otimes
        \mu(m_2+m_3) \ar[r]&\mu(m_1)\otimes \left(\mu(m_2)\otimes
          \mu(m_3)\right) \ar[d]^s\\
        \mu\left((m_1+m_2)+m_3\right)\ar[r]&\mu(m_1+m_2)\otimes
        \mu(m_3) \ar[r]&\left(\mu(m_1)\otimes \mu(m_2)\right)\otimes
          \mu(m_3)}
      \]
  \end{enumerate}
\end{dfn}

\begin{lem}
  \label{lem-no-choice}
  Let $Y$ be a viable \ggtc{} space and $\rho\in T^{[1]}$ a slab. Fix $y\in
  Y_\rho$, and use the notation of Lemma~\ref{lem-commute-res-V}:
  in particular, denote by $\sigma_1,\sigma_2$ the maximal cones of
  $\Sigma_y \subset M_y$ that meet along $\rho$. Denote by
  \[
\mu_i\colon M_y \to \PPic (V_y)
\]
the unique \mbox{$2$-homomorphism} such that, for $m\in C_{x_i}\subset M_y$,
$\mu_i(m)=\iota_i^\star
\bigl(\shO_{\widetilde{U}_{x_i}}\bigl(\widetilde{D}_{x_i}(m)\bigr)\bigr)$. Note
that $\mu_1=\mu_2$ on the subspace
$\langle C_{x_1}\cap C_{x_2} \rangle=\langle \rho\rangle$ and
denote this restricted \mbox{$2$-homomorphism} simply by $\mu$.

Let $d \in \sigma_2^\vee \subset \Hom (M_y, \ZZ)$ be the unique primitive vector that pairs to zero with all points in $\rho$. Next choose
$v \in M_y$ such that $\langle d, v\rangle=1$. 
We define a line bundle $\lambda (v)$ on $V_y$ as
\[
  \lambda (v)=\mu_1(-v)\otimes \mu_2(v)\,.
\]
Then:
\begin{enumerate}[(1)]
\item The line bundle $\lambda (v)$ is independent on the choice of $v$ in
the following sense: for all $r\in \langle \rho\rangle$, we construct an isomorphism
\[
\psi_r\colon \lambda (v)\overset{\cong}{\longrightarrow}\lambda (r+v)\,;
\]
\item The set of these isomorphisms has the \emph{cocycle property}:
\begin{enumerate}[(i)]
\item For all $v \in \sigma_2$ such that $\langle d, v\rangle=1$,
  $\psi_{0}=\id_{\lambda(v)}$;
\item For all $v \in \sigma_2$ such that $\langle d, v\rangle=1$, and
  for all $r,s\in \langle \rho\rangle$, the following diagram is
  commutative:
  \[
    \xymatrix{
    \lambda(v)\ar[rr]^{\psi_{r+s}}\ar[dr]_{\psi_s} & & \lambda(r+s+v) \\
    &\lambda(s+v) \ar[ur]_{\psi_r} & }
  \]
\end{enumerate}
\item Similarly, $\lambda (v)$ does not depend on the choice of numbering of
$\sigma_1$, $\sigma_2$ in the sense that if we swap numbering, then
$d$ is changed into $-d$, $v$ to $-v$, and the line bundle into
$\mu_2(v)\otimes \mu_1(-v)$.
\end{enumerate}
\end{lem}

\begin{proof} In the diagram in Figure~\ref{fig:psi_r},
  the outer pentagon is commutative
  by the pentagon axiom and the internal part of the diagram gives the
  construction of  $\psi_r\colon \lambda (v)\to \lambda (r+v)$. (Note
  that, since $r\in \langle \rho \rangle$, $\mu_1(r)=\mu_2(r)$, and we
  write simply $\mu(r)$ to signify either of these two equal line bundles.)
  \begin{figure}[h]
    \centering\vspace{-.5cm} 
  \begin{tikzpicture}[every node/.style={align=center}, arrow
    style/.style={->,>=Stealth},scale=0.8, every
    node/.style={transform shape}]
\node[align=center] (top) {$\Bigl( \mu_1(-v) \otimes \mu(-r) \Bigr) \otimes \Bigl( \mu(r) \otimes \mu_2(v) \Bigr)$};

\node[below left=1cm and -1cm of top] (midleft) {$\mu_1(-v) \otimes \Big[ \mu(-r) \otimes \Big(\mu(r) \otimes \mu_2(v)\Big) \Big]$};
\node[below right=1cm and -1cm of top] (midright) {$\Bigl[ \Bigl(\mu_1(-v) \otimes \mu(-r)\Bigr) \otimes \mu(r) \Bigr] \otimes \mu_2(v)$};

\node[below=.5cm of top] (lvr) {$\lambda(v+r)$};
\node[below=2cm of top] (lv) {$\lambda(v)$};
\draw[-] (top.south west) -- (midleft.north);
\draw[->] (top.south) -- (lvr.north);
\draw[-] (lvr.south) -- (lv.north) node[midway, left] {$\psi_r$};
\draw[-] (top.south east) -- (midright.north);

\node[below=1.2cm of midleft] (botleft) {$\mu_1(-v) \otimes \Bigl[ \Bigl(\mu(-r) \otimes \mu(r)\Bigr) \otimes \mu_2(v) \Bigr]$};
\node[below=1.2cm of midright] (botright) {$\Bigl[\mu_1(-v) \otimes  \Bigl(\mu(-r) \otimes \mu(r)\Bigr)\Bigr] \otimes \mu_2(v) $};

\draw[-] (midleft) -- (botleft);
\draw[-] (midright) -- (botright);
\draw[-] (botleft) -- (botright);

\draw[->] (botleft.north east) -- (lv.west);
\draw[->] (botright.north west) -- (lv.east);

\end{tikzpicture}
\caption{Construction of $\psi_r\colon \lambda(v)\to \lambda(r+v)$}
\label{fig:psi_r}
\end{figure}

  The statement that for all $r,s\in \langle \rho\rangle$,
  $\psi_{r+s}=\psi_r\circ\psi_s$ follows from contemplating the
  diagram in Figure~\ref{Helge's diagram}, where the outer circle is
  commutative by Mac Lane coherence theorem for monoidal
  categories. For reasons of space we are suppressing the $\otimes$
  symbols.
\begin{figure}[h]
    \centering
  \begin{tikzpicture}[every node/.style={align=center}, arrow
    style/.style={->,>=Stealth},scale=0.75, every
    node/.style={transform shape}]
\def\radius{6cm}
\node at (198:\radius) (n1) {$ \mu_1(-v)\Big\{\mu(-s)\Big[\Big(\Big\{\mu(-r) \mu(r)\Big\}\mu(s)\Big)\mu_2(v) \Big]\Big\}$};
\node at (162:\radius) (n2) {$ \mu_1(-v)\Bigg\{\Big[\mu(-s)\Big(\Big\{\mu(-r) \mu(r)\Big\}\mu(s)\Big) \Big]\mu_2(v)\Bigg\}$};
\node at (136:\radius) (n3) {$ \mu_1(-v)\Bigg\{\Big[\mu(-s)\Big(\mu(-r) \Big\{\mu(r)\mu(s)\Big\}\Big) \Big]\mu_2(v)\Bigg\}$}; 
\node at (90:\radius+.2cm) (n4) {$\mu_1(-v)\Bigg\{\Big[\Big(\mu(-s)\mu(-r)\Big) \Big\{\mu(r)\mu(s)\Big\} \Big]\mu_2(v)\Bigg\}$}; 
\node at (44:\radius) (n5) {$ \mu_1(-v)\Bigg\{\Big(\mu(-s)\mu(-r)\Big)\Big[ \Big(\mu(r)\mu(s)\Big) \mu_2(v)\Big]\Bigg\}$}; 
\node at (18:\radius) (n6) {$ \Big\{\mu_1(-v)\Big(\mu(-s)\mu(-r)\Big)\Big\} \Big[ \Big(\mu(r)\mu(s))\Big)\mu_2(v) \Big]$};
\node at (342:\radius) (n7) {$ \Big\{\Big(\mu_1(-v)\mu(-s)\Big)\mu(-r)\Big\} \Big[ \mu(r)\Big(\mu(s)\mu_2(v)\Big) \Big]$};
\node at (316:\radius) (n8) {$ \Big(\mu_1(-v)\mu(-s)\Big)\Big\{\mu(-r) \Big[ \mu(r)\Big(\mu(s)\mu_2(v)\Big) \Big]\Big\}$};
\node at (270:\radius) (n9) {$ \Big(\mu_1(-v)\mu(-s)\Big)\Big\{\Big[\mu(-r) \mu(r)\Big]\Big(\mu(s)\mu_2(v)\Big) \Big\}$}; 
\node at (224:\radius) (n10) {$ \Big(\mu_1(-v)\mu(-s)\Big)\Big\{\Big(\Big\{\mu(-r) \mu(r)\Big\}\mu(s)\Big)\mu_2(v) \Big\}$}; 
\draw (n1) to [bend left=10] (n2);
\draw (n2) to [bend left=10] (n3);
\draw (n3) to [bend left=10] ($(n4.south west)!0.25!(n4.south east)$);
\draw ($(n4.south west)!0.75!(n4.south east)$) to [bend left=10] (n5);
\draw (n5) to [bend left=10] (n6);
\draw (n6) to [bend left=10] (n7);
\draw (n7) to [bend left=10] (n8);
\draw (n8) to [bend left=10] (n9);
\draw (n9) to [bend left=10] (n10);
\draw (n10) to [bend left=10] (n1);
\node at (0:\radius +.9cm) (ll4) {\hbox{two moves}};
\node at (90:\radius -4.5cm) (ll4) {$\lambda(v)$};
\draw[arrow style] (n4) -- (ll4);
\draw[arrow style] (n2) -- (ll4);
\node at (0:\radius -4.5cm) (ll6) {$\lambda(r+s+v)$};
\draw[arrow style] (n6) -- (ll6);
\draw[arrow style] (n7) -- (ll6);
\draw (ll4) to [bend left=10] (ll6);
\node at (34:\radius -4.2cm) {$\psi_{r+s}$};
\node at (270:\radius -4.5cm) (ll9) {$\lambda(s+v)$};
\draw[arrow style] (n9) -- (ll9);
\draw[arrow style] (n10) -- (ll9);
\draw (ll6) to [bend left=10] (ll9);
\node at (330:\radius -4.4cm) {$\psi_r$};
\draw (ll9) to [bend left=20] (ll4);
\node at (180:\radius -5.2cm) {$\psi_{s}$};
\end{tikzpicture}
\caption{Proof that $\psi_{r+s}=\psi_r\circ \psi_s$.}
\label{Helge's diagram}
\end{figure}
\end{proof}

\begin{construction}
  \label{con:slab_bundle}
For every slab $\rho \in T^{[1]}$, we construct a line bundle
$\shL_\rho$ on $Y_\rho$, which we will call a slab bundle. 

Fix throughout $\rho\in T^{[1]}$.

First off, we construct a covering of $Y_\rho$ by Zariski open
subsets. Note that
\[
Y_\rho = \coprod_{\{\tau\in T\mid \tau\leq \rho\}} Y_\tau^\star
\]
is itself a stratified space. We denote by $T_\rho = \{\tau\in T\mid
\tau\leq \rho\}$ the index poset for the strata of $Y_\rho$. For all
$\tau \in T_\rho$, we denote by $V_\tau \subset Y_\rho$ the open star
of $\tau$ in $Y_\rho$. It is clear that $\{V_\tau \mid \tau \in
T_\rho\}$ is a Zariski open cover of $Y_\rho$.

Now choose a \emph{global numbering} for the maximal cones incident at $\rho$:
in other words, for all $\tau \in T_\rho$, a numbering
$\sigma_1(\tau)$, $\sigma_2(\tau)$ of the two maximal cones in
$\Sigma_\tau$ incident at $\rho$,\footnote{For all $\tau \in T_\rho$,
  we can identify $\rho$ with a submaximal cone in $\Sigma_\tau$. To be precise,
  we ought to denote this cone by $\rho_\tau$ but we abuse
  notation and denote all these cones by $\rho$ trusting that this
  will not cause confusion.} satisfying the consistency
condition: If $\tau_1\leq \tau_2$, then, for $i=1,2$, $\sigma_i(\tau_1)$ maps to
$\sigma_i(\tau_2)$ under the natural projection $M_{\tau_1} \to
M_{\tau_2}=M_{\tau_1}/\langle \tau_2\rangle$.\footnote{The existence
  of such a global numbering follows easily from the conditions
  spelled out in Section~\ref{sec:basic_setup}.} Note that there are
precisely two global numberings. 

For all $\tau$, Let
$d_\tau \in \sigma_2(\tau)^\vee \subset \Hom (M_\tau, \ZZ)$ be the
unique primitive vector that pairs to zero with all points in
$\rho$. Next choose $v_\tau \in M_\tau$ such that
$\langle d_\tau, v_\tau\rangle=1$.  Following
Lemma~\ref{lem-no-choice} with $y=\tau$, we define a line bundle
$\lambda (v_\tau)$ on $V_\tau$ as
\[
  \lambda (v_\tau)=\mu_1(\tau) (-v_\tau)\otimes \mu_2(\tau) (v_\tau)
\]
where we denote by $\mu_i(\tau) \colon M_{\tau} \to \PPic (V_\tau)$
the \mbox{$2$-homomorphisms} of Lemma~\ref{lem-no-choice}.

Next we glue the $\lambda(v_\tau)$ to a line bundle on all of
$Y_\rho$. If $\tau_1\leq \tau_2$, then
\begin{enumerate}[(1)]
\item $V_{\tau_2}\subset V_{\tau_1}$;
\item we have an identification
  $\pi=\pi_{\tau_1,\tau_2}\colon M_{\tau_1} \to M_{\tau_1}/\langle
  \tau_2\rangle =M_{\tau_2}$;
\item under the dual injection $N_{\tau_2}\subset N_{\tau_1}$,
  $d_{\tau_2}$ is identified with $d_{\tau_1}$ and
  $\langle d_{\tau_2,}\pi(v_{\tau_1}) \rangle =1$ and hence we can form
  $r_{\tau_1,\tau_2}= \pi(v_{\tau_1})-v_{\tau_2}\in \langle\rho\rangle \cap M_{\tau_2}$.
\end{enumerate}
At this point, we
\[
\text{glue $\lambda(v_{\tau_1})_{|V_{\tau_2}}=\lambda(v_{\pi(v_{\tau_1})})$ to
  $\lambda(v_{\tau_2})$ with $\psi_{r(\tau_1,\tau_2)}
  \colon \lambda(v_{\pi(v_{\tau_1})})\overset{\cong}{\longrightarrow}
  \lambda(v_{\tau_2})$}  
\]
whose construction is given by Lemma~\ref{lem-no-choice}. 

Next, we keep gluing on all the line bundles on all the $V_\tau$ formed from both
choices of numbering and all system of vectors $(v_{\tau})_{\tau \in
  M_\tau}$ as above by using the $\psi$ isomorphisms as above. The
cocycle condition ensures that all these gluings can be done
consistently.

We denote by $\shL_\rho$ the resulting line bundle on $Y_\rho$.
\end{construction}

\begin{dfn}
  \label{dfn:slab_bundle}
  Let $Y=\coprod_{\tau \in T} Y^\star_\tau$ be a viable \ggtc{} space and
  $\rho\in T^{[1]}$. We call the line bundle $\shL_\rho$ on $Y_\rho$
  of Construction~\ref{con:slab_bundle} the \emph{slab bundle}.

Abusing notation, we also denote by $\shL_\rho$ the direct image of this line bundle under the inclusion $Y_\rho\hookrightarrow Y$.
\end{dfn}

  


\subsection{The joint condition and the sheaf $\shLS_Y$}
\label{sec:joint-condition}
We show that for every joint $\omega\in T^{[2]}$ the restrictions of
all $\shL_\rho$ with $\omega<\rho$ to $Y_\omega$ satisfy a
relation. We then use this relation to define a subsheaf
$\shLS_Y \subset \bigoplus_{\rho\in T^{[1]}} \shL_\rho$.  A similar
relation first appeared as a relation of elements in a local
computation in Theorem 3.22 in \cite{MR2213573}. 

\smallskip

The following definition and theorem are stated in terms of a strictly
commutative \mbox{$2$-group} $\underline{A}$, because this is what we
need. The statement contains as a special case the situation where
$\underline{A}$ is the categorification of an abelian group $A$, where
a more concrete formulation of the definition and theorem are possible.

\begin{dfn}
  \label{dfn:contPL2homo}
  Let $M$ be a finitely generated free abelian group
  and $\Sigma$ a complete fan in $M$. Let
  $\underline{A}$ be a strictly commutative \mbox{$2$-group}.  A
  \emph{continuous piecewise linear \mbox{$2$-homomorphism}}
  $\mu\colon M\to \underline{A}$ with respect to $\Sigma$ is a
  collection of \mbox{$2$-homomorphisms}
\[
\mu_\sigma \colon M \to \underline{A}\, ,
\]
one for each maximal cone $\sigma\in\Sigma$, with the property that
whenever two maximal cones $\sigma_1,\sigma_2$ share a submaximal cone
$\rho=\sigma_1\cap\sigma_2$ then the restrictions of
$\mu_{\sigma_1}, \mu_{\sigma_2}$ to $\langle\rho\rangle$ are equal.
\end{dfn}

\begin{lem}[Abstract joint condition lemma]
  \label{lem:abstract_joint_conditon}
  Let $M$ be a lattice, $N=\Hom (M,\ZZ)$, and assume given a
  surjective homomorphism $\pi \colon M \to \overline{M}$, where
  $\overline{M}$ is a rank two lattice endowed with a complete fan
  $\overline{\Sigma}$. Endow $M$ with the fan
  \[
    \Sigma = \{\pi^{-1}(\eta)\mid \eta \in \overline{\Sigma}\}
  \]
  We denote the cones of $\Sigma$ as follows:
  \begin{itemize}
  \item the codimension-$2$ subspace (a.k.a.~joint) $\omega=\pi^{-1}\{0\}$;
  \item cyclically ordered codimension-$1$ cones (a.k.a.~slabs)
    $\rho_1, \dots, \rho_n$ incident at
    $\omega$;
  \item maximal cones $\sigma_i=\langle \rho_i\cup\rho_{i+1}\rangle_+$
    (for $i=1,\dots, n$ with the convention that $\rho_{n+1}=\rho_1$).
  \end{itemize}
  Let
  $(\underline{A},+,0_{\underline{A}})$ be a strictly commutative \mbox{$2$-group} and
  $\mu \colon M \to \underline{A}$ a continuous piecewise linear
  \mbox{$2$-homomorphism} with respect to $\Sigma$. Denote by
  $\mu_i=\mu_{\sigma_i}$ the \mbox{$2$-homomorphism} for
  $\sigma_i$ that forms part of the datum $\mu$.  Let $d_1, \dots, d_n \in N$ be
  the primitive normals to the slabs $\rho_1, \dots,\rho_n$, such that
  $d_i\ge 0$ on $\rho_{i+1}$.

  Choose $v_i\in M$ such that $\langle d_i,v_i\rangle =1$ and define
\begin{equation}
\lambda(v_i)=\mu_{i-1}(-v_i)+\mu_i(v_i)\,,
\label{abstr-lem-ai}
\end{equation}
\begin{enumerate}[(1)]
\item We construct the \emph{joint isomorphism}
\begin{equation}
  \sum_{i=1}^n d_i\otimes \lambda (v_i) \overset{\cong}{\longrightarrow}
  0 \otimes 0_{\underline{A}} 
\;\; \text{in}\;\;N\otimes \underline{A} =\underline{\twoHom} (M,\underline{A})\,,
   \label{eq:telescope}
 \end{equation}
\item The isomorphism constructed in Part~(1) does not depend on the
  choice of $v_i$ in the following sense. For all $i=1,\dots,n$ and
  $r_i\in \langle \rho_i \rangle$ we construct isomorphisms
\[
  \psi_{r_i} \colon \lambda(v_i) \overset{\cong}{\longrightarrow} \lambda(v_i+r_i)
\]
 such that the following diagram is commutative
\[
  \xymatrix{\sum_{i=1}^n d_i\otimes \lambda (v_i) \ar[dd]_{\{\psi_{r_i}\}}\ar[dr]^{\cong} & \\
    &  0 \otimes 0_{\underline{A}} \\
    \sum_{i=1}^n d_i\otimes \lambda (v_i+r_i) \ar[ur]_{\cong}  & } \,,
\]
and, furthermore, the set of these isomorphisms has the cocycle property:
\begin{enumerate}[(i)]
\item $\psi_{0}=\id $;
\item  for all $r_i,s_i\in \langle \rho_i\rangle$,
  $\psi_{r_i+s_i}=\psi_{r_i}\circ \psi_{s_i}$. 
\end{enumerate}
\item The joint isomorphism does not depend on the cyclic ordering in the sense that the isomorphism from the opposite ordering relates to the given one by a global multiplication by $(-1)$.
\end{enumerate}
\end{lem}

\begin{proof} For all $w\in M$ we produce an isomorphism
$\sum_{i=1}^n d_i(w)\lambda(v_i) \cong 0_{\underline{A}}$ in
$\underline{A}$. The key observation is that for all $i$
\[
d_i\left(-d_i(w) v_i +w \right) =
-d_i(w)\underbrace{d_i(v_i)}_{=1}+d_i(w)=0\,. 
\]
Hence $-d_i(w) v_i +w  \in \langle\rho_i\rangle$ and therefore
$\mu_{i-1}\bigl(-d_{i}(w)v_i+w\bigr)=\mu_i\bigl(-d_i(w)v_i+w\bigr)$.
Rewriting and using the natural transformations of functors provided by the monoidal structure of $\underline{A}$ yields
\begin{equation}
  d_i(w)\mu_{i-1}(-v_i) \cong d_i(w)\mu_i(-v_i) + \mu_i(w) -\mu_{i-1}(w).
\label{abstr-lem-proof-step}
\end{equation}
The sum \eqref{eq:telescope} in the assertion can be identified as a
telescoping sum when writing
 \begin{equation*}
   \begin{split}
     \sum_{i=1}^n d_i(w) \lambda (v_i)  \stackrel{\eqref{abstr-lem-ai}}{\cong}& \hspace{.5cm}d_1(w) \bigl(\mu_n(-v_1)+\mu_1(v_1)\bigr)\\[-3mm]
     & +d_2(w)\bigl(\mu_1(-v_2)+\mu_2(v_2) \bigr)\\
     & +d_3(w) \bigl(\mu_2(-v_3)+\mu_3(v_3) \bigr)\\
     & \hspace{.5cm}\vdots \\
    \stackrel{\eqref{abstr-lem-proof-step}}{\cong}& \hspace{.5cm}d_1(w)\bigl(\mu_1(-v_1)+\mu_1(v_1)\bigr)+\mu_1(w)-\mu_n(w) \\
     &+d_2(w)\bigl( \mu_2(-v_2)+\mu_2(v_2)\bigr)+\mu_2(w)-\mu_1(w)\\
     &+d_3(w)\bigl( \mu_3(-v_3)+\mu_3(v_3)\bigr)+\mu_3(w)-\mu_2(w)\\
     & \hspace{.5cm}\vdots \\
     \cong &\hspace{.5cm} 0_{\underline{A}}
   \end{split}
 \end{equation*}
 The resulting isomorphism to $0_{\underline{A}}$ does not depend on
 choices --- that is, the order of association and commutation in
 $\underline{A}$ --- by the Mac Lane coherence theorem for
 \mbox{$2$-groups}.

 For the independence on the choice of $v_i$, Part~(1) is proved in
 the same way as the corresponding statement in
 Lemma~\ref{lem-no-choice}. The proof of
 Part~(2) is straightforward and we omit the details. 
\end{proof}

\begin{cor}
  \label{cor:joint-compat}
  Let $Y$ be a viable \ggtc{} space, $\omega \in T^{[2]}$ a joint,
  and $\rho_1, \dots, \rho_n$ a cyclical ordering of the slabs
  incident at $\omega$. Let $d_1, \dots, d_n \in N_\omega$ be
  the primitive normals to the slabs $\rho_1,
  \dots,\rho_n$, such that
  $d_i\ge 0$ on $\rho_{i+1}$.

  We construct a \emph{joint isomorphism}:
\begin{equation}
J_\omega\colon \bigotimes_{i=1}^n d_i\otimes \shL_{\rho_i|Y_\omega}
\cong 0 \otimes \shO_{Y_\omega} \quad \text{in} \quad 
N_\omega \otimes \PPic Y_\omega \,.
   \label{joint-iso}
 \end{equation}
\end{cor}

\begin{proof}[Sketch of proof] 
  Fix throughout a joint $\omega \in T^{[2]}$. 

  The space $Y_\omega=\coprod_{\{\tau \in T \mid \tau \leq
  \omega\}} Y_\tau^\star$ is stratified. We denote by
$T_\omega=\{\tau \in T \mid \tau \leq \omega\}$ the index poset for
the strata of $Y_\omega$. For all $\tau\in T_\omega$, we denote by
$U_\tau \subset Y$ the open star of $\tau$ in $Y$. It is clear that
$\{U_\tau \cap Y_\omega \mid \tau \leq \omega\}$ is a Zariski open
cover of $Y_\omega$.

 Now $\Sigma_\omega$ is a fan in a rank two lattice $M_\omega$. Denote
 the cones of $\Sigma_\omega$ as follows:
 \begin{itemize}
 \item the cone $(0)$, corresponding to the joint $\omega$ itself;
 \item cyclically ordered rays $\overline{\rho}_1,\dots, \overline{\rho}_n$, a.k.a. slabs;
 \item maximal cones $\overline{\sigma}_i=\langle \overline{\rho}_i,\overline{\rho}_{i+1}\rangle_+$.
 \end{itemize}
Denote by $d_1,\dots, d_n\in N_\omega$ the primitive normals to the slabs
$\overline{\rho}_1, \dots, \overline{\rho}_n$ such that $d_i> 0$ on
$\overline{\rho}_{i+1}$. 

For all $i$, Construction~\ref{con:slab_bundle} constructs a slab
bundle $\shL_{\rho_i}$on $Y_{\rho_i}$.

The bundles $\shL_{\rho_i}$ are obtained from gluing together certain bundles
constructed on certain open covers of $Y_{\rho_i}$. We are only interested in the open subsets $U_\tau \cap
Y_{\rho_i}$ for $\tau \in T_\omega$: these open subsets don't cover
all of $Y_{\rho_i}$ but they do cover all of $Y_\omega$ and hence they
are sufficient for working with $\shL_{\rho_i|Y_\omega}$. Let us fix $\tau \in T_\omega$
and focus on one of these open subsets $U_\tau$.

We are going to apply Lemma~\ref{lem:abstract_joint_conditon} to the
situation $M=M_\tau$, $\overline{M}=M_\omega=M/\langle \omega\rangle$,
$\pi\colon M \to \overline{M}$ the projection to the quotient,
$\overline{\Sigma}=\Sigma_\omega$, and $\Sigma=\omega^{-1}
\Sigma_\tau$ the localized fan. Abusing notation slightly, we denote
by $\omega=\pi^{-1}(0)$, 
$\rho_i=\pi^{-1}(\overline{\rho}_i)$,
$\sigma_i=\pi^{-1}(\overline{\sigma}_i)$ the cones of $\Sigma$. Furthermore, we take
$\underline{A}=\PPic {Y_\omega}$, and $\mu_i\colon M \to
\underline{A}$ the unique \mbox{$2$-homomorphism} such that, for $m\in
\sigma_i\cap M$,\footnote{Recall that $\widetilde{X}$ is the
  normalization of $Y$: it is a stratified space where strata are
  pairs $(\tau, \sigma)$ of $\tau \leq \sigma\in T$ and $\sigma\in T^{[0]}$.}
\[
\mu_i(m) =\iota_i^\star \bigl(\cO_{\widetilde{U}_{(\tau, \sigma_i)}}\bigl(\widetilde{D}_{(\tau,\sigma_i)}(m)\bigr) \bigr)\,.
  \]

Consider local charts for the bundles $\shL_{\rho_i}$ constructed
by choosing a global numbering compatible with the cyclic order of the
rays $\rho_i$. The construction of the local
charts for $\shL_{\rho_i}$ further depend on vectors $v_i\in M$ such that $\langle
v_i,d_i\rangle=1$. The corresponding local chart for
$\shL_{\rho_i|U_\tau \cap Y_\omega} $ is $\lambda(v_i)$.
Lemma~\ref{lem:abstract_joint_conditon} then gives a joint isomorphism
\[
  J_\tau (\mathbf{v})
  \colon \bigotimes_{i=1}^n d_i\otimes \lambda(v_i)
  \cong 0 \otimes \shO_{U_\tau \cap Y_\omega}
  \quad \text{in} \quad N_\omega \otimes \PPic
(U_\tau \cap Y_\omega)
\]
defined locally on $U_\tau$ and depending on $\textbf{v}=(v_1,\dots
v_n)$ and the global numbering.

We need to prove that these local joint isomorphism glue to
give a global joint isomorphism. For this purpose, for all $\tau \leq
\tau^\prime$ (the case $\tau=\tau^\prime$ is included!), we need to check that
\[
J_{\tau} (\textbf{v})_{|U_{\tau^\prime} \cap Y_\omega}=J_{\tau^\prime}(\textbf{v}) \,.
\]
For a fixed list of vectors $\textbf{v}=(v_i)$ this is obvious, but we
also need to address the possibility of changing $\textbf{v}=(v_i)$. The
required consistency follows from the way that the local joint
condition behaves under change of $(v_i)$ given in
Lemma~\ref{lem:abstract_joint_conditon}(2). We also need to check
consistency under change of global numbering around each of the
$\rho_i$; this follows from Lemma~\ref{lem-no-choice}, Part~(3) and Lemma~\ref{lem:abstract_joint_conditon}, Part~(3).
\end{proof}

\begin{rem}
  \label{rem:joint-compat}
If $e_1,e_2$ is a lattice basis of $M_\omega$, the map $J_\omega$ is equivalent to two isomorphisms of line bundles
$$ J_{\omega,e_1}\colon \bigotimes_{i=1}^n (\shL_{i|Y_\omega})^{\otimes d_i(e_1)} \cong \shO_{Y_\omega};\qquad  J_{\omega,e_2}\colon \bigotimes_{i=1}^n (\shL_{i|Y_\omega})^{\otimes d_i(e_2)} \cong \shO_{Y_\omega}.$$   
\end{rem}

\begin{dfn}
  \label{def_shLS}
  We define the sheaf of sets $\shLS_Y$ as the subsheaf of the direct
  sum $\bigoplus_{\rho\in T^{[1]}} \shL_\rho$ on $Y$ satisfying the
  following condition.  For every point $y\in Y$, the stalk $\shLS_y$
  consist of those tupels of sections $(f_\rho)_{\rho\in T^{[1]}}$
  whose restrictions to $Y_\omega$ for every $\omega\in T^{[2]}$
  satisfies the \emph{joint condition}
  \[
    J_\omega\big(d_i\otimes (f_{i|Y_\omega}) \big)= 0 \otimes 1
  \]
at the generic point $\omega$ of $Y_\omega$.
\end{dfn}

\begin{notation}
  \label{nota:never_vanish}
  We denote by
  \[
    \shLS^\times_Y =\shLS_Y\cap \bigl(\oplus_{\rho\in
      T^{[1]}}\shL_\rho^\times\bigr) 
\]
  the subsheaf of nowhere vanishing sections.
\end{notation}

\begin{rem}
  \label{rem:shLS}
  \begin{enumerate}[(1)]
  \item In the special situation where $T^{[2]}=\emptyset$, we get
    $\shLS_Y=\bigoplus_{\rho\in T^{[1]}} \shL_\rho$ and in this
    situation $\shLS_Y$ is a coherent sheaf. It can be seen that, more
    generally, if $Y$ is a simple normal crossing scheme, then
    $\shLS_Y$ is a line bundle on $Y^{(1)}$.  
\item In Definition~\ref{def_shLS} we require the joint condition to
  hold at the generic point of $Y_\omega$ only. The only reason for
  not requiring it everywhere is to allow the sections $f_\rho$ to
  vanish somewhere.
  \end{enumerate}
\end{rem}

\section{$\shLS_Y^\times$ classifies log structures on $Y$}
\label{section4}

\subsection{Statement of the main result and road-map of its proof}
\label{sec:stat-main-result}

In this section, we prove the main result of the paper,
Theorem~\ref{mainmaintheorem}. Before we can give the precise
statement, we need to recall a few notions about log structures and
define some key concepts that enter it. The next definition is really
just meant to fix our notation.

\begin{dfn}
  \label{dfn:log_schemes}
  Let $X$ be a space. 
  \begin{enumerate}[(1)]
  \item A \emph{log structure} on $X$ is a pair $(\foP, \alpha)$ where $\foP$
is a sheaf of monoids\footnote{In this paper we take $\foP$ to be a
  sheaf in the Zariski topology. When reading this section, it helps
  to internalize early on that our sheaves are sheaves of monoids or
  groups and rarely are they coherent sheaves.} and $\alpha \colon
\foP \to (\shO_X, \times)$ is a homomorphism of sheaves of monoids
such that
\[
  \alpha_{|\alpha^{-1}(\shO_X^\times)} \colon
  \alpha^{-1}(\shO_X^\times) \to \shO_X^\times
\]
is an isomorphism.
\item A \emph{log scheme} is a pair $(X,\foP)$ of a scheme
$X$ and a log structure $\foP$. The symbol $X^\dagger$ signifies a
log scheme with underlying scheme $X$.
\item A \emph{morphism} of log schemes $f\colon(X,\foP)\to (Y,\foQ)$
  is an ordinary morphism of schemes, together with a homomorphism of
  sheaves of monoids $f^{-1}\foQ\to \foP$ that commutes with
  $f^{-1}\shO_Y\to \shO_X$ under the respective maps $\alpha$.
\item Let $k$ be a field. The \emph{standard log point} is the log scheme
$\Spec k^\dagger=(\Spec k,\foP_k)$, where $\foP_k=k^\times\times\NN$
and $\alpha\colon \foP_k\to k$ maps $(a,n)$ to $0$ if $n>0$ and to $a$
if $n=0$. 
\item Let $k$ be a field, $X$ a scheme over $\Spec k$, and
  $X^\dagger=(X,\foP)$ a log scheme. Note that to give a morphism
  $X^\dagger \to \Spec k ^\dagger$ is equivalent to give a global
  section $\one_{\foP}\in\Gamma(X,\foP)$ with
  $\alpha(\one_{\foP})=0$.\footnote{Given
    $X^\dagger \to \Spec k^\dagger$, $\one_\foP$ is the image of
    $1\in\NN$.}

  A log structure on $X$ \emph{over the standard log point}, or simply a log
structure on $X$ over $k^\dagger$, written
$X^\dagger/k^\dagger$, is a morphism $X^\dagger \to \Spec k^\dagger$ of
log schemes; equivalently, it is a pair $(X^\dagger, \one_{\foP})$ of a
log scheme $X^\dagger=(X,\foP)$ and section $\one_{\foP}\in \Gamma
(X,\foP)$ as just described.
\item The \emph{ghost sheaf} of a log structure $\foP$ is the quotient
sheaf ${\overline \foP}:=\foP/\alpha^{-1}(\shO_X^\times)$. 
We denote by $\one_{\overline \foP}\in \Gamma(X,{\overline \foP})$ the image of $\one_{ \foP}\in \Gamma(X,{ \foP})$.
The \emph{relative ghost sheaf} of a log scheme $X^\dagger/k^\dagger$
is the quotient sheaf $\overline\shM:=\overline\foP /\one_{\overline\foP}$.
  \end{enumerate}
\end{dfn}

Before reading the upcoming definition, the reader is advised to
rehearse the definition of viable \ggtc{} space.

\begin{dfn}
   \label{dfn:compatible}  Let  $\bigl( Y=\coprod_{\eta\in T} Y_\eta^\star,(\shP,
    \mathbf{1}),\{\widehat{f_\eta}\mid \eta\in T\}\bigr)$ be a viable \ggtc{} space. 
   \begin{enumerate}[(1)]
   \item   A \emph{log structure compatible with the \ggtc{}
       structure} on $Y$, or simply a \emph{compatible log structure}, is
   a log structure on $Y$ over $k^\dagger$,  $\bigl((Y,
   \foP), \one_{\foP}\bigr)$, together with a homomorphism of sheaves
   of monoids
\[
\psi \colon \foP \to \shP
\]
 such that:
 \begin{enumerate}[(a)]
 \item $\psi (\one_{\foP})=\one_{\shP}$ and $\psi$ induces an
   isomorphism
   \[
     \overline{\psi} \colon \foP/\shO_Y^\times = \overline{\foP}
     \overset{\cong}{\longrightarrow} \shP\]
   of ghost sheaves. In particular, $\psi$ also induces an isomorphism $\overline\shM\to \shM$ of the relative
   ghost sheaf of the log structure $Y^\dagger/k^\dagger$ to the relative ghost
   sheaf of the \ggtc{} space $Y$.
 \item For all $y\in Y$ and $p \in \foP_y$, denote by $[\psi (p)] \in
   \shM_y$ the image of $p$, and let $\tau \in \Sigma_y$ be the
   smallest cone that contains $[\psi (p)]$.\footnote{If $\psi(p)$ does
   not lie on a proper face of $\shP_y$, then of course $[\psi
   (p)]=(0)$.}

   The condition is:
   $\alpha(p)\in \shO_{Y, y}$ does not vanish identically on any
   component of $\overline{U}_\tau\cap U_y$, and
\[
\div \left( \alpha (p) \right)  = D_{\eta,\tau} ([\psi(p)]) \quad \text{in}
\quad\Div^+_\flat \bigl( \overline{U}_\tau\cap U_y \bigr)\,,
\]
where $\eta=b(y)$ and $D_{\eta,\tau} ([\psi(p)])$ is the Cartier divisor of
Definition~\ref{dfn:viability}  (its existence guaranteed by the
viability condition) and Notation~\ref{nota:viability}.
 \end{enumerate}
\item A \emph{morphism} of compatible log structures $(\foP, \psi)$,
  $(\foP^\prime,\psi^\prime)$ on $Y$ is a morphism of log structures
  $\varphi \colon \foP \to \foP^\prime$ such that for all sections
  $p\in \foP$, $\psi (p)=\psi^\prime (\varphi (p))$.
\item We denote by $\LS (Y)$ the set of isomorphism classes of
  compatible log structures on $Y$. 
  \end{enumerate}
\end{dfn}
  
\begin{thm} 
\label{mainmaintheorem}
Let $Y$ be a viable \ggtc{} space, and let
$\shLS_Y\subset\bigoplus_\rho \shL_\rho$ be the sheaf of
Definition~\ref{def_shLS}.

Denote by $\LS(Y)$ the set of isomorphism classes of log structures
on $Y$ over $k^\dagger$ compatible with the \ggtc{} structure.

The set-theoretic function
\[
  r\colon \LS(Y) \to \Gamma (Y,\shLS_Y^\times)
\]
constructed in \eqref{morphism-r} is a bijection.
\end{thm}

Next we give an outline of the proof. Details are carried out in the
subsections that follow. In outline, our proof follows closely the
proof of~\cite[Theorem~3.22]{MR2213573}; however, there are important
differences due to the fact that we start out from an
independent construction of the sheaf $\shLS_Y$.

We conclude with a synopsis of the following subsections.

\begin{proof}[Outline of the proof of Theorem~\ref{mainmaintheorem}]
Fix a viable \ggtc{} space $Y$ as above. 
A compatible log structure $Y^\dagger=(Y, \foP)$ sits in an extension
sequence
\[
0 \to \shO_Y^\times \to \foP \stackrel{\psi}\to \shP \to 0 \, .
\]
Recall that the relative ghost sheaf $\shM=\shP/\textbf{1}_\shP$ is a sheaf of groups. We
have that
\[
  \foP=\foM \times_{\shM} \shP
\]
where $\foM=\foP/\one_{\foP}$ is a sheaf of abelian groups\footnote{It is
  important to appreciate that $\foM$ is not in any sensible way a log
  structure: we lack the monoid homomorphism $\foM \to \shO_Y$.} that
is an extension in the category of sheaves of abelian groups on $Y$,
\[
0 \to \shO_Y^\times \to \foM \to \shM \to 0\, .
\]
The relative ghost sheaf $\shM$ is supported in
codimension one and $Y$ is reduced, so $\shHom(\shM, \shO_Y^\times) = 0$. The local-to-global Ext spectral sequence then gives $\Ext^1(\shM, \shO_Y^\times) =
H^0\bigl(Y, \shExt^1(\shM, \shO_Y^\times)\bigr)$.

A key point of the proof is to characterize the extensions that
give rise to compatible log structures. We introduce a subsheaf
\[
\shExtc^1(\shM, \shO_Y^\times) \subset \shExt^1(\shM, \shO_Y^\times)
\]
which we call the sheaf of \emph{regular extensions}, defined by a
prescribed asymptotic behaviour at the boundary. This subsheaf is the
same as the corresponding subsheaf defined in \cite{MR2213573} by
fixing a \emph{ghost type}. Here, we get around choosing local charts
for the log structure by defining the intrinsic data of the
\emph{local divisor system}
(Def.~\ref{dfn:local_divisor_system}) and embedding it in the
total ring of fractions.

The next step is to construct a sheaf homomorphism
\[
  \varphi\colon \shExtc^1(\shM, \shO_Y^\times)\to \shLS^\times_Y \, .
\]
Here we need to depart from~\cite{MR2213573}: we introduce the
\emph{local line bundle system} (Def.~\ref{dfn:local-line-bundle}) and
construct $\varphi$ as an application of the abstract joint condition
Lemma~\ref{lem:abstract_joint_conditon}. 

While the proof of injectivity of $\varphi$ is very similar
to~\cite[Theorem~3.22]{MR2213573}, the proof of surjectivity
departs somewhat from \cite{MR2213573}. It
requires us to introduce \mbox{$2$-homomorphisms} from a lattice into the local line
bundle system (Proposition~\ref{PLTY}), sections of these and then gluing them using the abstract joint condition once more.

Finally we construct \[
  \psi\colon \LS(Y) \to \shExtc^1(\shM, \shO_Y^\times)
\] and prove that it is also bijective. Composing with $\varphi$ gives a 
bijection
\[
r\colon 
\LS (Y)\overset{\cong}{\longrightarrow}\shLS^\times_Y 
\]

This \mbox{$2$-homomorphism} in the surjectivity of $\varphi$ also plays a central role in the
construction of a log structure from a section of $\shLS_Y$, see
Proposition~\ref{pro:main_th_local_case}. Our point of view in the proof of
Proposition~\ref{pro:main_th_local_case} closely matches that
of~\cite{MR2964607} as was pointed out to us by Bernd Siebert. A log structure in \cite{MR2964607} is defined as a certain
symmetric monoidal functor. In fact, our \mbox{$2$-group} $\uT(Y)$ of
\S~\ref{sec:homom-varphi-colon} agrees with the category $\Div(X)$
considered in~\cite[Example~2.5]{MR2964607}.
\end{proof}

\paragraph{\emph{Synopsis of the following sections}}

In \S~\ref{sec:frac_ideals} we recall the basics of total rings of
fractions and fractional ideals.

In \S~\ref{sec:proof-that-shls} we work in the affine local situation
$y\in Y$ and we construct a canonical resolution of $\shM$ (in the
category of sheaves of abelian groups on $Y$).

In \S~\ref{sec:sheaf-shextc1shm-sho} we use this resolution to compute
$\shExt^1(\shM, \shO_Y^\times)$ in the affine local situation, and to
define the subsheaf $\shExtc^1(\shM, \shO_Y^\times)$. The local
properties that define it in fact define it for every $Y$, not
necessarily local.

In \S~\ref{sec:homom-varphi-colon} we define a homomorphism
$\varphi \colon \shExtc^1(\shM, \shO_Y^\times)\to \shLS_Y$ in the
affine local situation. In \S~\ref{sec:glob-varphi-non} we show that
the definition globalizes to all $Y$, not necessarily local.

In \S~\ref{sec:varphi-an-isom} we show that $\varphi \colon
\shExtc^1(\shM, \shO_Y^\times)\to \shLS_Y$ is injective; in
\S~\ref{sec-surjective} we show that it is surjective.

In the final two sections we complete the proof of the main theorem:
\S~\ref{sec:local-sect-prot} does it in the affine local situation,
and \S~\ref{sec:glob-log-struct} does it in the general case.

\subsection{Sheaf of total ring of fractions and invertible fractional
  ideals} 
\label{sec:frac_ideals}
For a scheme $X$, the sheaf of total rings of fractions $\shK_X$ is the sheafification of the presheaf
$U\mapsto S(U)^{-1} \Gamma(U,\shO_U)$
where $S(U)\subset \Gamma(U,\shO_U)$ is the subset of those non-zerodivisors of $\Gamma(U,\shO_U)$ that are also non-zerodivisors at every stalk of $\shO_U$.
If $U$ is affine then the stalk condition can be shown to be vacuous and $S(U)$ is just the set of non-zerodivisors in $\Gamma(U,\shO_U)$, see p.204 in \cite{MR570309}.
The case that concerns us is when $X$ is a reduced scheme whose set of irreducible components is locally finite, case (b) on p.205 in \cite{MR570309}. In this case,
$$\shK_X = j_*(\shO_{X|\operatorname{Ass}(X)})$$
where $j\colon {\operatorname{Ass}}(X)\to X$ is the inclusion of the set of points $x\in X$ for which the maximal ideal in $\shO_{X,x}$ is associated to zero.
The subsheaf of groups consisting of sections that are nowhere zero-divisors is denoted $\shK_X^\times \subset \shK_X$.
 \begin{exa}
   Consider $X=\Spec A$ where $A=k[x,y]/(xy)$, so $X$ has
   irreducible components $X_1=(y=0)$ and $X_2=(x=0)$.  For all
   $a,b\in k[x]$, $c,d \in k[y]$ with $b,d\neq 0$, consider the
   rational function $f=\frac{ax+cy}{bx+dy} \in K(A)$. Then
   $f_{\vert X_1}=\frac{a}b$ and $f_{\vert X_2}=\frac{c}d$, so
   $f=(\frac{a}b,\frac{c}d)$ under the natural isomorphism
   $K(A)=k(x) \times k(y)$.
 \end{exa}
Following \cite[\S\,19,\,20,\,21]{EGA_IVd},
 an \emph{invertible fractional ideal} on a scheme $X$ is a coherent $\shO_X$-submodule $\shI \subset \shK_X$ that is locally principal.
 In other words there is an affine cover of $X$ such that for all open subsets $U=\Spec A\subset X$ that belong to this cover, 
 \begin{equation}
 \label{local-fraction-ideal}
 \shI(U)=A\cdot \frac{s}{t}\subset \shK_X(U)
 \end{equation}
 where $s,t$ are both non-zerodivisors in $A$.
 The product of two fractional ideals inside $\shK_X$ is again a fractional ideal giving the set of fractional ideals the structure of an abelian group.
 The sheaf of groups of fractional ideals on $X$ is denoted by $\shI d.inv_X$. The sheaf of Cartier divisors on $X$ is, by definition, the sheaf of groups
 \[
 \shDiv_X = \shK_X^\times /\shO_X^\times.
 \]
 The natural homomorphism $\shDiv_X\to \shI d.inv_X$ is
 bijective~\cite[Proposition (21.2.6)]{EGA_IVd}.  We denote by
 $\shDiv^+_X\subset\shDiv_X$ the subsheaf of those divisors whose
 invertible sheaf has local forms \eqref{local-fraction-ideal}
 with $t=1$.  We set $\Div X=\Gamma(X,\shDiv_X)$ and
 $\Div^+\!X=\Gamma(X,\shDiv^+_X)$.

\subsection{The affine local situation}
\ \\*[3mm]  
\label{sec:proof-that-shls}

\begin{minipage}[c]{0.6\textwidth}
  \begin{setup}
  \label{setup:affine-local-situ}
   In this section we work with
  the following setup, which we refer to as the \emph{affine local
    situation}:
  \begin{enumerate}[(a)]
  \item $Y=\coprod_{\tau \in T} Y_\tau^\star$ is an affine viable \ggtc{} space;
  \item $Y$ has a unique smallest stratum $Y_\eta^\star = Y_\eta$,
    which is necessarily closed. 
  \item We write $M=M_\eta$, $\Sigma =\Sigma_\eta$, etc.
  \end{enumerate}
\end{setup}
  
The purpose of this section is to prove
Lemma~\ref{lem:resolution_of_ghost}, giving a canonical
resolution of the quotient sheaf
\[
  \shM = \shP^{\text{gp}}/\one\,.
\]
Recall $\shM$ is so defined that, for a cone $\tau \in \Sigma$ the
stalk $M_\tau$ is the quotient $M/\langle \tau
\rangle$.
\\[-2mm]
\end{minipage}
\hfill
\begin{minipage}[t]{0.35\textwidth}
  \captionsetup{width=\textwidth}
\qquad\ \includegraphics[width=.65\textwidth]{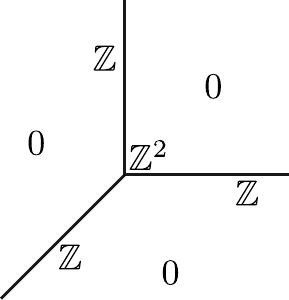}
\captionof{figure}{The stalks of the sheaf $\shM=\shP^{\text{gp}}/\one$ at various points of $T$ for the normal crossing surface $xyz=0$.}
\end{minipage}

\begin{notation}
  \label{nota:extensions-by-zero}
  For all $\tau \in \Sigma$, denote by $\underline{\langle \tau \rangle}_{U_\tau}$ the constant sheaf on
$U_\tau$ with group $\langle \tau \rangle$. Denoting by
$j_\tau \colon U_\tau \hookrightarrow Y$ the inclusion, we write:
  \[
L_!(\tau)=j_{\tau\, !} \underline{\langle \tau \rangle}_{U_\tau}
\]
where $j_{\tau\, !}$ is the extension by zero. More explicitly, if
$U\subset Y$ is a connected open subset, we have
\[
  L_!(\tau) (U)=
  \begin{cases}
    \langle\tau \rangle \quad &\text{if $U\subset U_\tau$}, \\
    0        \quad & \text{otherwise.}
  \end{cases}
\]
\end{notation}

\begin{construction}
  \label{con:resolution}
Recall that for an open embedding $j\colon U\to U'$, we have $j^*=j^!$
and hence for a sheaf $\shF$ on $U'$, we have an adjunction morphism
$j_!j^*\shF\to \shF$.  With that in mind, and using that, for
$\tau_1< \tau_2$, the open embedding
$j_{\tau_2}\colon U_{\tau_2}\to Y$ factors through
$j_{\tau_1}\colon U_{\tau_1}\to Y$, we have a natural map
$$ d_{\tau_1\tau_2}\colon j_{\tau_1\, !} \underline{\langle \tau_1 \rangle}_{U_{\tau_1}} \to j_{\tau_2\, !} \underline{\langle \tau_2 \rangle}_{U_{\tau_2}} $$
which at a stalk of a point $p\in Y$ is either the natural inclusion
$\langle \tau_1 \rangle\subset \langle \tau_2 \rangle$ or the zero
map, depending on whether $p\in U_{\tau_2}$ or not.  We define a
homomorphism
$$\delta_i\colon \bigoplus_{\tau_0\lneq ...\lneq\tau_i} L_!(\tau_i)\to \bigoplus_{\tau_0\lneq ...\lneq\tau_{i-1}} L_!(\tau_{i-1})  $$
via $\delta_i(\sum_{\tau_0\lneq ...\lneq\tau_i} a_{\tau_0\lneq ...\lneq\tau_i})=\sum_{\tau_0\lneq ...\lneq\tau_i} \delta_i(a_{\tau_0\lneq ...\lneq\tau_i})$ and a component of
$\delta_i(a_{\tau_0\lneq ...\lneq\tau_i})$ is trivial except for 
$$\delta_i(a_{\tau_0\lneq ...\lneq\tau_i})_{{\tau_0\lneq...\widehat\tau_j...\lneq\tau_{i}}} =(-1)^j a_{\tau_0\lneq ...\lneq\tau_i}$$
where $\tau_0\lneq...\widehat\tau_j...\lneq\tau_{i}$ refers to the result of removing $\tau_j$ from the sequence $\tau_0\lneq...\lneq\tau_{i}$ for $0\le j<i$; and for
\[
  \delta_i(a_{\tau_0\lneq
    ...\lneq\tau_i})_{{\tau_0\lneq...\lneq\tau_{i-1}}} =(-1)^i
  d_{\tau_{i-1}\tau_i}(a_{\tau_0\lneq...\lneq\tau_{i}})\,.
  \]
In particular,
\begin{equation}
\delta_1\left(\sum_{\tau_0\lneq \tau_1} a_{\tau_0\lneq
    \tau_1}\right)_{\tau}=  \sum_{\tau_0\lneq \tau}
a_{\tau_0\lneq\tau} - \sum_{\tau\lneq \tau_1}
d_{\tau\tau_1}(a_{\tau\lneq\tau_{1}})  \,.
\label{eq-delta1}
\end{equation}
\end{construction}

\begin{lem}
  \label{lem:resolution_of_ghost}  Let $\uM$ denote the constant sheaf on $Y$ with group $M$.
  Construction~\ref{con:resolution} gives a resolution of $\shM$ by sheaves on $Y$:
\[
\cdots \to  \bigoplus_{\tau_0\leq \tau_1} L_!(\tau_1) \stackrel{\delta_1}\longrightarrow
  \bigoplus_{\tau \in \Sigma} L_!(\tau)
  \stackrel{\delta_0}\longrightarrow \uM \to \shM \to 0 \,.
\] 
\end{lem}

\begin{proof} Exactness is to be checked at the stalk level, so let us
  fix a point $p\in Y$.  The complex of stalks only depends on the
  stratum $Y_\tau^\star$ that contains $p$. We have $p\in U_{\tau'}$
  if an only if $\tau'<\tau$.  In particular, $L_!(\tau')_p=0$ unless
  $\tau'<\tau$.  The complex has an increasing filtration $F_k$ by
  subcomplexes given by requiring the maximum for the dimension of the
  cones in the the chain $\tau_0\lneq ...\lneq\tau_i$ to be at most $k$, that
  is, $\dim\tau_i\le k$.  To show the exactness of the original sequence,
  it suffices to show the exactness of the graded quotients with
  respect to the filtration.  The graded quotient decomposes as
  $F_k / F_{k-1} = \bigoplus_{\dim\tau=k} C_{\tau}$ and $C_{\tau}$ is
  the complex that result from applying $L_!(\tau)\otimes\cdot$ to the
  complex
  $$ ...\to\bigoplus_{\tau_0\lneq ...\lneq\tau_i=\tau}\ZZ\to...   \to\bigoplus_{\tau_0\lneq\tau_1=\tau} \to \ZZ$$
  with differential similar to $\delta_i$ as before. This complex is
  exact, as it can be identified with the augmented chain complex for
  the homology of a contractible space.
\end{proof}
\noindent\vspace{2mm}
\begin{minipage}[c]{0.6\textwidth}
\begin{dfn}
  \label{dfn:relations}
  The \emph{relation sheaf} is the sheaf $\shR$ on $Y$ defined by the
  exact sequence:
\begin{equation} \label{relation-seq}
0 \to \shR \to \uM \to \shM \to 0 \,.
\end{equation}
\end{dfn}
\subsection{The sheaf $\shExtc^1(\shM, \shO_Y^\times)$}
\label{sec:sheaf-shextc1shm-sho}

 \begin{setup}
  \label{setup:affine-local-situ-2}
   We continue with the \emph{affine local
    situation} of  Setup~\ref{setup:affine-local-situ}.
\end{setup}

In this section we define a subsheaf $\shExtc^1(\shM, \shO_Y^\times)
\subset \shExt^1(\shM, \shO_Y^\times)$, which we call the sheaf of
regular extensions, consisting of sections with a prescribed asymptotic behavior towards the Zariski
closure $\overline{U}_\tau$ of $U_\tau$ in $Y$. Before we can state
this definition, we need to discuss some preliminaries.


\end{minipage}\hspace{0.05\textwidth}
\begin{minipage}[c]{0.35\textwidth}
  \captionsetup{width=\textwidth}
\qquad\ \includegraphics[width=.65\textwidth]{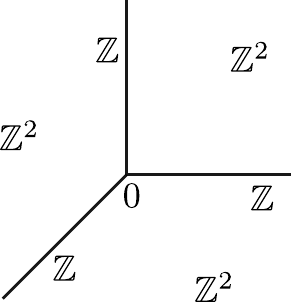}
\captionof{figure}{The stalks of the relation sheaf $\shR$ at various points of $T$ for the normal crossing surface $xyz=0$.}
\end{minipage}

\begin{lem}
  \label{lem:concrete_description}
For all cones $\tau^\prime \leq \tau$ in $\Sigma$, denote by $\rho^{\tau^\prime}_{\tau}\colon \shO_Y^\times(U_{\tau^\prime})\to \shO_Y^\times(U_{\tau})$ the restriction homomorphism.

We have the following concrete description of the relation sheaf: for every open $U\subset Y$
\begin{equation} 
\label{eq1}
\begin{split}
  \Gamma(U,&\shHom (\shR, \shO_Y^\times)) = \left\{\left. h\in\shHom
      \left( \bigoplus_\tau L_!(\tau),\shO_Y^\times
      \right)(U)\,\right|\, h_{|\im\delta_1}=0\right\}=\\
  & =\left\{ \left(h_\tau\right)_{\tau
      \in \Sigma}
    \left| \begin{array}{l} h_\tau\colon\langle \tau \rangle \to \shO_Y(U\cap U_\tau)^\times \;\text{ is a group homomorphism}\\
             \text{and for every }\tau^\prime \leq \tau \text{ we have
             }
             \rho^{\tau^\prime}_{\tau} \circ h_{\tau^\prime} =
             h_{\tau | \langle \tau^\prime \rangle}
\end{array}
    \right.
    \right\}.
\end{split}
\end{equation}   
\end{lem}

\begin{proof}
  Straightforward from Lemma~\ref{lem:resolution_of_ghost} and the
  adjoint properties of the functor $j_!$. 
\end{proof}

\begin{dfn}
\label{dfn:local_divisor_system}  
  The \emph{local divisor system} is the system of group homomorphisms
  $\{D_{\eta,\tau} |\tau \in \Sigma \}$ where, for all $\tau \in \Sigma$,
  \[
    D_{\eta,\tau} \colon \langle\tau\rangle \to \Div \overline{U}_\tau
  \]
  is the group homomorphism of Definition~\ref{dfn:viability} and Notation~\ref{nota:viability}.
\end{dfn}

\begin{notation}
  \label{nota:local_divisor_system}
  Because in this section $\eta \in T$ is the unique smallest stratum,
  in this section we suppress the subscript ``$\eta$'' from the notation and
  simply write $D_\tau$ instead of $D_{\eta, \tau}$.
\end{notation}

\begin{lem}
\label{lem-Dtau}
  The local divisor system satisfies the conditions (A),(B),(C) given
  below.
  
\begin{figure}[h]
  \includegraphics[width=.6\textwidth]{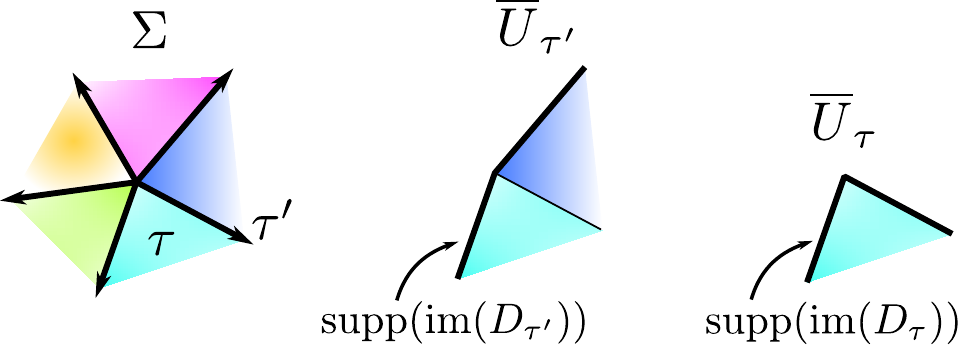}
  \caption{$\overline{U}_{\tau}\subset\overline{U}_{\tau'}$ for $\tau^\prime \leq \tau$}
    \label{fig-fan}
  \end{figure}

Note that if  $\tau^\prime \leq \tau$ then $U_{\tau^\prime} \supset
U_\tau$. In the statement of the conditions we denote by 
$\overline\rho^{\tau^\prime}_{\tau} \colon \Div \overline{U}_{\tau^\prime}
\to \Div \overline{U}_{\tau}$ the restriction of Cartier
divisors: this restriction is defined because $\overline{U}_\tau$ is a union of irreducible components
of $\overline{U}_{\tau^\prime}$. Figure~\ref{fig-fan} gives an
illustration.

The conditions are:
\begin{description}
\item[(A) Fan property] If $\tau^\prime \leq \tau$, then 
    \[
      \overline\rho^{\tau^\prime}_{\tau} \circ D_{\tau^\prime} =
      D_{\tau|\langle \tau^\prime\rangle}.\]

\item[(B) Positivity] $D_\tau (\langle \tau \rangle_+) \subset \Div^+ \overline{U}_\tau$.
  
\item[(C) Support property] The composition of $D_\tau$ with the
  restriction $\Div \overline{U}_\tau\to \Div U_\tau$ gives the
  trivial map. In particular, none of the divisors in
  $D_\tau (\langle \tau \rangle)$ contain the stratum $Y_\tau$ in
  their support; hence all of these divisors are restrictable to
  $Y_\tau$.
\end{description}
\end{lem}

\begin{proof} The statement is a
  straightforward consequence of the definition. By construction of
  $D$, see Definition~\ref{dfn:viability}, we may assume that
  $Y=\Spec k(\eta) [\Sigma_y]$ with $D_\tau(m)=\div z^m$, where the result is
  basically obvious.

  Part~(A) follows from observing that for $m\in\tau'$, by definition,
  $D_{\tau'}(m)=\div z^m\in \Div \overline{\Omega}_{\tau'}$ and so
  $\overline\rho^{\tau^\prime}_{\tau} \circ D_{\tau'}(m)$ is its
  restriction to $\overline{\Omega}_{\tau}$ which of course agrees
  with $D_{\tau}(m)=\div z^m\in \Div \overline{\Omega}_{\tau}$.

  Part~(B) is clear because $D_{\tau}(m)=\div z^m$ is a principal
  divisor of a regular function.

  Moreover, since $z^m$ is invertible
  on $\Omega_\tau$ for $m\in\tau$, its support is contained in
  $\overline{\Omega}_{\tau}\setminus \Omega_\tau$, so we deduce Part~(C).
\end{proof}

We are now ready to define the subsheaf
$\shExtc^1(\shM, \shO_Y^\times)$. As in Section~\ref{sec:frac_ideals},
we denote by $\shK_Y$ be the sheaf of total rings of fractions on $Y$
and a Cartier divisor $D$ on $Y$ gives the subsheaf
$\shO_Y (D)\subset \shK_Y$.

\begin{dfn}
  \label{dfn:regularity} 
  \begin{enumerate}[(1)]
  \item Let $U\subset Y$ be an open subset, and consider a section
    $(h_\tau)_{\tau \in \Sigma}\in \Gamma(U,\shHom (\shR,
    \shO_Y^\times))$.  Denote by
    $\widetilde{h}_\tau \colon \langle \tau \rangle \to \shK_Y
    (\overline{U}_\tau\cap U)$ the composition of
    $h_\tau\colon \langle \tau \rangle \to \shO_Y^\times(U_\tau\cap
    U)$ with the inclusion
    $\O_Y^\times (U_\tau\cap U)\hookrightarrow
    \shK_Y(\overline{U}_\tau\cap U)$.

  We say that $h$ is \emph{regular} if for all
  $\tau\in \Sigma$ and $m\in \tau$, $\widetilde{h}_\tau (m)$ is a generator
  of $\shO_{\overline{U}_\tau\cap U} \bigl( -D_\tau (m)\bigr) \subset
  \shK_{\overline{U}_\tau\cap U}$ as a $\shO_{\overline{U}_\tau\cap U}$-module. We denote by
\[
\shHom_c (\shR, \shO_Y^\times) \subset \shHom (\shR, \shO_Y^\times)
\]
the subsheaf of regular sections.
\item Consider the exact sequence:
  \begin{equation}
    \label{eq:exact_ext}
    \shHom(\uM, \O_Y^\times) \to \shHom (\shR,
    \shO_Y^\times) \to \shExt^1(\shM, \shO_Y^\times) \to 0 \,.    
  \end{equation}

  The \emph{sheaf of regular extensions}, denoted by $\shExtc^1 (\shM,
  \shO_Y^\times)$, is the image of $\shHom_c (\shR, \shO_Y^\times)$ in
  $\shExt^1(\shM, \shO_Y^\times)$.
\end{enumerate}
\end{dfn}


\subsection{A morphism $\varphi \colon \shExtc^1 (\shM,
  \shO_Y^\times) \to \shLS_Y$ in the affine local situation}
\label{sec:homom-varphi-colon}

\begin{setup}
  \label{setup:affine-local-situ-3}
    We continue with the \emph{affine local
    situation} of  Setup~\ref{setup:affine-local-situ}, and we assume
  in addition that
  \begin{enumerate}[(d)]
  \item For all $\tau \in \Sigma$, and all $m\in \tau$, $\shO(-D_\tau
    (m))$ is a trivial line bundle on $\overline{U}_\tau$.
  \end{enumerate}
\end{setup}

\begin{construction}
  \label{con:localphi}
  Let $U\subset Y$ be open and $h\in \Gamma(U,\shExtc^1 (\shM,
  \shO_Y^\times))$. Because of 
  assumption~(d) of Setup~\ref{setup:affine-local-situ-3}, $h$ lifts to
\[
  (h_\tau)_{\tau \in \Sigma}\in \shHom_c (\shR, \shO_Y^\times)(U) \,.
\]
Consider two maximal cones $\sigma_1, \sigma_2\in \Sigma$ meeting
along a common facet $\rho$.  We choose $e_2\in (\langle \rho \rangle +
\sigma_2)\cap M$ at
integral distance $1$ from $\rho$ and set\footnote{It is intentional that $\lambda(v_i)=\mu_{i-1}(-v_i)+\mu_i(v_i)$ from \eqref{abstr-lem-ai} has a different sign in the arguments when compared to 
$\varphi_\rho (h) = {\widetilde{h}_{\sigma_1}(e_2)}_{|Y_\rho} \otimes {\widetilde{h}_{\sigma_2} (-e_2)}_{|Y_\rho}$.}
\begin{equation}
\label{def-of-varphi}
\varphi_\rho (h) = {\widetilde{h}_{\sigma_1}(e_2)}_{|Y_\rho} \otimes {\widetilde{h}_{\sigma_2} (-e_2)}_{|Y_\rho}
\end{equation}
and the regularity of $h$ implies that $\varphi_\rho (h)$ is a
generator of the chart
\begin{equation}
\lambda_\rho (e_2)= {\shO_{U_1}\bigl(D_{\sigma_1}
(-e_2)\bigr)}_{|Y_\rho}
\otimes {\shO_{U_2}\bigl(D_{\sigma_2}
(e_2)\bigr)}_{|Y_\rho} 
\label{eq-rep-Lrho}
\end{equation}
of the line bundle $\shL_{\rho | U}$. 
\end{construction}

\begin{lemma-definition}
  \label{lem-dfn:localphi}
  Let $U\subset Y$ be open and
  $h\in \Gamma(U, \shExtc^1 (\shM, \shO_Y^\times))$.
  \begin{enumerate}[(1)]
  \item  The section $\varphi_\rho(h) \in \Gamma(Y_\rho \cap U,\shL_\rho)$ of
    Construction~\ref{con:localphi} does not depend on the choice of
    the lift $(h_\tau)_{\tau \in \Sigma}\in \shHom_c (\shR,
  \shO_Y^\times)(U)$, nor on the choice of $e_2$. Thus,
  Construction~\ref{con:localphi} provides for all $\rho \in T^{[1]}$
  a morphism of sheaves
  \[
    \varphi_\rho \colon \shExtc^1 (\shM, \shO_Y^\times) \to \shL_\rho \,.
  \]
\item Assembling the morphisms of Part~(1) gives a morphism
  $\oplus_{\rho \in T^{[1]}} \varphi_\rho
  \colon \shExtc^1 (\shM, \shO_Y^\times) \to \oplus_{\rho \in
    T^{[1]}}\shL_\rho $. The image of this morphism lies in
  $\shLS_Y^\times$.
  \end{enumerate}
The two Parts show that Construction~\ref{con:localphi} provides a
morphism $\varphi \colon \shExtc^1 (\shM, \shO_Y^\times) \to
\shLS_Y^\times$. 
\end{lemma-definition}

\begin{proof}[Beginning of the proof of Lemma-Definition~\ref{lem-dfn:localphi}]

We show that $\varphi (h)_\rho$ does not depend on the choice of lift
$(h_\tau)_{\tau \in \Sigma}\in \shHom_c (\shR, \shO_Y^\times)(U)$
and also that it does not depend on the choice of $e_2$.  Another
choice of lift differs from the chosen one by an element
\[
  u\in \shHom (\uM , \shO_Y^\times)(U) = \Hom_{\text{groups}}(M, \shO_Y(U)^\times)\,,
\]
so it can be written as $(uh_\tau)_{\tau \in \Sigma}$ and then
\[
  \varphi (uh)_\rho ={\bigl(u(e_2)
      \widetilde{h}_{\sigma_1}(e_2)\bigr)}_{|Y_\rho} \otimes
  {\bigl(u(e_2)^{-1}\widetilde{h}_{\sigma_2} (-e_2)\bigr)}_{|Y_\rho}
\] 
agrees with $\varphi (h)_\rho$.

A different choice $e'_2$ of $e_2$ differs by
$r=e_2'-e_2\in \langle\rho\rangle$ and leads to a different chart
$\lambda_\rho(e_2^\prime)$ of $\shL_\rho$ which is isomorphic to the
one in \eqref{eq-rep-Lrho} via the isomorphism $\psi_r$ from
Lemma~\ref{lem-no-choice}.  It is straightforward to see that the
assignment of the section $\varphi (h)_\rho$ to the tuple
$(h_\tau)_{\tau \in \Sigma}$ is compatible with $\psi_r$.

We have thus constructed a morphisms of sheaves of sets
$\oplus_{\rho} \varphi_\rho \colon \shExtc^1 (\shM, \shO_Y^\times) \to
\bigoplus_\rho \shL_\rho$.  We next explain why its image is contained
in $\shLS^\times_Y$, i.e., why its elements are nowhere vanishing and
satisfy the joint condition. That they are nowhere vanishing follows
immediately from the regularity property.

Before showing that the joint condition holds we need to discuss some
preliminaries.
\end{proof}

  \begin{dfn}
    \label{dfn:TX}
    Let $X$ be a space. The \emph{\mbox{$2$-group} $\uT (X)$ of trivialized line
      bundles} on $X$ is the strictly commutative \mbox{$2$-group}
    where: 
\begin{itemize}
\item Objects of $\uT (X)$ are trivialized line bundles on $X$, that
  is, pairs $(\shL, s)$ where $\shL$ is a
(trivial) line bundle on $X$, and $s\colon \shO_X\to \shL$ an
isomorphism;
\item A morphism from $(\shL_1, s_1)$ to $(\shL_2,s_2)$ is an
  isomorphisms $s\colon \shL_1\to \shL_2$ such that $s_2=s\circ s_1$.
\end{itemize}
The monoidal structure in $\uT (X)$ is given by tensor product
\[
(\shL_1,s_1)(\shL_2, s_2)=(\shL_1\otimes \shL_2, s_1\otimes s_2)
\]
and the distinguished neutral element is the trivial line bundle
$\shO_X$ with the unit trivializing section $1$. For $(\shL,s)$ we
have the inverse $(\shL^\star,(s^\star)^{-1})$ where
$\shL^\star=\shHom(\shL,\shO_X)$ denotes the dual and $s^\star$ the dual map.  For any pair
$(\shL_1, s_1),(\shL_2,s_2)\in\uT (X)$, there is a unique morphism
$(\shL_1, s_1)\to(\shL_2,s_2)$ in $\uT (X)$, so in particular, we have
unique morphisms
\[
(\shL_1\otimes \shL_2, s_1\otimes s_2)\to (\shL_2\otimes \shL_1,  s_2\otimes s_1),
\]
\[
(\shL, s)(\shL^\star, s^{-1})\to(\shO_X, 1)
  \]
  and $\uT (X)$ is thus strictly commutative.
  \end{dfn}

\begin{dfn}
  \label{dfn:local-line-bundle}
  Fix a codimension two cone
  $\omega\in \Sigma$ and denote by
  $\omega^{-1}\Sigma$ the localized fan.
  The \emph{local line bundle system} associated to a regular element
  $(h_\tau)_{\tau \in \Sigma}\in \shHom_c (\shR,
\shO_Y^\times)(U)$ is the continuous piecewise linear $2$-homomorphism (recall Definition~\ref{dfn:contPL2homo})
\[
  h^\omega\colon M \to \uT (Y_\omega \cap U)
\]
 such that if $\sigma \in \omega^{-1}\Sigma$ is a maximal cone and $m\in
 \sigma$, $h^\omega (m) =
 (\shO(-D_\sigma (m)),\widetilde{h}_\sigma )_{|Y_\omega\cap U}$.
 Regularity of $(h_\tau)_{\tau \in \Sigma}$ makes $h^\omega$
 well-defined.
\end{dfn}

\begin{proof}[End of the proof of Lemma-Definition~\ref{lem-dfn:localphi}]
  To see that $\varphi$ is well-defined as a morphism to $\shLS_Y$, we
  need to argue that elements in the image of $\varphi$ satisfy the
  joint condition.  We have defined the local line bundle system
  associated to a regular element
  $(h_\tau)_{\tau \in \Sigma}\in \shHom_c (\shR, \shO_Y^\times)(U)$ on
  an open $U\subset Y$. The abstract joint condition
  Lemma~\ref{lem:abstract_joint_conditon} applied to $h^\omega$ for
  every codimension two cell $\omega$ implies that the image of
  $\varphi$ is indeed contained in $\shLS_Y(U)$.
\end{proof}

\subsection{Globalizing the construction of $\varphi$}
\label{sec:glob-varphi-non}

\begin{lem}
  \label{lem:affine-cover}
  Let $Y$ be a \ggtc{} space. Then $Y$ has a cover by affines that
  satisfy properties (a)--(d) of Setup~\ref{setup:affine-local-situ}
  and Setup~\ref{setup:affine-local-situ-3}.
\end{lem}

\begin{proof}
  The proof is straightforward. To satisfy (d), we need to show that
  every $y\in Y$ has an affine neighbourhood $y\in U$ such that (d)
  holds on $U$. If $y\in Y_\eta^\star$, we may restrict to
  $Y=U_\eta$. We want a neighbourhood $y\in U$ such that for all $\tau
  \in \Sigma_\eta$ and all $m\in \tau \cap M_\eta$, the line bundles
  $\shO(-D_{\eta,\tau}(m))$ are trivial on $U$. This is easy to
  achieve based on the facts that: $T$ is finite, and all monoids
  $\tau \cap M_\eta$ are finitely generated.
\end{proof}

The goal of this section is to show that for all \ggtc{} spaces $Y$
the morphisms of sheaves defined in Lemma-Definition~\ref{lem-dfn:localphi}
in the local situation automatically glue to define a morphism of
sheaves $\varphi \colon \shExtc^1 (\shM, \shO_Y^\times)\to
\shLS_Y^\times$. The key result that makes everything work is the following:

\begin{lem}
  \label{lem:localizing} Let $Y$ be an affine \ggtc{} space satisfying
  conditions (a)--(d) of Setup Setup~\ref{setup:affine-local-situ}
  and Setup~\ref{setup:affine-local-situ-3} and $\eta\in T$ unique
  smallest stratum. Write $M=M_\eta$, $\Sigma = \Sigma_\eta$, etc.
  We identify the poset of strata with the poset of cones of $\Sigma$;
  note that, under this identification, $\eta$ corresponds to the cone
  $\{0\}\in \Sigma$. 

  Consider now an affine open $Y^\prime \subset Y$, also satisfying
  conditions (a)--(d). In particular, $Y^\prime$ has a smallest
  stratum $\eta^\prime \in \Sigma$ (and we allow the possibility
  $\eta^\prime =\{0\}$). We write
  \[
    M^\prime = M_{\eta^\prime}=M/\langle \eta^\prime\rangle, \; \pi
    \colon M \to M^\prime \;\text{the projection}, \quad
      \Sigma^\prime = \Sigma_{\eta^\prime}=\Sigma/\eta^\prime, \quad
      \text{etc.}
    \]
  Note that the cones of $\Sigma^\prime$ (a.k.a. the strata of
  $Y^\prime$) are the projections of the cones $\tau \in \Sigma$ such
  that $\eta^\prime \subset \tau$.

  In the notation of Definition~\ref{dfn:local_divisor_system}, we
  have:\footnote{In the display,
    $U_\tau$ is the star of $\tau$ in $Y$, and 
    $\overline{U}_\tau$ is the Zariski closure in
    $Y$.}
   \begin{description}
\item[(D) Sheaf property] 
For all $\eta^\prime \subset \tau$ and all $m\in\tau\cap M$,
\[
  D_{\eta,\tau}(m)_{|Y^\prime \cap \overline{U}_\tau} = D_{\eta^\prime,
    \pi(\tau)} (\pi(m)) \,.
\]
 \end{description}
\end{lem}

\begin{proof} As for the proof of Lemma~\ref{lem-Dtau}, the statement is a
  straightforward consequence of the definition. By construction of
  $D$, see Definition~\ref{dfn:viability}, we may assume that
  $Y=\Spec k(\eta) [\Sigma_y]$ with $D_\tau(m)=\div z^m$, where the
  result it is basically obvious.  
\end{proof}

\begin{lem}
  \label{mapvarphigeneral}
   For all \ggtc{} spaces $Y$
 the morphisms of sheaves defined in
 Lemma-Definition~\ref{lem-dfn:localphi} glue to define a morphism of
 sheaves
 \[
   \varphi \colon \shExtc^1 (\shM, \shO_Y^\times)\to
   \shLS_Y^\times \,.
 \]
\end{lem}

\begin{proof}
  By Lemma~\ref{lem:affine-cover}, it is enough to consider the
  situation $Y^\prime \subset Y$ of Lemma~\ref{lem:localizing}.

  We want to check that the respective definitions of $\varphi$
  resulting from using either $Y$ or $Y'$ agree on the smaller open
  set $Y^\prime$. This is, basically, an entirely straightforward exercise on
  unpacking the definitions, but we will spell it out in some detail. 

  The issue is that we are working with two relation sheaves,
  defined and related by the following commutative diagram of sheaves
  on $Y^\prime$ with exact rows and
  columns, where the notation is self-explanatory:
  \[
    \xymatrix{ &0 &0 & & \\
0\ar[r] &\shR^\prime\ar[r] \ar[u]&\uM^\prime\ar[r]\ar[u] &\shM_{|Y^\prime}\ar[r] & 0 \\
0\ar[r] &\shR_{|Y^\prime}\ar[r]\ar[u] &\uM_{|Y^\prime}\ar[r]\ar[u] & \shM_{|Y^\prime}\ar[r]\ar@{=}[u]& 0\\
&\underline{\langle\eta^\prime\rangle}\ar@{=}[r]\ar[u] &\underline{\langle\eta^\prime\rangle} \ar[u]& & \\
& 0\ar[u]&0\ar[u] & & }\]
 The two relation sheaves lead to two computations of the extension sheaf,
  summarised in the following commutative diagram of sheaves on
  $Y^\prime$ with exact rows and
  columns, where the notation is self-explanatory:
\[
  \xymatrix{0\ar[d] &0\ar[d] & & \\
    \shHom(\uM^\prime, \O_{Y^\prime}^\times )\ar[r] \ar[d]
      &\shHom(\shR^\prime, \O_{Y^\prime}^\times) \ar[r] \ar[d]^\chi &
      \shExt^1(\shM_{|Y^\prime}, \shO_{Y^\prime}^\times)
      \ar[r]\ar@{=}[d] & 0\\
   \shHom(\uM_{|Y^\prime}, \O_{Y^\prime}^\times) \ar[r] \ar[d]& \shHom (\shR_{|Y^\prime}, \shO_{Y^\prime}^\times)
   \ar[r] \ar[d]& \shExt^1(\shM_{|Y^\prime}, \shO_{Y^\prime}^\times) \ar[r]
   & 0\\
  \shHom (\underline{\langle\eta^\prime\rangle}, \shO^\times_{Y^\prime})\ar@{=}[r]& \shHom (\underline{\langle\eta^\prime\rangle}, \shO^\times_{Y^\prime}) &  &}
\]
 Let $\rho \in \Sigma^\prime$ be a submaximal cone at which maximal
 cones $\sigma_1^\prime, \sigma_2^\prime$ are incident. 
 Denote by $\varphi^\prime_\rho \colon \shExt^1_c(\shM_{|Y^\prime},
 \shO_{Y^\prime}^\times)\to \shL_\rho$ the morphism of sheaves obtained by
 construction~\ref{con:localphi} on the space $Y^\prime$, and by
 $\varphi_\rho  \colon \shExt^1_c(\shM_{|Y^\prime},
 \shO_{Y^\prime}^\times) \to \shL_\rho$ the morphism of sheaves obtained
 by Construction~\ref{con:localphi} on the space $Y$. We
 want to show that $\varphi^\prime = \varphi$.

 Consider a regular extension $h\in \shExt^1_c(\shM_{|Y^\prime},
 \shO_{Y^\prime}^\times)$. In order to compute $\varphi^\prime (h)$,
 we need to choose a lift in the first row
 \[
(h_{\tau^\prime})_{\tau^\prime\in \Sigma^\prime} \in \shHom_c(\shR^\prime, \O_{Y^\prime}^\times)
\]
 and follow Construction~\ref{con:localphi} on $Y^\prime$. Recall that we denote by
 $\pi\colon M \to M^\prime$ the natural projection: the cones
 $\tau^\prime\in \Sigma^\prime$ are the projections of the cones of
 $\Sigma$ that contain $\eta^\prime$. The result follows from
 the observation that
 \[
(h_{\tau^\prime}\circ \pi)_{\eta^\prime \subset \tau^\prime \in \Sigma} = \chi
\bigl( (h_{\tau^\prime})_{\tau^\prime\in \Sigma^\prime} \bigr)
   \]
   is a lift of $h$ in the second row; that because of
   Lemma~\ref{lem:localizing} it lies in
   $\shHom_c(\shR_{|Y^\prime}, \shO_{Y^\prime}^\times)$; and that
   therefore it is good to feed to construction~\ref{con:localphi} on
   $Y$. Unpacking the results leads to $\varphi
   (h)=\varphi^\prime (h)$.
\end{proof}

 \subsection{The morphism \protect$\varphi$ is injective}
 \label{sec:varphi-an-isom}

 \begin{lem}
   \label{lem:phi_injective}
   Let $Y$ be a \ggtc{} space. The morphism $\varphi$ of
   Lemma~\ref{mapvarphigeneral} is injective. 
 \end{lem}

 \begin{proof}
   In order to show the injectivity of $\varphi$, we may work on
   $Y$ affine satisfying properties (a)--(d) of
   Setup~\ref{setup:affine-local-situ} and
   Setup~\ref{setup:affine-local-situ-3}. 

   Assume that we are given two regular tuples
   $(h_\sigma)_{\sigma \in \Sigma}, (h'_\sigma)_{\sigma \in \Sigma}
   \in \shHom (\shR, \shO_Y^\times)$ that map to the same section of
   $\shLS^\times_Y$ under $\varphi$.

   For $\sigma\in\Sigma$ a maximal cone, the quotient
   $g=(\widetilde{h}_\sigma/\widetilde{h}'_\sigma)_{\sigma \in
     \Sigma}$ is a homomorphism
   $g_\sigma\colon M \to \shO (\overline U_\sigma)^\times$. To prove
   the injectivity of $Y$, we want to glue these maps to a
   homomorphism $g\colon M\to \shO (Y)^\times$, that is,
   $g \in \Gamma(Y,\shHom(\uM, \O_Y^\times))$.

For gluing along a slab $\rho$ with $\sigma_1,\sigma_2$ the maximal cones containing $\rho$, since $g_{\sigma_1},g_{\sigma_2}$ already agree on $\langle\rho\rangle$, we only need to consider $e_2\in\sigma_2$ at integral distance one from $\rho$ and we need to have $g_{\sigma_1}(e_2)_{|Y_\rho}=g_{\sigma_2}(e_2)_{|Y_\rho}$.
Assuming $\varphi(h)=\varphi(h')$ gives that $\varphi\big((h_\sigma)_{\sigma \in \Sigma}\big)_\rho$ and $\varphi\big((h'_\sigma)_{\sigma \in \Sigma}\big)_\rho$ define the same section of $\shL_\rho$.
In view of the definition \eqref{def-of-varphi}, we then have 
$$\widetilde{h}_{\sigma_1}(e_2)_{|Y_\rho}\otimes\widetilde{h}_{\sigma_2}(-e_2)_{|Y_\rho} = \widetilde{h}'_{\sigma_1}(e_2)_{|Y_\rho}\otimes\widetilde{h}'_{\sigma_2}(-e_2)_{|Y_\rho}.$$
and rearranging factors yields
\begin{equation}
\label{eq-g-glues}
g(e_2)_{|Y_\rho}=g(-e_2)^{-1}_{|Y_\rho}
\end{equation}
as desired.  We can use \eqref{eq-g-glues} to glue the homomorphisms
$g_{\sigma_1}\colon M\to \Gamma(\overline{U}_{\sigma_1}
\shO_Y^\times)$ and
$g_{\sigma_2} \colon M \to
\Gamma(\overline{U}_{\sigma_2},\shO_Y^\times)$ to a homomorphism $g
\colon M\to \Gamma\big( \overline{U}_{\sigma_1}\cup
\overline{U}_{\sigma_2} ,\shO_Y^\times\big)$.  Attaching similar
glueings along other slabs eventually yields a homomorphism
$g\colon M\to \Gamma(Y,\shO_Y^\times)$ with the property that on
each $\overline{U}_\tau$ it equals
$\widetilde{h}_\tau/\widetilde{h}'_\tau$. It follows from
Equation~\ref{eq:exact_ext} that the tuples
$(h_\sigma)$ and $(h'_\sigma)$ project to the same section in
$\shExtc^1 (\shM, \shO_Y^\times)$ and hence $\varphi$ is injective.
 \end{proof}

 \subsection{The morphism \protect$\varphi$ is surjective}
 \label{sec-surjective}

\begin{thm}
\label{thm-varphi-bijective}
  For every viable \ggtc{} space $Y$, the morphism
  \newline$\varphi\colon \shExtc^1 (\shM, \shO_Y^\times) \to
  \shL\shS^\times_Y $ is bijective.
\end{thm}

The statement is local hence we may work in the following:

\begin{setup}
  \label{setup:surjectivity} In the rest of this section we assume
  that $Y$ satisfies properties (a)--(d) of Setup~\ref{setup:affine-local-situ}
  and Setup~\ref{setup:affine-local-situ-3} and use freely the
  notation of Setup~\ref{setup:affine-local-situ}. 
\end{setup}

We begin with some preparations; the proof of the theorem is at the
end of the section.

\begin{notation}
  \label{nota:E}
  For all maximal cones $\sigma \in \Sigma$ we denote by
  \[
    E_\sigma\in N\otimes_\ZZ\PPic (Y_\sigma)=\underline{\twoHom}
    (M,\PPic Y_\sigma)
\]
the $2$-homomorphism such that for all $m\in M$, $E_\sigma
(m)=\shO_{Y_\sigma}(-D_\sigma(m))$. 
\end{notation}

 \begin{construction}
   \label{lem-slab-transfer-bundle}
   Consider two maximal cones $\sigma_1, \sigma_2\in \Sigma$ meeting
   along a slab $\rho$. Denote by $d_{\sigma_1\sigma_2}\in N$ the
   primitive normal to $\rho$ which evaluates non-negatively on
   $\sigma_2$.


   We will construct an isomorphism:
   \begin{equation}
     \label{eq:surj_one}
          \psi_{\sigma_1 \sigma_2} \colon \bigl[ E_{\sigma_1|Y_\rho} -
     d_{\sigma_1\sigma_2}\otimes 
     \shL_{\rho} \bigr] \overset{\cong}{\longrightarrow}  E_{\sigma_2|Y_\rho}
     \quad \text{in $N\otimes_\ZZ\PPic Y_\rho$}
   \end{equation}
 
   Note that this thing unpacks into a collection of isomorphisms, one for
   each $m\in M$:
   \[
     \psi_{\sigma_1 \sigma_2} (m) \colon E_{\sigma_1}(m)_{|Y_\rho} \otimes \shL_{\rho}^{-d_{\sigma_1\sigma_2}(m)}
     \overset{\cong}{\longrightarrow} E_{\sigma_2}(m)_{|Y_\rho}
   \]
   and such that $\psi_{\sigma_1 \sigma_2} (m)$ depends ``linearly'' on
   $m$.
   
For $m\in\langle\rho\rangle\cap M$ we have
$d_{\sigma_1\sigma_2}(m)=0$: in this case we define $\psi_{\sigma_1
  \sigma_2} (m)$ to be the obvious isomorphism obtained
by restricting to $Y_\rho$ the restrictable divisors $D_{\sigma_i}(m)$.

Now fix $e_2\in M$ with $d_{\sigma_1\sigma_2} (e_2)=1$: this
gives us the chart
\[
  \lambda(e_2) =
  {\shO_{Y_{\sigma_1}}\bigl(D_{\sigma_1}
  (-e_2)\bigr)}_{|Y_\rho} \otimes
{\shO_{Y_{\sigma_2}}\bigl(D_{\sigma_2} (e_2)\bigr)}_{|Y_\rho}
\]
for $\shL_\rho$ and we will carry out the construction by working in
this chart.

For $m=e_2$, we define:
\begin{multline*}
  \psi_{\sigma_1 \sigma_2} (e_2) \colon 
   \bigl({E_{\sigma_1}}_{|Y_\rho} - d_{\sigma_1\sigma_2}\otimes
    \lambda(e_2) \bigr)(e_2) = E_{\sigma_1}(e_e)_{|Y_\rho} \otimes \lambda(e_2)^{-1}=\\
  = \shO_{Y_{\sigma_1}}(-D_{\sigma_1}(e_2))_{|Y_\rho} \otimes
  \bigl({\shO_{Y_{\sigma_1}}\bigl(D_{\sigma_1} (-e_2)\bigr)}_{|Y_\rho}
    \otimes {\shO_{Y_{\sigma_2}}\bigl(D_{\sigma_2}
      (e_2)\bigr)}_{|Y_\rho}\bigr)^{-1}\\
   \overset{\cong}{\longrightarrow}
   \shO_{Y_{\sigma_2}}\bigl(-D_{\sigma_2}(e_2)\bigr) =
   E_{\sigma_2|Y_\rho}(e_2)  \,.
\end{multline*}

For general $m\in M$, there is a unique way to write
$m=m^\prime +k e_2$ with $m^\prime \in \langle \rho \rangle$ and $k\in
\ZZ$, and we define $\psi_{\sigma_1 \sigma_2} (m)=\psi_{\sigma_1 \sigma_2}
(m^\prime) \otimes \psi_{\sigma_1 \sigma_2} (e_2)^{\otimes k}$.

 We leave it to the reader to check that all the isomorphisms that one
 constructs similarly by working in different charts for $\shL_\rho$
 glue (see the proof of Corollary~\ref{cor:joint-compat} for a model
 discussion) and that the construction is linear in $m\in M$.
 \end{construction}
 
  



\begin{notation}
  \label{lem-slab-transfer-bundle-2}
  Assume Setup~\ref{setup:surjectivity} and Notation~\ref{nota:E}.

 Fix a nowhere vanishing section
 $f_\rho\in \Gamma(Y_\rho,\shL_\rho)$.

 Construction~\ref{lem-slab-transfer-bundle} provides an isomorphism
 \begin{equation}
   \label{eq:slab-transitions}
\psi_{\sigma_1 \sigma_2} (f_\rho) \colon E_{\sigma_1|Y_\rho}
\overset{\cong}{\longrightarrow} E_{\sigma_2|Y_\rho} \quad{in} \quad N\otimes
\PPic (Y_\rho)   
 \end{equation}
given, for all $m\in M$, as the composition 
\[
E_{\sigma_1}(m)_{|Y_\rho}\xlongrightarrow{(\ ) \times
  f_\rho^{-d_{\sigma_1\sigma_2}(m)}}
E_{\sigma_1}(m)_{|Y_\rho} \otimes \shL_\rho^{-d_{\sigma_1\sigma_2}(m)}
\xlongrightarrow{\psi_{\sigma_1\sigma_2}(m)} E_{\sigma_2}(m)_{|Y_\rho} \,.
\]
\end{notation}

\begin{dfn}
   \label{dfn:section_upgrade}
   Let $Y$ be a space, $M$ a lattice, $N=\Hom (M, \ZZ)$ and $E$ an
   object of the category $N\otimes_\ZZ \PPic (Y)$.

   A \emph{section upgrade} of $E$ is an object $\widehat{E}$ of
   $N\otimes_\ZZ \uT (Y)$ that maps to $E$ under the forgetful
   functor; in other words, $\widehat{E}=(E,e)$ where
   $e\colon N \otimes \shO_Y \to E$ is an isomorphism.
 \end{dfn}

\begin{pro}
  \label{PLTY}
  Let $Y$ be an affine viable \ggtc{} space as in
  Setup~\ref{setup:surjectivity}, $y\in Y$ a (closed) point, and let
  $E$ be as in Notation~\ref{nota:E}.

  For all $(f_\rho)_{\rho\in T^{[1]}}\in \Gamma(Y,\shLS^\times_{Y})$, possibly
  after shrinking $Y$ to a smaller affine neighbourhood of $y\in Y$,
  there exists $\widehat{H}\in N \otimes \uT(Y)$ such that:
  \begin{enumerate}[(1)]
  \item For every maximal cone $\sigma$,
    $\widehat{H}_{|Y_\sigma}=(E_\sigma,e_\sigma)$ is a section upgrade of
    $E_\sigma$; 
  \item For all pairs of maximal cones $\sigma_1,\sigma_2$ that meet in a slab
    $\rho$, denoting by \newline $\psi_{\sigma_1\sigma_2}(f_\rho) \colon
    E_{\sigma_1|Y_\rho} \overset{\cong}{\longrightarrow} E_{\sigma_2|Y_\rho}$
    the isomorphism of Notation~\ref{lem-slab-transfer-bundle-2}, we have
    \[
\psi_{\sigma_1\sigma_2}(f_\rho)(e_{\sigma_1|Y_\rho})=e_{\sigma_2|Y_\rho} \, .
      \]
  \end{enumerate}
\end{pro}

\begin{proof}
  The boundary complex of a polytope or polyhedral cone is shellable.
  Recall that this means that there is a shelling, that is an enumeration 
  \[
    \sigma_1,\sigma_2, \dots
  \]
  of its maximal cones such that for every $k\geq 1$ we have that
  $B_k=\left(\bigcup_{i=1}^k \sigma_i\right)\cap \sigma_{k+1} $ is a
  pure polyhedral complex of dimension $\dim M-1$ homeomorphic to
  either the cone over a ball or the cone over a sphere.
  Pick a shelling of $\Sigma$ and fix it for the rest of the proof.
  
  Choose a section upgrade $\widehat{H_1}= (E_{\sigma_1},
  e_{\sigma_1})$.  The cones $\sigma_1$ and $\sigma_2$ share exactly
   one submaximal face $\rho$ and Notation~\ref{lem-slab-transfer-bundle-2} implies that
   \[
     \bigl( E_{\sigma_2|Y_\rho},
     \psi_{\sigma_1\sigma_2}(f_\rho)(e_{\sigma_1|Y_\rho})\bigr)
   \]
   is a section upgrade of $E_{\sigma_2|Y_\rho}$. We extend it to a
   section upgrade $(E_{\sigma_2},e_{\sigma_2})$ of $E_{\sigma_2}$.
   (This extension might require us to shrink $Y$ to a smaller affine
   neighbourhood of $y$.) After this step, we have
   produced section upgrades
   $\widehat{E_{\sigma_1}}, \widehat{E_{\sigma_2}}$ of
   ${E}_{\sigma_1},{E}_{\sigma_2}$ that glue along $Y_\rho$ to give
   \[
     \widehat{H_2} \in N \otimes \uT(Y_{\sigma_1}\cup
     Y_{\sigma_2}) \,.
   \]

   Writing $Y_{k-1}=Y_{\sigma_1}\cup \cdots \cup Y_{\sigma_{k-1}}$, 
   assume by induction that we have constructed:
   \[
\widehat{H_{k-1}} \in N \otimes \uT(Y_{k-1})
\]
such that (1) and (2) hold for all cones $\sigma_i$ and for all pairs of cones
$\sigma_i, \sigma_j$ with $i, j\leq k-1$.

We will construct $\widehat{H_k}$ as follows. Below we write
$\widehat{H_{k-1}}=(H_{k-1},u_{k-1})$ where $H_{k-1} \in N \otimes
\PPic(Y_{k-1})$ and $\widehat{H_{k-1}}$ is a section upgrade. 
Let $i_1<\cdots <i_r\leq k-1$ be the
indices such that $\sigma_{i_m} \cap \sigma_k =\rho_m\in T^{[1]}$. We have that
\[
Y_{k-1} \cap Y_{\sigma_k} = Y_{\rho_1}\cup \dots \cup Y_{\rho_r} \,.
\]
The first step is to construct $H_k\in N \otimes \PPic (Y_k)$ by gluing
the $H_{k-1}$ to $E_{\sigma_k}$ by using the isomorphisms
  \[
    \psi_{\sigma_{i_m}\sigma_k}(f_{\rho_m}) \colon
    H_{k-1|Y_{\rho_m}}\overset{\cong}{\longrightarrow} E_{\sigma_k|Y_{\rho_m}} \,.
  \]

\paragraph{\textbf{Claim}} \emph{These isomorphisms agree on
the joints $Y_\omega$ where two of the $Y_\rho$ meet.}

\smallskip

To prove the claim, let us choose such a $\omega$ where two of the
$\rho$ meet and see what is going on on $Y_\omega$. Since
\[
  B_{k-1}=\rho_1\cup \dots \cup \rho_r\,
\]
shellability implies that exactly two $\rho$ meet at $\omega$, say
$\rho_a$ and $\rho_b$. We need to show that
\[
  \psi_{\sigma_{i_a}\sigma_k}(f_{\rho_a}) _{|Y_\omega} \colon
  H_{k-1|Y_\omega} \to E_{\sigma_k|Y_\omega}
  \quad \text{equals} \quad
  \psi_{\sigma_{i_b}\sigma_k}(f_{\rho_b})_{|Y_\omega} \colon
  H_{k-1|Y_\omega} \to E_{\sigma_k|Y_\omega}\,.
\]
Shellability implies that there are two increasing sequences of
indices:
\[
  a_1<\cdots < a_s \leq k-1
  \quad \text{and} \quad b_1<\cdots < b_t\leq k-1 \quad \text{where}
  \quad a_1=b_1, \; a_s =i_a,\; b_t=i_b
\]
such that
\[
\sigma_{a_1}, \sigma_{a_2},\dots, \sigma_{a_s},\sigma_k,
\sigma_{b_t},\dots, \sigma_{b_2}
\]
is a cyclic enumeration of all the maximal cones of $\Sigma$ incident
at $\omega$. In what follows we denote by $f_{a_i}$ the slab section on
$\sigma_{a_i}\cap \sigma_{a_{i+1}}$ and by $f_{b_i}$ the slab section on
$\sigma_{b_i}\cap \sigma_{b_{i+1}}$  
The claim follows from the identity:
\begin{equation}
  \label{eq:shelling-identity}
\psi_{\sigma_{a_s}\sigma_k}(f_{a_s}) \cdots \psi_{\sigma_{a_1}\sigma_{a_2}}(f_{a_1})
= \psi_{\sigma_{b_t}\sigma_k}(f_{b_t}) \cdots \psi_{\sigma_{b_1}\sigma_{b_2}}(f_{b_1})  
\end{equation}
on $Y_\omega$. Evaluating at $m\in M$ and using the description in
Notation~\ref{lem-slab-transfer-bundle-2},
identity~\ref{eq:shelling-identity} is the joint condition:
\[
f_{a_s}^{d_{\sigma_{a_s}\sigma_k}(m)} \cdots
f_{a_1}^{d_{\sigma_{a_1}\sigma_{a_2}}(m)}=
f_{b_t}^{d_{\sigma_{b_t}\sigma_k}(m)} \cdots
f_{b_1}^{d_{\sigma_{b_1}\sigma_{b_2}}(m)} \,.
\]

This proves the Claim, and the Claim readily implies the result. 
\end{proof}

\begin{proof}[Proof of Theorem~\ref{thm-varphi-bijective}]
  We have shown injectivity in \S~\ref{sec:varphi-an-isom}.  It
  suffices to show surjectivity of $\varphi$ at a stalk $y\in Y$. 
  Given
  $s=(f_\rho)_\rho\in \Gamma(Y,\shLS^\times_{Y})$, let
  $\widehat{H}\in N\otimes \uT(Y)$ denote the object given by
  Proposition~\ref{PLTY}. By the construction of $\widehat{H}$, for every cone
  $\tau\in\Sigma_y$, the restriction of $\widehat{H}\colon M\to \uT(Y)$ gives a
  group homomorphism
  \[
    h_{\tau}\colon \langle\tau\rangle\to  \shO (U_\tau)^\times
  \]  
Moreover, the regularity condition from
Definition~\ref{dfn:regularity} is satisfied by the construction of
$h$ because $m\in\langle\tau\rangle$ maps to a generator of
$\shO_{\overline{U}_{\tau}}(-D_\tau(m))$.  The collection of maps
$h=(h_{\tau})_{\tau\in \Sigma}$ is compatible in the sense that
$\rho_\tau^{\tau'}\circ h_{\tau'}=h_{\tau| {\langle\tau'\rangle}}$
whenever $\tau'\le\tau$ and so, by \eqref{eq1}, the collection $h$ of
$h_{\tau}$ constitutes a section of $\shHom_c(\shR,\shO_Y^\times)$ and
thus it gives a class in $\Gamma(Y,\shExtc^1 (\shM, \shO_Y^\times))$. By the
the key property of $\widehat{H}$ stated in Proposition~\ref{PLTY}(2), we have $\varphi(h)=s$.
\end{proof}

\subsection{Local sections of \protect$\shExt^1_c$ give log
  structures locally}
\label{sec:local-sect-prot}

\begin{pro}  
\label{pro:main_th_local_case}
Let $Y$ be an affine viable \ggtc{} space as in Setup~\ref{setup:surjectivity}.
Consider a regular extension of sheaves of groups
\[
  0\to \shO_Y^\times \to \foM \to \shM\to 0
\]
Then the fiber product $\foP:=\foM \times_{\shM} \shP$ comes equipped with
a monoid homomorphism $\alpha$ to $(\shO_Y,\times)$ that yields a
compatible log structure on $Y/k^\dagger$.
\end{pro}  
\begin{proof}
  Consider the relation sheaf sequence \eqref{relation-seq} (this is
  the same as the first line in \eqref{pushoutdiagram} below).
  By definition $\foM$ is in $\shExtc^1(\shM, \shO_Y^\times)$ if
  $\foM=\partial (h)$ for a regular $h\in \shHom_c (\shR, \shO_Y^\times)$,
  where $\partial$ denotes the boundary map
  \[
    \partial\colon\shHom(\shR,\shO_Y^\times)\to \shExt^1 (\shM,
    \shO_Y^\times) \,.
  \]
  In concrete terms, given $h\colon \shR\to \shO_Y^\times$, the
  extension $\partial h$ is constructed by push out: 
\begin{equation}
\label{pushoutdiagram} 
\begin{aligned}
\xymatrix{
0\ar[r]& \shR \ar[r]\ar_(.45)h[d]& \uM \ar[r]\ar[d]& \shM \ar[r]\ar@{=}[d]& 0\\
0 \ar[r]& \shO_Y^\times \ar[r] & \foM \ar[r]& \shM \ar[r]& 0
}\end{aligned}
\end{equation}
with exact rows and cocartesian squares, where
\begin{equation}
\label{def-log-str}
\foM = \shO_Y^\times\oplus_\shR \uM = (\shO_Y^\times\oplus \uM) /
\{(h(m),-m)\,|\,m\in\shR \} \, .
\end{equation}

Our $\foM$ arises in this way from a regular $h\in \shHom_c (\shR,
\shO_Y^\times)$. From the concrete description of
Lemma~\ref{lem:concrete_description}, $h=(h_\tau)_{\tau \in \Sigma}$ where
$h_\tau\colon \langle \tau \rangle \to \shO (U_\tau)^\times$ is a
group homomorphism and, for every $\tau^\prime \leq \tau$,
$\rho^{\tau^\prime}_\tau \circ h_{\tau^\prime} = h_\tau$.
Definition~\ref{dfn:regularity} states that $h$ is regular if, denoting by
\[
\iota_\tau \colon \shO_Y^\times (U_\tau) \hookrightarrow \shK (\overline{U}_\tau)
\]
the natural inclusion, we have that for all $\tau\in \Sigma$ and all $m\in
\langle \tau \rangle$
\[
  \widetilde{h}_\tau (m) = \iota_\tau \circ h_\tau (m) \quad \text{is a generator
    of}
  \quad \shO_{\overline{U}_\tau}\left(-D_\tau(m)\right) \,.
\]

We are now ready to define the log structure
$\alpha\colon \foP := \foM \times_{\shM} \shP \to\shO_Y$. For all
$\tau \in \Sigma$, we describe
\[
\alpha_\tau \colon \foP(U_\tau)=\foM (U_\tau)\times_{M_\tau} P_\tau
\to \shO(U_\tau) .
  \]

For all $\tau\in \Sigma$, denote by $\pi_\tau \colon M \to M_\tau =
M/\langle \tau \rangle$ and by $\psi_\tau \colon P_\tau \to M_\tau = P_\tau
/\mathbf{1}_\tau$ the natural projections. 
An element of $\foP(U_\tau)$ is a pair $((\xi, m),p)$ where $\xi \in
\shO (U_\tau)^\times$, $m\in M$, $p\in P_\tau$, and $\pi_\tau
(m)=\psi_\tau (p)$.

If $\psi_\tau (p) = 0$, that is, $p\in \mathbf{1}_\tau + P_\tau$, then we
set $\alpha_\tau \bigl( (\xi, m),p \bigr)=0$.

Otherwise, there is a smallest cone $\tau \leq \tau^\prime$ such that
$m\in\tau^\prime$ and $\pi_\tau (m)=\psi_\tau (p)$. Note that, in this
case, $\overline{U}_{\tau^\prime}\cap U_\tau \subset U_\tau$ is a union of
irreducible components. In this case, we set:
\[
  \alpha_\tau \bigl( (\xi, m),p \bigr) =
  \begin{cases}
    \widetilde{h}_{\tau^\prime}(m) \xi_{|\overline{U_{\tau^\prime}}
      \cap U_\tau}\; & \text{on $
      \overline{U}_{\tau^\prime}\cap U_\tau $;}\\ 
    0 \; & \text{on $U_\tau \setminus \overline{U}_{\tau^\prime}$ \,.}
  \end{cases}
  \]
  The key observation here is that, when $m\in \tau^\prime$, the divisor
  $D_{\tau^\prime} (m) \in \Div^+\overline{U}_{\tau^\prime}$ is
  effective; and hence the rational function
  $\widetilde{h}_{\tau^\prime}(m) \xi_{|\overline{U}_{\tau^\prime}\cap U_\tau}$ is
    in fact regular and it vanishes on the boundary, and hence it can
    be extended by zero to a regular function on all of $\overline U_\tau$.

    It is straightforward to check that $\alpha$ is well-defined; that
    it is a log structure; and that the global section
    $\bigl((1,0),\textbf{1}_\shP\bigr)$ defines a morphism of log structures
    $Y^\dagger \to (\Spec k)^\dagger$.
    \end{proof}
  
\subsection{Proof of
  Theorem~\ref{mainmaintheorem}}\label{sec:glob-log-struct} 

We want to generalize the constructions in the previous section from the affine situation to the general situation.
We begin by rephrasing Proposition 3.11 in \cite{MR2213573} as the following Lemma.
\begin{lem}
Let $X$ be a reduced Deligne--Mumford stack with two log structures\newline $\alpha,\alpha'\colon \foP \to \shO_X$ with identical monoid sheaf. We denote $\beta=\alpha'_{|(\alpha')^{-1}(\shO_X^\times)}$, so $\beta$ is an isomorphism onto $\shO_X^\times$.
If $\alpha\circ\beta^{-1}\colon \shO_X^\times\to \shO_X^\times$ is the identity then $\alpha=\alpha'$.
\label{lem-unique}
\end{lem}

\begin{proof} 
  Consider $\alpha,\alpha'$ at the generic point $\eta$ of a component
  of $X$,
  $\alpha_\eta,\alpha'_\eta\colon \foP_\eta \to \shO_{X,\eta}$.  Since
  $\shO_{X,\eta}$ is a field, given $m\in \foP_\eta$ either
  $\alpha_\eta(m)=0$ or $\alpha_\eta(m)$ is invertible, similarly for
  $\alpha'_\eta(m)$. It follows from the assumptions that
  $\alpha_\eta=\alpha'_\eta$. Since $X$ is reduced, for any open set
  $U$, $\shO_X(U)$ injects into $\bigoplus_{\eta\in U} \shO_{X,\eta}$
  where $\eta$ runs over the generic points of the components of $X$.
  The homomorphism $\alpha$ is thus determined by the homomorphisms $\alpha_\eta$ and
  thus $\alpha=\alpha'$.
\end{proof}

\begin{proof}[Proof of Theorem~\ref{mainmaintheorem}] 
  Let $Y$ be a viable \ggtc{} space.


  \smallskip

\textbf{Step~1} \emph{We construct a morphism of Zariski sheaves of
  sets on $Y$:
   \begin{equation} 
\label{psi}
    \psi\colon \LS \to \shExt^1_c(\shM, \shO_Y^\times) \,.
  \end{equation}}

\smallskip

We construct a set-theoretic function
$\LS(Y) \to \Gamma\bigl(Y,\shExtc^1(\shM, \shO_Y^\times)\bigr)$ that assembles to a sheaf
morphism in the input $Y$. Let $\alpha\colon \foP\to\shO_Y$ be a
compatible log structure. Using $\psi\colon \foP \to \shP$ from
Part~(1) of Definition~\ref{dfn:compatible}, we can write
$\foP=\foM \times_{\shM} \shP$ where $\foM=\foP/\one_{\foP}$.  All
relevant maps being bijective, we identify the projection of
$\alpha^{-1}(\shO_Y^\times)$ in $\foM$ with $\shO_Y^\times$ and obtain
an exact sequence as shown as the bottom row in
\eqref{pushoutdiagram}.  We need to show that its extension class lies
in $\shExtc^1 (\shM, \shO_Y^\times)$.  This is a local question, so we
may assume that $[\foM]=\partial(h)$ for some
$h\in\Gamma(Y,\shHom(\shR,\shO_Y^\times))$ so we have a diagram
like~\eqref{pushoutdiagram}. We may also assume that $Y$ is affine as in
Setup~\ref{setup:surjectivity}.
We want to show that $h=(h_\tau)_{\tau \in \Sigma}$ is regular. We
need to ``extract'' $h_\tau(m)$ from the datum of the log structure. Fix
$m\in M$ and let $\tau \in \Sigma$ be the smallest cone that contains
it. There is a unique
$\widehat{m}\in P_\tau\setminus (\mathbf{1}_\tau+P_\tau)$ that maps to
$m$ under the projection $P_\tau \to M_\tau=P_\tau /\mathbf{1}_\tau$,
and it is tautologically the case that, interpreting the pair
$((1,m),\widehat{m})$ as an element of
$\foP(U_\tau)=\foM(U_\tau)\times_{M_\tau}P_\tau$, 
\[
  h_\tau (m)=\alpha(\widehat{m}) 
  \in \shO (U_\tau) \,.
\]
By Part~(b) of Definition~\ref{dfn:compatible}, $\div
(h_\tau(m))=D_\tau(m)$. This holds for all $m\in M$, hence $h$ is
regular.

\smallskip

\textbf{Step~2} \emph{The morphism $\psi$ is injective.}

\smallskip

We need to show that if two
compatible log structures give rise to the same extensions,
then they are isomorphic as compatible log structures.
If the extensions are the same then the two log structures must have
isomorphic monoid sheaves compatibly with the respective embeddings of
$\shO_Y^\times$ and the result follows from Lemma~\ref{lem-unique}. 

\smallskip

\textbf{Step~3} \emph{The morphism $\psi$ is surjective.}

\smallskip


Start from a class in
$\Gamma(Y,\shExtc^1 (\shM, \shO_Y^\times))\subset \Gamma(Y,\shExt^1
(\shM, \shO_Y^\times))$. Since $\shM$ is supported on the union of
slabs, which is of codimension one, we have that
$\shHom(\shM,\shO_Y^\times)=0$. The local-to-global Ext spectral
sequence then implies that
$\Gamma(Y,\shExt^1 (\shM, \shO_Y^\times))=\Ext^1(\shM,
\shO_Y^\times)$, so we obtain a sheaf $\foM$ that sits in the middle
of an exact sequence as in Proposition~\ref{pro:main_th_local_case}.
We then produce the monoid sheaf $\foP := \foM \times_{\shM} \shP$
which locally comes with a homomorphism to $(\shO_Y,\cdot)$ by
Proposition~\ref{pro:main_th_local_case}. These local maps glue to a
global one by the uniqueness of these maps according to
Lemma~\ref{lem-unique}.

\smallskip

\textbf{Step~4} \emph{End of proof of Theorem~\ref{mainmaintheorem}.}

\smallskip

We define the composition
\begin{equation}
    \label{morphism-r}
r=\varphi\circ\psi\colon \LS \to \shLS_Y^\times 
\end{equation}
of $\psi$ from \eqref{psi} with the bijection $\varphi$ from Theorem~\ref{thm-varphi-bijective}.
\end{proof}
  
  
\section{Examples}
\label{sec:examples}

\subsection{Two components}
\label{sec:two-components}

In this section, $Y$ consists of two smooth components $Y_1, Y_2$
meeting along a smooth irreducible divisor $D\subset Y_{i}$ ($i=1,2$)
all defined over $k$. First, we describe a sheaf of monoids
$\shP$ on $Y$ that allows us to equip $Y$ with the structure of a \ggtc{}
space; then, we study the log structures $\foP$ on $Y$ that have ghost
sheaf $\shP$, and finally we see how to keep track of a morphism to
the standard log point $\Spec k ^\dagger$. This is the most elementary
example of the theory but it is nevertheless quite rich and it is well
worth it to invest the time necessary to understand it
completely.\\[1mm]
\noindent
\begin{minipage}[c]{0.65\textwidth}
  The simplest source of examples of log structures on $Y$ are
  embeddings $i\colon Y \hookrightarrow X$ into a smooth scheme
  $X$. More generally, we can look at an embedding where locally
  analytically (or \'etale locally) $Y=(t=0) \subset X$ where
  $X=(xy+t^r=0)\subset \AA^3_{x,y,t}\times \AA^m$ for $m=\dim D$.

Given such an embedding one can form the \emph{divisorial log structure}
$\foP_{X,Y}=\shO_X\cap \shO_{X\setminus Y}^\times$ on $X$, and the log
structure $\foP_Y=i^\star \foP_{X,Y}$ on $Y$. Below we make log structures
on $Y$ of this \'etale local type without reference to an embedding. 
 Fix an integer $r>0$, consider the sublattice
\end{minipage}
\hfill
\begin{minipage}[t]{0.3\textwidth}
\captionsetup{width=\textwidth}
\vspace{-22mm}
\qquad\ \includegraphics[width=.65\textwidth]{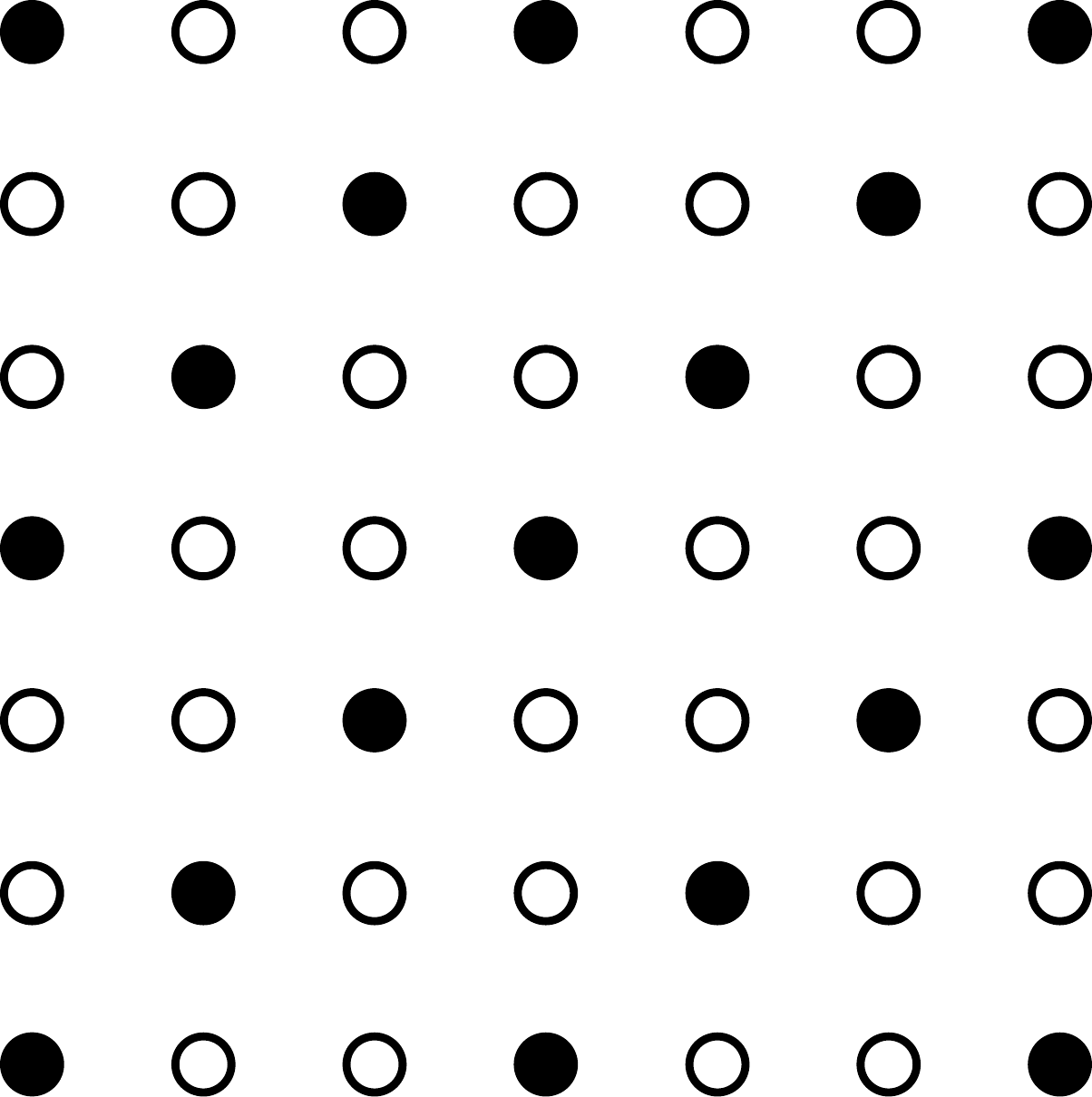}
\captionof{figure}{The sublattice $M\subset\ZZ^2$ for $r=3$.}
\label{fig-lattice}
\end{minipage}
\[
  M=\{(n_1,n_2)\mid n_1 \equiv n_2 \Mod{r}\}\subset \ZZ^2,
\]
illustrated in Figure~\ref{fig-lattice},
and let $P=\langle e_1,e_2\rangle_+\cap M$ be the monoid of lattice points in the positive
quadrant. This monoid is generated by $(r,0)$, $(0,r)$ and $(1,1)$ with the obvious relation.
We use $P$ to define a sheaf of monoids on $Y$ as
follows. 
For a connected Zariski open subset $U\subset Y$, we set
\[
  \shP (U) =
  \begin{cases}
    P \; & \text{if $U\cap D\neq \emptyset$,}\\
    P/\langle re_2 \rangle\cong \NN \; & \text{if $U\subset Y\setminus Y_2$,} \\
    P/\langle re_1 \rangle \cong \NN \; & \text{if $U\subset Y\setminus Y_1$.} \\ 
  \end{cases}
\]

Now $Y$ is a \ggtc{} space with ghost sheaf $\shP$ in a
natural way: if locally at the generic point $\eta \in D$ we have $Y=(xy=0)$ where
say $Y_1=(y=0)$ and $Y_2=(x=0)$, then we have
$k[P]=k[x,y,t]/(xy+t^r)$. In what follows, it is crucial to understand
that $x\in k[P]$ vanishes  with multiplicity $r$ along $Y_2$.

\begin{pro}
  \label{pro:log_2components}
  Let $r>0$ and $Y=Y_1+Y_2$ a toroidal crossing space as just
  described. To give a log structure $\foP$ on $Y$ is equivalent to
  giving line bundles $\shL_1, \shL_2,\shL$ on $Y$, homomorphisms
  $\alpha_1\colon \shL_1\to \shO_Y$, $\alpha_2\colon \shL_2\to \shO_Y$
  such that
\begin{align}
  \label{eq:example1}
 \alpha_{1|Y_2}\colon
\shL_{1|Y_2}\overset{\cong}{\longrightarrow}\shO_{Y_2}(-D)\subset
  \shO_{Y_2} & \quad \text{and} \quad \alpha_{1|Y_1} = 0, \\
   \alpha_{2|Y_1}\colon
\shL_{2|Y_1}\overset{\cong}{\longrightarrow}\shO_{Y_1}(-D)\subset
  \shO_{Y_1} & \quad \text{and} \quad \alpha_{2|Y_2} = 0;
\end{align}
and an identification
$s\colon \shL^{\otimes r} \overset{\cong}{\longrightarrow}
\shL_1\otimes \shL_2$.
\end{pro}

\begin{proof}[Sketch of proof] We construct a log structure from the
  data spelled out in the proposition.
  First of all, for $p=n_1(r,0)+q(1,1)\in P$ --- that is,
  $q\geq 0,n_1\in\ZZ$ satisfying $n_1+\frac{q}{r} \geq 0$ --- write
  \[
    \shL^p = \shL_1^{\otimes n_1} \otimes
    \shL^{\otimes q}.
  \]
  Note how this choice singles out a \mbox{$2$-homomorphism} from the monoid
  $P$ --- viewed as a category --- to the \mbox{$2$-group} $\PPic Y$.
  The log structure is the sheaf of monoids:
  \[
    \foP (U) =
    \begin{cases}
      \coprod_{p\in P} \shL^p (U)^\times \; & \text{if $U\cap D\neq
        \emptyset$,}\\
      \coprod_{q\in \NN} \shL^{\otimes q} (U )^\times\; & \text{if $U\subset Y\setminus
        Y_2$,} \\
     \coprod_{q\in \NN} \shL^{\otimes q} (U )^\times \; & \text{if $U\subset Y\setminus
        Y_1$.}
    \end{cases}
\]
 together with the obvious homomorphism $\alpha \colon \foP \to
 \shO_Y$. 
\end{proof}

We want to understand isomorphism classes of log structures on $Y$
over $\Spec k^\dagger$, where we map $\NN\to P$ by taking $1$ to
$\mathbf{1}=(1,1)$: it is surprising to see that this set has a
simpler description:

\begin{pro}
  \label{pro:2components_over_logpoint}
  Let $Y=Y_1+Y_2$ be a toroidal crossing space as above. The set of
  isomorphism classes of log structures on $Y$ over $\Spec k^\dagger$
  is the set of nowhere-vanishing sections of the sheaf
  \[
\shL_D=\left(N_{Y_1} D\right)\otimes \left(N_{Y_2}D\right)
  \]
\end{pro}

\begin{proof}[Sketch of proof] Suppose given a log structure $\foP$ on
  $Y$, making it into a log scheme $Y^\dagger$. In the notation of the
  previous proposition, a morphism $Y^\dagger \to \Spec k^\dagger$ is
  precisely the datum of a nowhere-vanishing global section $\sigma_0 \in H^0(Y,\shL)$. This
  also gives a section $\sigma =\sigma_0^r \in H^0(Y, \shL^{\otimes r})$ that we pass through the isomorphism $s$ to 
  trivialize the line bundle $\shL_1\otimes \shL_2$. So for example we have
\begin{align*}
  {\shL_1}_{|Y_2}&=\shO_{Y_2}(-D)\\
  {\shL_1}_{|Y_1}&\stackrel{s(\sigma)}{=}{\shL_2^\star}_{|Y_1}=\shO_{Y_1}(D)
\end{align*}
So $\shL_1$ is glued together from an isomorphism
\[
f\colon {\shO_{Y_2}(-D)}_{|D}=N_{Y_2}^\star D \overset{\cong}{\longrightarrow} N_{Y_1} D ={\shO_{Y_1}(D)}_{|D}
 \]
 or, equivalently, a nowhere-vanishing section of
 \[
\shLS_Y = (N_{Y_1} D) \otimes (N_{Y_2} D).
 \]
\end{proof}
In particular in order for the log structure over $\Spec k^\dagger$
to even exist, one must have that $\shLS_Y$ is a trivial line bundle.

\begin{rem}
  \begin{enumerate}[(1)]
  \item Perhaps surprisingly, the sheaf $\shLS_Y=\shL_D$ in the
    example, and thus the fact of the existence of a compatible log structure
    over the standard log point, does not depend on $r$. 
  \item It is a matter of convention whether
    $\shLS_Y=N_{Y_1}D\otimes N_{Y_2}D$ or its dual. Our convention is
    motivated by the fact that we want $\shLS_Y$ to have lots of
    sections when $N_{Y_1}D\otimes N_{Y_2}D$ has
    lots of sections. 
    These sections of $\shLS_Y$ are useful even when
    they vanish somewhere and we think of them as giving ``singular''
    log structures, as justified in Definition~\ref{dfn-gs} below. Furthermore, sections with vanishing loci in $\shLS_Y$ may arise from restrictions of divisorial log structures of a pair $(X,Y)$ to the divisor $Y$.
The convention also fits with the existence of an isomorphism $\shLS_Y\cong \shT^1_Y$ in the normal crossing case, \cite{MR4304077}. 
\label{signmatters}
\end{enumerate}
\end{rem}

\begin{dfn} 
\label{dfn-gs}
We adopt the following language from the Gross--Siebert program initiated in \cite{MR2213573}.
\begin{enumerate}
\item We say that a section of $\shLS_Y$ whose vanishing locus is a
  divisor $Z\subset D$ is a log structure on $Y$ \emph{singular along}
  $Z$.  A priori, the log structure is only defined on $Y\setminus
  Z$. A log structure on all of $Y$ can be produced by taking the
  direct image of the log structure on $Y\setminus Z$.  This direct
  image log structure fails to be coherent along $Z$.  Failure of
  coherence prevents the log scheme from satisfying the formal log
  smoothness lifting criterion as we explain in the simplest example
  below, so calling the log structure \emph{singular} along $Z$ is
  justified.  With regards to part \ref{signmatters} of the previous
  remark we view a log structure from a section of $\shLS_Y$ on an
  open set that extends with \emph{poles} in its complement as
  pathological.
\item In the above example, the sheaf $\shLS_Y$ was a line bundle on
  $D$, identical with the unique slab bundle.  More generally, as in
  Definition~\ref{def_shLS}, $\shLS_Y$ is defined as a subsheaf of the
  direct sum of slab bundles cut out by the joint condition.  We call
  a section $f_\rho$ of a slab bundle $\shL_\rho$ a \emph{slab
    section}, typically studied in a local trivialization of the
  bundle where we may also call it a \emph{slab function}.
\item Each slab bundle comes with a positive integer like the integer $r$ above which appears in its frame from the relation $xy=z^r$. We refer to this integer as the \emph{kink} of the slab.
\end{enumerate}
\end{dfn} 

\subsection{The easiest singular log structure}
\label{sec:easiest-singular-log}

It pays off to study the easiest singular log structure that there is
and understand it completely.

Consider $\AA^3$ with coordinates $x,y,u$; we take $X$ to be the surface
\[
  X= \bigl(xy=0\bigr) \subset \AA^3.
\]
Note that $X$ consists of two components
\[
X_1=\AA^2_{x,u}= \left(y=0\right),\quad  X_2=\AA^2_{y,u}= \left(x=0\right)
\]
Let us write $S=X_1\cap X_2=\AA^1_u$. The slab bundle for the slab $S$ is trivial. We endow $X$ with the log
structure over $k^\dagger$ given by the slab function
\[
f_\rho (u) =u
\]
on $S$. 
Denoting the origin by $\mathbf{0}$, 
we set $Z=\{\mathbf{0}\}\subset S$, write $U=X\setminus Z$, $S^\star=S\setminus Z$ and denote by $j\colon U \to X$
the natural inclusion. 
Since $f_\rho$ vanishes outside $U$, the log structure only exists
on $U$ and we denote its sheaf of monoids by $\foP_U$. 
To obtain a log structure on $X$ we use $\foM_X=j_\star \foP_U$.
and arrive at a log morphism $f\colon X^\dagger \to \Spec k^\dagger$
which is not formally log smooth at $Z$. 
This can be seen as follows. One first verifies that the stalk $\alpha_\mathbf{0}\colon \foM_{X,\mathbf{0}}\to \shO_{X,\mathbf{0}}$ agrees with the stalk of log structure obtained as pull-back from $\Spec k^\dagger$.
That is, the log morphism $f$ is \emph{strict} at $\mathbf{0}$.
Since the underlying morphism is not formally smooth, there is a test diagram with a square zero extension as in the definition of formal smoothness that has no diagonal map. 
Enrich this diagram with pull-back log structures from $\Spec k^\dagger$ and it also will not have a diagonal map, so formal log smoothness fails at $\mathbf{0}$.
On the other hand, formal log smoothness works over $U$ which we leave as an exercise.

As explained in \S~\ref{sec:two-components}, the log structure on
$U$ is determined by line bundles
\[
  \shL_{1\, U}, \quad \shL_{2\,U}, \quad \shL_U
\]
on $U$, and homomorphisms $\alpha_i\colon \shL_{i\,U}\to \shO$
satisfying the conditions in Proposition~\ref{pro:log_2components}, and $s\colon
\shL_U \overset{\cong}{\longrightarrow} \shL_{1\,U} \otimes \shL_{2\,U} $. Following
the recipe in the proof of Proposition~\ref{pro:2components_over_logpoint},
the line bundle $\shL_{2 U}$, for example, is obtained by assembling the line bundles
$\shO_{U_1}$ on $U_1$ and $\shO_{U_2}$ on $U_2$ via the isomorphism 
\[
\text{multiplication by $u$}: \shO_{U_1|S^\star} =\O_{S^\star}
\overset{\cong}{\longrightarrow} \O_{S^\star} =  \shO_{U_2|S^\star}.
\]
It follows from this that $\shL_2=j_\star \shL_{2\, U}$ is the
reflexive sheaf on $X$ associated to the $k[x,y,u]/_{(xy)}$-module
\[
M_1=\Gamma\left(\shL_{1\, U} ,U\right) =\langle e_1=(1,u),\, e_2=(0,y)\mid ye_1=ue_2, xe_2\rangle.
\]
Note for example that $(x,0)=xe_1$, $(1,u+y)=e_1+e_2$,
et cetera. Although $\shL_2$ is not a line bundle on $X$,
the restrictions $\shL_{2|X_1}$ and $\shL_{2|X_2}$ are line bundles.

The homomorphism $\alpha
\colon M \to k[x,y]/_{(xy)}$ is defined as
\[
\alpha (1,u)=x,\quad \alpha (0,y) = 0,
\]
that is
\[
  \alpha \colon \shL_{2|X_1}\overset{\cong}{\longrightarrow}\shO_{X_1}(-S),\quad
  \alpha \colon \shL_{2|X_2} =0 
\]

Naturally $\shL_1=\shL_2^\vee$ and $\shL =\shO_X=\shL_1[\otimes] \shL_2$ where $[\otimes]$ refers to the reflexive tensor product.
\footnote{It is an interesting (and surprisingly nontrivial) exercise to verify that the reflexive sheaf of relative log differentials
$\W^1_{X^\dagger/k^\dagger}:= j_\star\Omega^1_{U^\dagger/k^\dagger}$ is isomorphic to $\shL_1\oplus\shL_2$. The associated $k[x,y]/_{(x,y)}$-module is
\[
 \langle D_x, D_y,
 D_x^u, D_y^u \mid xD_y+yD_x=0, xD_x^u=uD_x,yD_y^u=uD_y \rangle
\] 
(Where, morally, $D_x=dx$, $D_y=dy$, $D_x^u = u \frac{dx}{x}$, $D_y^u=u\frac{dy}{y}$. Note for example
that $D_x^u+D_y^u=du$.)}

Take $Y_1=X_1$, let $Y_2\to X_2$ be the blow up of $X_2$ at the
origin, assemble $Y_1$ and $Y_2$ to obtain $Y$ and denote by $f\colon
Y\to X$ the proper birational map; the exceptional set is $E\cong
\PP^1\subset Y$ and $f_{|Y\setminus E}\colon Y\setminus E \to X\setminus Z$ is an isomorphism. Then
\[
\shLS_Y=f^\star \left(\shLS_X(-Z)\right).
\]
There is a unique log structure on $Y$ over $k^\dagger$ that agrees on
\[
  Y\setminus E \overset{f}{=} X\setminus Z
\]
with the given one and is log smooth over $k^\dagger$ because the
section $f_\rho=u\in \Gamma(S^\star,\shLS_Y)$ extends as a nowhere
vanishing section of $\shLS_Y$ on $S$.

\subsection{A reducible quartic del Pezzo surface}
\label{sec:reduc-quart-del}

\subsubsection{Description of the surface and log structure} 
\label{delPezzo}
We start from $X$ the toroidal crossing surface given as the union of three toric
surfaces as in the moment polygonal complex pictured in
Figure~\ref{fig:surface}.
\begin{figure}[h]
  \centering
  \begin{tikzpicture}
      \foreach \x in {-1, ..., 1}
      \foreach \y in {-1, ..., 1}
      {
        \node[thickstyle] (\x - \y) at (\x, \y){};
      }
      \draw[thick] (-1, -1) -- (1, -1) -- (0, 1) -- cycle;
      \draw[thick] (-1,-1) -- (0,0);
      \draw[thick] (0,0) -- (0,1);
      \draw[thick] (0,0) -- (1,-1);
      \node [anchor=east] at (0.5, 0.2) {$u$};
      \node [anchor=east] at (-0.75, -1.25) {$x$};
      \node [anchor=east] at (0.25, -1.25) {$w$};
      \node [anchor=east] at (1.25, -1.25) {$y$};
      \node [anchor=east] at (0.25, 1.25) {$z$};
  \end{tikzpicture}
  \caption{The moment polyhedral complex of the surface $X$}
  \label{fig:surface}
\end{figure}
From the picture we see that $X$ is naturally embedded in $\PP^4$ and given by equations:
\begin{equation*}
  X= \left\{
      \begin{aligned}
        xy-w^2&=0 \\
        zw & = 0
       \end{aligned}
      \right.
\end{equation*}
    in homogeneous coordinates $x:y:z:u:w$. The three components of
    $X$ are:
    \[
      X_1=\left(z=xy-w^2 = 0\right); \quad X_2=\left(w=x=0\right); \quad \text{and $X_3=\left(w=y=0\right)$} 
   \]
   The obvious --- and, essentially, unique --- polarization (see page \pageref{polar}) has kink $\kappa=1$ along the $x$ and
   $y$ axes and $\kappa=2$ along the $z$-axis.\footnote{In fact, this is the
   simplest real life example where we see a necessity for allowing a kink $>1$.} The datum of
   a log structure on $X$ over $k^\dagger$ consists of slab sections:
       \begin{align*}
       f_{\rho_x}& = a_0+a_1x+a_2x^2 \\
       f_{\rho_y}&= b_0+b_1y+b_2y^2 \\
       f_{\rho_z}&=c_0 + c_1z+c_2z^2+c_3z^3+c_4z^4
     \end{align*}
     (in the affine patch $u=1$) and the joint condition
     at the origin states that 
     \begin{equation}
     a_0=b_0,\quad \hbox{and}\quad c_0=a_0^2.
     \label{joint-cond-del-Pezzo}
     \end{equation} From now on we fix general slab
     sections and denote by $X^\dagger/k^\dagger$ the surface $X$
     equipped with the log structure and structure morphism $X^\dagger \to \Spec
     k^\dagger$ given by these functions.

     The morphism is not log smooth: it is singular along the union of points
     $Z\subset X$ where the slab secctions vanish. For generic coefficients, we find $8$ such points. 
     We will next see how these log structures arise naturally when we
     deform the surface $X$.

     \subsubsection{Log deformations of $X^\dagger/k^\dagger$}
     Consider the family of surfaces over $\AA^1$ with coordinate
     $t$ given by the equations
  \begin{equation}
    \label{eq:defo_dP4}
    \foX =
    \left\{
    \begin{aligned}
    xy-w^2& =ts_1wu+t^2(c_2u^2+c_3uz+c_4z^2)\\
    zw& = t(s_0u^2+a_1xu+a_2x^2+b_1yu+b_2y^2)
  \end{aligned}
  \right.
  \end{equation}
  in $\PP^4 \times \AA^{1}$.
  Note the following:
  \begin{enumerate}[(1)]
    \item Together with the projection $\pi \colon \foX \to \AA^1$
    the equations represent a flat deformation --- in fact, a smoothing ---
    of the surface $X=\foX_0$. 
   The general fibre is a del Pezzo surface of
    degree~$4$: a smooth intersection of two quadrics in $\PP^4$.
   \item Denote by $i \colon X \hookrightarrow \foX$ the natural
     inclusion. The total space $\foX$ is endowed with a natural
     (singular!)~divisorial log structure $\foM_{\foX, X}$; let us denote by
     $\foX^\dagger$ the corresponding log scheme. 
     Similarly, let $(\AA^1)^{\dagger}$ be $\AA^1$ together with the divisorial log structure from $t=0$.
     There is an obvious
     morphism of log schemes $\pi^\dagger \colon \foX^\dagger \to
     (\AA^1)^{\dagger}$ which when pulled back to $X$ gives a log
     structure $X^\dagger/k^\dagger$. The conflict of notation
     will be resolved momentarily when we see that this log structure
     is the same one that
     we defined above in \S \ref{delPezzo}.  The slab sections for this log
     structure on $X$ can be easily read from the equations; they are:
     \begin{align*}
       f_{\rho_x}& = s_0+a_1x+a_2x^2 \\
       f_{\rho_y}&= s_0+b_1y+b_2y^2 \\
       f_{\rho_z}&= s_0^2+s_0s_1z+c_2z^2+c_3z^3+c_4z^4
     \end{align*}
     For example we compute $f_{\rho_z}$ in the affine open set $u=1$ by localizing at $z=0$, using the second equation to
     solve for $w$,
     \[
w =\frac{ts_0}{z}\mod t(x,y)
\]
and plugging into the first equation which gives
\[
(xz)(yz)=t^2\bigl(s_0^2+s_0s_1z+c_2z^2+c_3z^3+c_4z^4\bigr) \mod t^2(x,y).
\]
   From this expression, we also read off from the exponent of $t$ that gives the kink $k=2$.\footnote{Staring at
   the formula for the slab section $f_{\rho_z}$ together with
   Equation~\ref{eq:defo_dP4} we see the slightly curious fact
   that ``half'' of $f_{\rho_z}$ contributes to the first-order
   deformation of $X$, and half to the second-order deformation!}
 \end{enumerate}

 \subsubsection{Log crepant log resolutions}

The following facts are elementary and easy to verify. We construct a
surface $Y$ and proper birational morphism $f\colon Y \to X$ as
follows. First let $Y_2\to X_2$ be the blow up of the two points
$(f_{\rho_y}=0)$ on the $y$-axis. Next let $Y_3\to X_3$ be the blow up of:
\begin{itemize}
\item the two points $(f_{\rho_x}=0)$ on the $x$-axis, and
\item the four points $(f_{\rho_z}=0)$ on the $z$-axis.
\end{itemize}
Denote by $Y$ be the surface obtained by re-assembling the three components
$Y_1=X_1$, $Y_2$, $Y_3$ in the obvious way, sketched in Figure~\ref{fig:resolved-surface}; and by $f\colon Y \to X$
the obvious proper birational morphism.
\begin{figure}[h]
  \centering
  \begin{tikzpicture}[scale=1.2]
      \foreach \x in {-1, ..., 1}
      \foreach \y in {-1, ..., 1}
      {
        \node[thickstyle] (\x - \y) at (\x, \y){};
      }
      \draw[thick] (-1, -1) -- (1, -1) -- (0, 1) -- cycle;
      \draw[thick] (-1,-1) -- (0,0);
      \draw[thick] (0,0) -- (0,1);
      \draw[thick] (0,0) -- (1,-1);
      \node [anchor=east] at (0.42, 0.2) {$u$};
      \node [anchor=east] at (-0.75, -1.25) {$x$};
      \node [anchor=east] at (0.25, -1.25) {$w$};
      \node [anchor=east] at (1.25, -1.25) {$y$};
      \node [anchor=east] at (0.25, 1.25) {$z$};
      \foreach \yi in {0.2,0.35,0.5,0.65}{
        \draw[thick] (0,\yi) arc [start angle=90, delta angle=70, radius=0.15];
      }
      \foreach \t in {0.2,0.5}{
        \draw[thick] (-\t,-\t) arc [start angle=45, delta angle=70, radius=0.15];
      }
      \foreach \t in {0.25,0.55}{
        \draw[thick] (\t,-\t) arc [start angle=-45, delta angle=70, radius=0.15];
      }
  \end{tikzpicture}
  \caption{A picture of the resolved surface $Y$ with exceptional
    curves indicated as little arcs in the triangles that correspond
    to the components of $Y$ that contain them.}
  \label{fig:resolved-surface}
\end{figure}

Our construction of $\shLS$ gives a way to compare $\shLS_Y$ with
$\shLS_X$. Denote by $S=\Sing X$ and $T=\Sing Y$ the singular sets
with the reduced scheme structure; and note that $f_{|T}\colon T\to S$
is an isomorphism. We use $f_{|T}$ to view $Z$ as a subset of $T$. The
resolution was constructed precisely so that the intersection of the
union of all exceptional curves with $T$ equals $Z$. Moreover, we
precisely understand how the slab bundles change under this blow-up.
A concise way to write this is as follows.  The sheaf $\shLS_Y$ is
supported on $T$ and $\shLS_X$ on $S$ and we have
\begin{equation}
\shLS_Y = f^\star \shLS_X  (-Z).
\label{blow-up-LS}
\end{equation}
The slab sections $f_{\rho_x}$, $f_{\rho_y}$ and $f_{\rho_z}$ give a section $s$ of $\shLS_X$ in the affine neighbourhood $u=1$ of the unique joint and this section extends uniquely to all of $X$ without receiving extra zeros outside the affine chart. The vanishing set of $s$ is $Z$. Therefore the section $s$
is the image of a section $\widetilde{s}$ under the natural inclusion $\shLS_X(-Z)\subset \shLS_X$ and moreover $\widetilde{s}$ is nowhere vanishing.
Hence, $\widetilde{s}$ gives a compatible log structure on $Y$ that is smooth over $k^\dagger$, in notation $Y^\dagger/k^\dagger$.
Furthermore, since $\widetilde{s}$ and $s$ agree over the set
\[
  Y\setminus f^{-1} (Z)\overset{f}{=} X\setminus Z.
\]
and $\widetilde{s}$ maps to $s$ under the inclusion
$\shLS_X(-Z)\subset \shLS_X$, we also obtain that the log structure
given by $\widetilde{s}$ on $Y\setminus X$ is the one given
by $s$ on $X\setminus Z$. The morphism of schemes $f\colon Y \to X$
is \emph{log crepant} in the sense that
\[
  \Omega^2_{Y^\dagger/k^\dagger}
\]
is $f$-trivial. We say that $Y^\dagger/k^\dagger$ together with $f\colon Y \to
X$ is a log crepant log resolution of $X^\dagger$ over $k^\dagger$.

\subsection{The \mbox{$3$-fold} transverse $A_1$-singularity}
\label{sec:mbox3-fold-transv}
\ \vspace{-.2cm}
\subsubsection{Description of the space and its log structure}
We take $X$ to be the affine cone over the projective toroidal crossing surface of \S~\ref{delPezzo}: the toroidal crossing \mbox{$3$-fold} union of three
toric affine pieces given in $\AA^5$ by the equations
\begin{equation*}
 X= \left\{
      \begin{aligned}
        xy-w^2&=0 \\
        zw & = 0
       \end{aligned}
      \right.
\end{equation*}
    in affine coordinates $x,y,z,u,w$. Note that $u$ does not appear
    in the equations: $X$ is the product of a toroidal crossing
    surface with $\AA^1_u$.

    The three components of
    $X$ are
    \[
      X_1=\left(z=xy-w^2 = 0\right), \quad X_2=\left(w=x=0\right) \quad \text{and}\quad X_3=\left(w=y=0\right). 
   \]
   The obvious --- and, essentially, unique --- polarization has kink
   $k=1$ along the coordinate surfaces $\AA^2_{x,u}$ and $\AA^2_{y,u}$, and $k=2$
   along $\AA^2_{z,u}$.

   We choose the very special log structure on $X$ over $k^\dagger$
   given by the slab sections:
\begin{equation}
  \label{log-str-3-fold} 
  \begin{aligned}
    f_{\rho_x} &= u \\
    f_{\rho_y} &= u \\
    f_{\rho_z} &= u^2 - z^2
  \end{aligned}
\end{equation}
The joint condition along the $u$-axis is satisfied. We
denote by $X^\dagger/k^\dagger$ the \mbox{$3$-fold} $X$ equipped with
the log structure and structure morphism $X^\dagger \to \Spec
k^\dagger$ given by these slab sections.

The singular locus of $X$ has three irreducible components as follows:\\[-.5cm]

\noindent
\begin{minipage}[c]{0.65\textwidth}
\begin{align*}
  S_1=&\AA^2_{z,u}=X_2\cap X_3,\\ 
   S_2=&\AA^2_{x,u}=X_1\cap X_3,\\ 
   S_3=&\AA^2_{y,u}=X_1\cap X_2
\end{align*}

The structure morphism $\pi \colon X^\dagger \to \Spec k^\dagger$ is
singular along the locus $Z\subset X$ where the slab sections
vanish. This singular locus has three irreducible
components as follows:
\[
  Z_1 =(u^2-z^2=0)\subset S_1, \quad \]
\[  Z_2 =(u=0) \subset S_2, \quad \]
\[  Z_3 =(u=0) \subset S_3, \quad
 \]
\end{minipage}
\hfill
\begin{minipage}[c]{0.3\textwidth}
  \captionsetup{width=\textwidth}
\qquad\ \includegraphics[width=.65\textwidth]{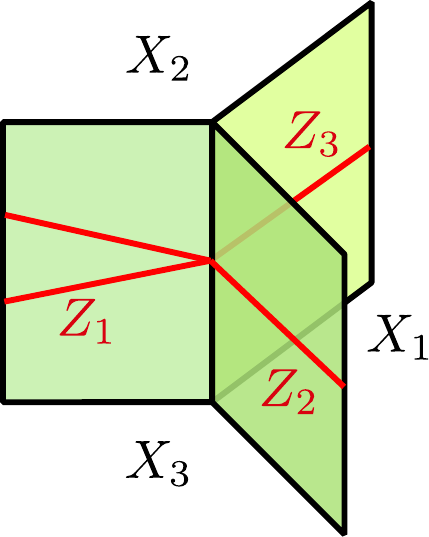}
\captionof{figure}{The log singular locus of $X^\dagger$.}
\label{fig-lattice2}
\end{minipage}

\smallskip

\subsubsection{Log deformations of $X^\dagger/k^\dagger$.}
Consider the family of \mbox{$3$-folds} over $\AA^1$ with
coordinate $t$ given by the equations
  \begin{equation}
    \label{eq:defo_cA1}
    \foX =
    \left\{
    \begin{aligned}
    xy-w^2& =-t^2\\
    zw& = tu
  \end{aligned}
  \right.
\end{equation}
in $\AA^5\times \AA^1$, together with the projection $\pi \colon \foX
\to \AA^1$. The equations describe a smoothing of $X$; and the
restriction to $X$ of the divisorial log structure $\foM_{\foX, X}$ is
the one given in \eqref{log-str-3-fold}.

 \subsubsection{Log crepant log resolutions}
\label{sec-blowup-threefold}
The following facts are elementary and easy to verify. We construct a
\mbox{$3$-fold} $Y$ and proper birational morphism $f\colon Y \to X$ as
follows. First let $Y_2\to X_2$ be the blow up of the singular curve
$Z_3\subset X_2$.\\ 

\noindent
\begin{minipage}[c]{0.65\textwidth}
Next let $Y_3\to
X_3$ be as follows:
\begin{itemize}
\item first, let $f^\prime \colon Y_3^\prime \to X_3$ be the blow up
  of the curve $Z_2\subset X_3$, and
\item second, let $Y_3 \to Y_3^\prime$ be the blow up of the strict
  transform $Z_1^\prime\subset Y_3^\prime$ of the curve $Z_1\subset
  X_3$ given by $(u^2-z^2=0)\subset \AA^2_{z, u}$. Note that
  $Z_1^\prime$ is nonsingular --- and consists of two components that
  we call $Z_+$, $Z_-$ corresponding to the factors $u+z$, $u-z$.
\end{itemize}
Denote by $Y$ be the \mbox{$3$-fold} obtained by re-assembling the three components
$Y_1=X_1$, $Y_2$, $Y_3$ in the obvious way; and by $f\colon Y \to X$
the obvious proper birational morphism.\\
\end{minipage}
\hfill
\begin{minipage}[c]{0.3\textwidth}
  \captionsetup{width=\textwidth}
\qquad\ \includegraphics[width=.85\textwidth]{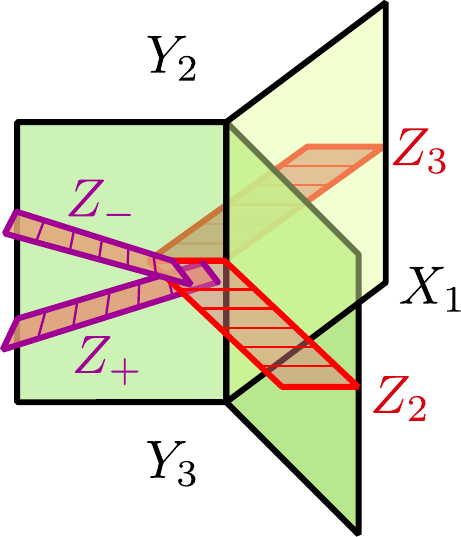}
\captionof{figure}{Schematic of the resolution.}
\label{fig-lattice3}
\end{minipage}

Our construction of $\shLS$ gives a way to compare $\shLS_Y$ with $f^\star
(\shLS_X)$, which we next explain. Denote by $T=\Sing Y$ the singular set with
the reduced scheme structure; $T$ consists of three irreducible
components
\[
 T_1=Y_2\cap Y_3, \quad T_2=Y_3\cap
 Y_1,\quad T_3= Y_1\cap Y_3
\]
We have that $f_{|T_1} \colon T_1\to S_1$ is the blow up of the origin,
$f_{|T_2}\colon T_2\to S_2$ is an isomorphism, and so is
$f_{|T_3}\colon T_3\to S_3$.

Next we identify the slab line bundles $\shL_{Y,1}$, $\shL_{Y,2}$, $\shL_{Y,3}$ on
$T_1$, $T_2$, $T_3$.

We need a notation for the exceptional divisors of the morphisms
$Y_i\to X_i$. The first of these morphisms, $f_{|Y_1}\colon Y_1\to X_1$, is an
isomorphism, so there is nothing to be done here.  Denote by
$E_2\subset Y_2$ the exceptional divisor of
$f_{|Y_2}\colon Y_2\to X_2$, so $E_2$ is a $\PP^1$-bundle over $Z_3$. 
The exceptional divisor of
$f_{|Y_3}\colon Y_3\to X_3$ consists of three components: the strict
transform $E_3$ of the exceptional divisor of $f^\prime \colon
Y_3^\prime \to X_3$, and the divisors $F_+$, $F_-$ that dominate the
curves $u+z=0$, $u-z=0$ in $S_1$.

We denote by $\Xi = E_{2}\cap T_1 = E_{3}\cap T_1\subset T_1$ the exceptional divisor of
$T_1\to S_1$, and we write $Z_+=F_{+}\cap T_1$, $Z_-=F_{-}\cap T_1$. Note that
\[
(f_{|T_1})^\star (Z_1) = F_++F_- + 2\Xi 
\]

With this notation in hand we are ready to compute the slab bundles.
The Stanley Reisner ring along the joint $u$ is given by the fan at
the monomial $u$ in Figure~\ref{fig:surface}. We see that the
monomials $xz$ and $yz$ correspond to opposite lattice vectors, so we
can use them to compute the slab bundle on $T_1$.  We also see from
the figure that the monomial $xz$ defines the boundary divisor
$\partial Y_3$ on $Y_3$ and the monomial $yz$ defines the boundary
divisor $\partial Y_2$ on $Y_2$, so we arrive at
\begin{multline*}
  \label{eq:A1slabs}
  \shL_{Y,1}=\shO_{Y_2}(\partial Y_2)_{|T_1}\otimes \shO_{Y_3}(\partial Y_3)_{|T_1}=(N_{Y_2} T_1) \otimes (N_{Y_3} T_1) \otimes \shO_{T_1}
  (2Y_1) = \\
  =(f^\star N_{X_2} S_1) \otimes (f^\star N_{X_3} S_1) (-F_+-F_-)
  \otimes \shO_{S_1} (2X_1)(-2\Xi) =\\
  =(f^\star \shL_{X_1})(-F_+-F_--2\Xi) =
  f^\star \bigl(\shL_{X,1}(-Z_1)\bigr)
\end{multline*}
The formulas for $\shL_{Y,2}$ and $\shL_{Y,3}$ are much easier to
understand:
\[
\shL_{Y,2}=\shL_{X,2}(-Z_2), \quad \shL_{Y,3} = \shL_{X,3} (-Z_3)
\]

We summarize these calculations with the formula
\begin{equation}
  \label{eq:LS_blowup_formula}
\shLS_Y=f^\star \bigl( \shLS_X(-Z)\bigr).
\end{equation}
Pulling back the section of $\shLS_X$ defined by the slab sections \eqref{log-str-3-fold} results in a nowhere vanishing section of $\shLS_Y$.
We see from this that
\begin{enumerate}[(1)]
\item There is a unique log structure on $Y$ over $k^\dagger$ that on
  \[
Y\setminus f^{-1}(Z) \overset{f}{=} X\setminus Z
\]
 is the given one on $X\setminus Z$;
\item The resulting log scheme $Y^\dagger/k^\dagger$ is log smooth
  over $\Spec k^\dagger$.
\end{enumerate}
One can check that the resulting morphism $Y \to X$ is log crepant, so it gives a log crepant log resolution of $X^\dagger$ over $k^\dagger$.

This example shows how the sheaf $\shLS_X$ can be used in constructing
log smooth resolutions.

\bibliography{bib_logstr}

\end{document}